\documentclass[reqno]{amsart}

\usepackage[normalem]{ulem}

\usepackage[english]{babel}
 
\usepackage{hyperref}
\usepackage{bbm}
\usepackage{amsfonts,amsmath,amsxtra,amsthm,amssymb,latexsym,enumerate}
\usepackage{ifpdf}
\usepackage{tikz}
\usetikzlibrary{decorations.pathreplacing}

\usepackage{color}

\makeatletter
\@namedef{subjclassname@2020}{%
  \textup{2020} Mathematics Subject Classification}
\@addtoreset{footnote}{section} 
\renewcommand{\thefootnote}{\ifcase\value{footnote}\or$\dagger$\or$\ddagger$\or$\dagger\dagger$\or(****)\or(\#)\or(\#\#)\or(\#\#\#)\or(\#\#\#\#)\or($\infty$)\fi}

\makeatother

\usepackage{perpage} %the perpage package
\MakePerPage{footnote} %the perpage package command\makeatother

\newcommand*{\mailto}[1]{\href{mailto:#1}{\nolinkurl{#1}}}
\newcommand{\arxiv}[1]{\href{http://arxiv.org/abs/#1}{arXiv:#1}}
\newcommand{\doi}[1]{\href{http://dx.doi.org/#1}{doi: #1}}
\newcommand{\msc}[1]{\href{http://www.ams.org/msc/msc2020.html?t=&s=#1}{#1}}

\newcommand{\ack}{\section*{Acknowledgments}}

\newif\iflong
\newtheorem{theorem}{Theorem}[section]
\newtheorem{corollary}[theorem]{Corollary}
\newtheorem{proposition}[theorem]{Proposition}
\newtheorem{lemma}[theorem]{Lemma}
\newtheorem{hypothesis}{Hypothesis}[section]
\theoremstyle{definition}
\newtheorem{definition}[theorem]{Definition}
\newtheorem{example}[theorem]{Example}
\newtheorem{remark}[theorem]{Remark}

\newcommand{\nn}{\nonumber}
\newcommand{\be}{\begin{equation}}
\newcommand{\ee}{\end{equation}}
\newcommand{\ti}{\tilde}
\newcommand{\wt}{\widetilde}
\newcommand{\wh}{\widehat}
\newcommand{\id}{{\mathbbm 1}}

\numberwithin{equation}{section}

\DeclareMathOperator{\supp}{supp} 
\DeclareMathOperator{\diam}{diam} 
\DeclareMathOperator{\Span}{span} 

\DeclareMathOperator{\re}{Re} 
 
\DeclareMathOperator{\dom}{dom}
\DeclareMathOperator{\mul}{mul}
\DeclareMathOperator{\ran}{ran}

  \def\mG{\mathsf{G}} 
  \def\mN{\mathsf{N}}

\newcommand\R{{\mathbb{R}}}
\newcommand\C{{\mathbb{C}}}

\newcommand\Z{{\mathbb{Z}}}

\newcommand\HH{{\mathbb{H}}}

\newcommand{\gt}{\mathfrak{t}}
\newcommand{\gC}{\mathfrak{C}}

\newcommand\cA{{\mathcal{A}}}
\newcommand\cI{{\mathcal{I}}}
\newcommand\cH{{\mathcal{H}}}
\newcommand\cT{{\mathcal{T}}}

\newcommand\cK{{\mathcal{K}}}
\newcommand\cN{{\mathcal{N}}}
\newcommand\cU{{\mathcal{U}}}
\newcommand\cG{{\mathcal{G}}}
\newcommand\cE{{\mathcal{E}}}
\newcommand\cP{{\mathcal{P}}}
\newcommand\cV{{\mathcal{V}}}

\newcommand\cR{{\mathcal{R}}}

\newcommand\OO{{\mathcal{O}}}
\newcommand{\varrhoo}{\varrho}
\newcommand\mR{{\mathfrak{R}}}

\newcommand\bH{{\mathbf{H}}}

\newcommand\rH{{\rm{H}}}
\newcommand\E{{\rm{e}}}

\newcommand\vol{{\rm{vol}}}

\newcommand\Nr{{\rm{n}}}
\newcommand\rD{{\rm{d}}}
\newcommand\I{{\rm{i}}}

\newcommand\f{{\bf{f}}}

\def\wt#1{{{\widetilde #1} }}

\begin{document}

\title[Extensions of Quantum Graphs]{Self-adjoint and Markovian Extensions\\ of Infinite Quantum Graphs}

\author[A. Kostenko]{Aleksey Kostenko}
\address{Faculty of Mathematics and Physics\\ University of Ljubljana\\ Jadranska ul.\ 19\\ 1000 Ljubljana\\ Slovenia\\ and 
%Faculty of Mathematics\\ University of Vienna\\ 
%Oskar-Morgenstern-Platz 1\\ 1090 Vienna\\ Austria}
Institute for Analysis and Scientific Computing\\ Vienna University of Technology\\ Wiedner Hauptstra\ss e 8-10/101\\1040 Vienna\\ Austria}
\curraddr{Faculty of Mathematics and Physics\\ University of Ljubljana\\ Jadranska ul.\ 19\\ 1000 Ljubljana\\ Slovenia\\ and 
Faculty of Mathematics\\ University of Vienna\\ 
Oskar-Morgenstern-Platz 1\\ 1090 Vienna\\ Austria}
\email{\mailto{Aleksey.Kostenko@fmf.uni-lj.si}}
%\urladdr{\url{http://www.mat.univie.ac.at/~kostenko/}}

\author[D. Mugnolo]{Delio Mugnolo}
\address{Lehrgebiet Analysis\\ Fakult\"at Mathematik und Informatik\\ FernUniversit\"at in Hagen\\ Hagen\\ Germany}
\email{\mailto{delio.mugnolo@fernuni-hagen.de}}

\author[N. Nicolussi]{Noema Nicolussi}
\address{Faculty of Mathematics\\ University of Vienna\\
Oskar-Morgenstern-Platz 1\\ 1090 Vienna\\ Austria}
\email{\mailto{noema.nicolussi@univie.ac.at}}

\thanks{{\it Research supported by the Austrian Science Fund (FWF) 
under Grants No.\ P 28807 (A.K. and N.N.) and W 1245 (N.N.), 
 by the German Research Foundation (DFG) under Grant No.\ 397230547 (D.M.), and by the Slovenian Research Agency (ARRS) under Grant No.\ J1-1690 (A.K.). The authors would like to acknowledge that this article is based upon work from COST Action CA18232 MAT-DYN-NET, supported by COST (European Cooperation in Science and Technology).}}
 \thanks{{ \arxiv{1911.04735}}}
 \thanks{J.\ London Math.\ Soc., to appear; \doi{10.1112/jlms.12539}}


\begin{abstract}
We investigate the relationship between one of the classical notions of boundaries for infinite graphs, \emph{graph ends}, and self-adjoint extensions of the minimal Kirchhoff Laplacian on a metric graph. We introduce the notion of \emph{finite volume} for ends of a metric graph and show that  finite volume graph ends is the proper notion of a boundary for Markovian extensions of the Kirchhoff Laplacian. In contrast to manifolds and weighted graphs, this provides a transparent geometric characterization of the uniqueness of Markovian extensions, as well as of the self-adjointness of the Gaffney Laplacian --- the underlying metric graph does not have finite volume ends. 
  If however finitely many finite volume ends occur (as is the case of edge graphs of normal, locally finite tessellations or Cayley graphs of amenable finitely generated groups), we provide a complete description of Markovian extensions upon introducing a suitable notion of traces of functions and normal derivatives on the set of graph ends.
\end{abstract}

\maketitle

{\scriptsize{\tableofcontents}}

%%%%%%%%%%%%%%%%%%%%%%%%%%%%%%%%%%%%%%%%%%%%%%%%%%%%%%%%%%%%
%%%%%%%%%%%%%%%%%%%%%%%%%%%%%%%%%%%%%%%%%%%%%%%%%%%%%%%%%%%%
\section{Introduction}
%%%%%%%%%%%%%%%%%%%%%%%%%%%%%%%%%%%%%%%%%%%%%%%%%%%%%%%%%%%%
%%%%%%%%%%%%%%%%%%%%%%%%%%%%%%%%%%%%%%%%%%%%%%%%%%%%%%%%%%%%

This paper is concerned with developing extension theory for infinite \emph{quantum graphs}. Quantum graphs are Schr\"{o}dinger operators on \emph{metric graphs}, that is combinatorial graphs where edges are considered as intervals with certain lengths. Motivated by a vast amount of applications in chemistry and physics, they have become a popular subject in the last decades (we refer to~\cite{bcfk06, bk13, ekkst08, post} for an overview and further references). From the perspective of Dirichlet forms, quantum graphs play an important role as an intermediate setting between Laplacians on Riemannian manifolds and difference Laplacians on weighted graphs. On the one hand, being locally one-dimensional, quantum graphs allow to simplify considerations of complicated geometries. On the other hand, there is a close relationship between random walks on graphs and Brownian motion on metric graphs, however, in contrast to the discrete case, the corresponding quadratic form in the metric case is a strongly local Dirichlet form and in this situation more tools are available (see~\cite{baba03, fo14, lu16, lu19} for various manifestations of this point of view). 
Let us also mention that metric graphs can be seen as non-Archimedian analogs of Riemann surfaces, which finds numerous applications in algebraic geometry (see~\cite{amca13, bano07, baru10, shwu19} for further references).

The most studied quantum graph operator is the \emph{Kirchhoff Laplacian}, which provides the analog of the Laplace--Beltrami operator in the setting of metric graphs. Its spectral properties are crucial in connection with the heat equation and the Schr\"odinger equation and any further analysis usually relies on the \emph{self-adjointness of the Laplacian}. Whereas on finite metric graphs the Kirchhoff Laplacian is always self-adjoint, the question is more subtle for \emph{graphs with infinitely many edges}. 

 For instance, a uniform lower bound for the edge lengths guarantees self-adjoint\-ness (see~\cite{bk13, post}), but this commonly used condition is independent of the combinatorial graph structure and clearly excludes a number of interesting cases (the so-called \emph{fractal metric graphs}). Moreover, most of the results on strongly local Dirichlet forms require completeness of a given metric space w.r.t. the ``intrinsic" metric (cf., e.g.,~\cite{stu94}), which coincides with the natural path (geodesic) metric in the case of metric graphs. Geodesic completeness (w.r.t. the natural path metric) guarantees self-adjointness of the (minimal) Kirchhoff Laplacian, however, this result is far from being optimal (see~\cite[\S4]{ekmn} and also Section~\ref{ss:II.04} below). 
The search for self-adjointness criteria for infinite quantum graphs is an open and -- in our opinion -- rather difficult problem. 

If the (minimal) Kirchhoff Laplacian is not self-adjoint, the natural next step is to ask for a description of its self-adjoint extensions, which corresponds to possible descriptions of the system in quantum mechanics or, if we speak about Markovian extensions, possible descriptions of Brownian motions. 
 Naturally, this question is tightly related to finding appropriate boundary notions for infinite graphs. Our goal in this paper is to investigate the connection between extension theory and one particular notion, namely \emph{graph ends}, a concept which goes back to the work of Freudenthal~\cite{f44} and Halin~\cite{hal} and provides a rather refined way of compactifying graphs. 
 However, the definition of graph ends is purely combinatorial and naturally must be modified to capture the additional metric structure of our setting. Based on the correspondence between graph ends and topological ends of metric graphs, we introduce the concept of \emph{ends of finite volume}. 
 First of all, it turns out that finite volume ends play a crucial role in describing the Sobolev spaces $H^1$ and $H^1_0$ on metric graphs. More specifically,  we show that the presence of finite volume ends is the only reason for the strict inclusion $H^1_0\subsetneq H^1$ to hold. This in particular provides a surprisingly transparent geometric characterization of the uniqueness of Markovian extensions of the minimal Kirchhoff Laplacian as well as the self-adjointness of the so-called \emph{Gaffney Laplacian} (we are not aware of its analogs either in the manifold setting or in the context of weighted graph Laplacians, cf.~\cite{ghklw, gm, hkmw, klss, ma99, ma05}). As yet another manifestation of the fact that finite volume graph ends represent the proper boundary for Markovian extensions of the Kirchhoff Laplacian, we provide a complete description of \emph{all finite energy extensions} (i.e., self-adjoint extensions with domains contained in $H^1$, and all Markovian extensions clearly satisfy this condition), however, under the additional assumption that there are only finitely many finite volume ends. Let us stress that this class of graphs includes a wide range of interesting models (Cayley graphs of a large class of finitely generated groups, tessellating graphs, rooted antitrees etc. have exactly one end and in this case there are no finite volume ends exactly when the total volume of the corresponding metric graph is infinite). Moreover, we emphasize that in all those cases the dimension of the space of finite energy extensions is equal to the number of finite volume ends, however, for deficiency indices, i.e., the dimension of the space of self-adjoint extensions, this only gives a lower bound (for example, for Cayley graphs the dimension of the space of finite energy extensions is independent of the choice of a generating set, although deficiency indices do depend on this choice in a rather nontrivial way). On the other hand, it may happen that these dimensions coincide. The latter holds only if the maximal domain is contained in $H^1$, that is, if every self-adjoint extension is a finite energy extension. This is further equivalent to the validity of a certain non-trivial Sobolev-type inequality (see \eqref{eq:IntroEst2} below). The appearance of this condition demonstrates the mixed dimensional behavior of infinite metric graphs since the analogous estimate holds true in the one-dimensional situation, but usually fails in the PDE setting.  

Let us now sketch the structure of the article and describe its content and our results in greater details.
 
In Section~\ref{sec:QG} we collect basic notions and facts about graphs and metric graphs (Section~\ref{ss:II.01}); graph ends (Section~\ref{ss:II.02}); the minimal and maximal Kirchhoff Laplacians (Section~\ref{ss:II.03}); deficiency indices and their connection with the spaces of $L^2$ harmonic and $\lambda$-harmonic functions (Section~\ref{ss:II.04}).
 
The core of the paper is Section~\ref{sec:Ends}, where we discuss the Sobolev spaces $H^1(\cG)$ and $H^1_0(\cG)$ and introduce the set of finite volume ends $\gC_0(\cG)$ (Definition~\ref{def:finvol}). We show that $\gC_0(\cG)$ is the proper boundary for $H^1$ functions, which can also be seen as an ideal boundary by applying $C^\ast$-algebra techniques (see Remark~\ref{rem:IdealBdry}). The central result of this section is Theorem~\ref{th:H10}, which shows that $H^1(\cG)= H^1_0(\cG)$ if and only if there are no finite volume ends. The latter also leads to a surprisingly transparent geometric characterization of the uniqueness of Markovian extensions of the Kirchhoff Laplacian (Corollary~\ref{cor:Markovian1}) as well as the self-adjointness of the Gaffney Laplacian $\bH_G$ (see Remark~\ref{rem:H=H=uniq}(ii) for details and the definition of $\bH_G$). 

Section~\ref{sec:DefInd} contains further applications of the above considerations. Namely, Theorem~\ref{th:indicesends} demonstrates that deficiency indices of the minimal Kirchhoff Laplacian can be estimated from below by the number of finite volume ends. This estimate is sharp (e.g., if there are infinitely many finite volume ends) and we also find necessary and sufficient conditions for the equality to hold. In particular, if there are only finitely many ends of finite volume, $\#\gC_0(\cG)<\infty$, the latter is equivalent to the validity of the following Sobolev-type inequality (see Remark~\ref{rem:sobolev})\be\label{eq:IntroEst2}
	\| f'\|_{L^2(\cG)}\le C (\| f \|_{L^2(\cG)}+ \| f'' \|_{L^2(\cG)})
\ee
for all $f$ in the maximal domain of the Kirchhoff Laplacian. Metric graphs are locally one-dimensional and the corresponding inequality is trivially satisfied in the one-dimensional case, however, globally infinite metric graphs are more complex and hence \eqref{eq:IntroEst2} rather resembles the multi-dimensional setting of PDEs (in particular, \eqref{eq:IntroEst2} does not hold true if $\cG$ has a \emph{non-free} finite volume end, see Proposition~\ref{prop:H1InfEnds}). 

In the next sections, we focus on a particular class of self-adjoint extensions whose domains are contained in $H^1$ (we call them \emph{finite energy extensions}). These extensions have good properties and their importance stems from the fact that they contain the class of Markovian extensions (they also arise as self-adjoint restrictions of the Gaffney Laplacian). In Section~\ref{sec:V} we show that (under some additional mild assumptions) their \emph{resolvents and heat semigroups} are \emph{integral operators} with  continuous, bounded kernels and they belong to the trace class if $\cG$ has finite total volume (Theorems~\ref{prop:resolvprop} and~\ref{prop:semigprop}).

In Section~\ref{sec:VI} we proceed further and show that finite volume ends is the proper boundary for this class of extensions. Namely, under the additional and rather restrictive assumption of \emph{finitely many ends with finite volume}, in Sections~\ref{ss:VI:i}--\ref{sec:VI.02}, we introduce a suitable notion of a \emph{normal derivative} at graph ends (as a by-product, this also gives an explicit description of the domain of the Neumann extension, see Corollary~\ref{prop:domNM}).  Section~\ref{sec:VI.03} contains a complete description of finite energy extensions and also of Markovian extensions (Theorem~\ref{th:ThetaCD}). Let us stress that the case of infinitely many ends is incomparably more complicated and will be the subject of future work. 

In general, the inequality in \eqref{eq:IntroEst2} is difficult to verify/contradict and even simple examples can exhibit rather complicated behavior (see Appendix~\ref{sec:rope}). The only reason for which \eqref{eq:IntroEst2} fails to hold is the presence of $L^2$ harmonic functions having infinite energy, that is, not belonging to $H^1$. Moreover, in order to compute deficiency indices of the Kirchhoff Laplacian one, roughly speaking, needs to find the dimension of the space of $L^2$ harmonic functions and description of self-adjoint extensions requires a thorough understanding of the behavior of $L^2$ harmonic functions at ``infinity". Dictated by a distinguished role of harmonic functions in analysis, there is an enormous amount of literature dedicated to various classes of harmonic functions (positive, bounded etc.), which is further related to different notions of boundaries (metric completion, Poisson and Martin boundaries, Royden and Kuramochi boundaries etc.) and search for a suitable notion in this context (namely, $L^2$ harmonic functions) is a highly nontrivial problem, which seems not to be very well studied either in the context of \emph{incomplete} manifolds (cf.~\cite{ma99, ma05}) or in the case of weighted graphs (see~\cite{hklw, hkmw}). We further illustrate this by considering the case of rooted \emph{antitrees}, a special class of infinite graphs with a particularly high degree of symmetry (see Section~\ref{sec:antitrees}).  Infinite rooted antitrees have exactly one graph end, which makes them a good toy model for our purposes. The above considerations show that the space of finite energy $L^2$ harmonic functions is nontrivial only if a given metric antitree has finite total volume and in this case the only such functions are constants. However, adjusting lengths in a suitable way for a concrete polynomially growing antitree (Figure~\ref{fig:antitree}) we can make the space of $L^2$ harmonic functions as large as we please (even infinite dimensional!).

%%%%%%%%%%%%%%%%%%%%%%%%%%%%%%%%%%%%%%%%%%%%%%%%%%%%%%%%%%%%
\subsection*{Notation}
%%%%%%%%%%%%%%%%%%%%%%%%%%%%%%%%%%%%%%%%%%%%%%%%%%%%%%%%%%%%
$\Z$, $\R$, $\C$ have their usual meaning; $\Z_{\ge a} := \Z\cap [a,\infty)$.\\
$z^\ast$ denotes the complex conjugate of $z\in\C$.\\  
For a given set $S$, $\#S$ denotes its cardinality if $S$ is finite; otherwise we set $\#S=\infty$.\\
If it is not explicitly stated otherwise, we shall denote by $(x_n)$ a sequence $(x_n)_{n=0}^\infty$.\\[1mm]
$C_b(X)$ is the space of bounded, continuous functions on a locally compact space $X$. \\
$C_0(X)$ is the space of continuous functions vanishing at infinity.\\% on a locally compact Hausdorff space $X$; \\
%i.e., each $f \in C_0(X)$ is continuous and $X_\varepsilon = \{x \in X \colon |f(x)| \ge \varepsilon \}$ is compact for all $\varepsilon > 0$. \\%there is a compact set $K \subseteq X$ such that $|f(x)|< \varepsilon$ for all $x \in X\setminus K$. \\
For a finite or countable set $X$, $C(X)$ is the set of complex-valued functions on $X$.\\[1mm]
$\cG_d = (\cV,\cE)$ is a discrete graph (satisfying Hypothesis~\ref{hyp:locfin}).\\
$\cG = (\cG_d,|\cdot|)$ is a  metric graph (see p.~\pageref{not:metricgraph}).\\
$\varrho$ is the natural (geodesic) path metric on $\cG$ (see p.~\pageref{not:natpathmet}).\\
$\varrho_m$ is the star metric on $\cV$ corresponding to the star weight $m$ (see \eqref{def:rho_m}).\\
$\Omega(\cG_d)$ denotes the graph ends of $\cG_d$ (see Definition~\ref{def:rays}).\\
$\gC(\cG)$ denotes the topological ends of a metric graph $\cG$ (see Definition~\ref{def:TopEnds}).\\
$\gC_0(\cG)$ stays for the finite volume topological ends of $\cG$ (see Definition~\ref{def:finvol}).\\
$\wh{\cG}$ is the end (Freudenthal) compactification of $\cG$ (see p.~\pageref{p:FreudenthalCompactification}). \\[1mm] 
 ${\bH}^0_0$ is the pre-minimal Kirchhoff Laplacian on $\cG$ (see \eqref{eq:H00}).\\
 ${\bH}_0$ is the minimal Kirchhoff Laplacian, the closure of ${\bH}^0_0$ in $L^2(\cG)$ (see \eqref{eq:H00}).\\
 $\Nr_\pm({\bH}_0)$ are the deficiency indices of $\bH_0$ (see \eqref{def:n_pm}).\\ 
${\bf H}_F$ and ${\bf H}_N$ are the Friedrichs and Neumann extensions of $\bH_0$ (see p.~\pageref{p:Friedrichsextension} and, respectively,  p.~\pageref{p:NeumannExt}).\\
$\bH$ is the maximal Kirchhoff Laplacian on $\cG$ (see \eqref{eq:H}).

%%%%%%%%%%%%%%%%%%%%%%%%%%%%%%%%%%%%%%%%%%%%%%%%%%%%%%%%%%%%
%%%%%%%%%%%%%%%%%%%%%%%%%%%%%%%%%%%%%%%%%%%%%%%%%%%%%%%%%%%%
\section{Quantum graphs} \label{sec:QG}
%%%%%%%%%%%%%%%%%%%%%%%%%%%%%%%%%%%%%%%%%%%%%%%%%%%%%%%%%%%%
%%%%%%%%%%%%%%%%%%%%%%%%%%%%%%%%%%%%%%%%%%%%%%%%%%%%%%%%%%%%

%%%%%%%%%%%%%%%%%%%%%%%%%%%%%%%%%%%%%%%%%%%%%%%%%%%%%%%%%%%%
\subsection{Combinatorial and metric graphs}\label{ss:II.01}
%%%%%%%%%%%%%%%%%%%%%%%%%%%%%%%%%%%%%%%%%%%%%%%%%%%%%%%%%%%%
In what follows, $\cG_d = (\cV, \cE)$ will be an unoriented graph with countably infinite sets of vertices $\cV$ and edges $\cE$. For two vertices $u$, $v\in \mathcal{V}$ we shall write $u\sim v$ if there is an edge $e_{u,v}\in \mathcal{E}$ connecting $u$ with $v$. 
For every $v\in \mathcal{V}$, we denote  the set of edges incident to the vertex $v$ by $\cE_v$ and 
\be\label{eq:combdeg}
\deg_{\cG}(v):= \#\{e|\, e\in\cE_v\} 
\ee
is called \emph{the degree} (\emph{valency} or \emph{combinatorial degree}) of a vertex $v\in\cV$. When there is no risk of confusion about which graph is involved, we shall simplify and write $\deg$ instead of $\deg_{\cG}$.  \emph{A path} $\cP$ of length $n\in\Z_{\ge 0}\cup\{\infty\}$ is a sequence of vertices $(v_0,v_1,\dots, v_n)$ such that $v_{k-1}\sim v_k$ for all $k\in \{1,\dots,n\}$. 

The following assumption is imposed throughout the paper.

\begin{hypothesis}\label{hyp:locfin}
$\cG_d$ is \emph{infinite}, \emph{locally finite} ($\deg(v) < \infty$ for every $v \in \cV$), \emph{connected} (for any two vertices $u,v\in\cV$  there is a path connecting $u$ and $v$), and \emph{simple} (there are no loops or multiple edges).
\end{hypothesis}

Next, let us assign each edge $e\in\cE$ a finite length $|e| \in (0,\infty)$. We can then naturally associate with $(\cG_d,|\cdot|) = (\cV,\cE,|\cdot|)$ a metric space $\cG$: first, we identify each edge $e \in \cE$ with a copy of the interval $\cI_e := [0, |e|]$. The topological space $\cG$ is then obtained by ``gluing together" the ends of edges corresponding to the same vertex $v$ (in the sense of a topological quotient, see, e.g.,~\cite[Chapter~3.2.2]{bbi}). 
The topology on $\cG$ is metrizable by the {\em natural path metric}\label{not:natpathmet} $\varrho$ --- the distance between two points $x,y \in\cG$ is defined as the arc length of the ``shortest path" connecting them (if $x$ or $y$ are not vertices, then we need to allow also paths which start or end in the middle of edges; the length of such paths is naturally defined by taking the corresponding portion of the interval). The metric space $\cG$ arising from the above construction is called a \emph{metric graph} (associated to $(\cG_d, |\cdot|) =(\cV, \cE, |\cdot|)$). \label{not:metricgraph}

Notice that, by definition, $(\cG, \varrho)$ is a {\em length space} (see~\cite[Chapter~2.1]{bbi} for definitions and further details). Moreover (see, e.g.,~\cite[Chapter~1.1]{hae}), a metric graph $\cG$ is a Hausdorff topological space with countable base and each $x \in \cG$ has a neighborhood isometric to a star-shaped set $\mathcal{S} (\deg(x),r_x)$ of degree $\deg(x)\in \Z_{\ge 1}$,
\begin{align} \label{eq:star}
\mathcal{S}  (\deg(x),r_x) := \big\{z= r\E^{2\pi \I k/\deg(x)}|\, r\in [0,r_x),\ k=1,\dots,\deg(x) \big\}\subset \C.
\end{align}
Notice that $\deg(x)$ in \eqref{eq:star} coincides with the combinatorial degree if $x$ belongs to the vertex set, and $\deg(x) =2$ for every non-vertex point $x$ of $\cG$.

Sometimes, we will consider $\cG_d$ as a rooted graph with a fixed root $o \in \cV$. In this case we denote by $S_n$, $n \in \Z_{\ge 0}$ the $n$-th combinatorial sphere with respect to the order induced by $o$ (notice that $S_0=\{o\}$). 

%%%%%%%%%%%%%%%%%%%%%%%%
\subsection{Graph ends}\label{ss:II.02}  
%%%%%%%%%%%%%%%%%%%%%%%%
 One possible definition of a boundary for an infinite graph is the notion of the so-called \emph{graph ends} (see~\cite{f44,hal} and~\cite[\S~21]{woe}). 
 
 \begin{definition}\label{def:rays}
 A sequence of distinct vertices $(v_n)_{n\in\Z_{\ge 0}}$ (resp., $(v_n)_{n\in\Z}$) such that $v_{n}\sim v_{n+1}$ for all $n\in\Z_{\ge 0}$  (resp., for all $n\in\Z$) is called a \emph{ray} (resp., \emph{double ray}). A subsequence of such a sequence is called a \emph{tail}.

Two rays $\cR_1,\cR_2$ are called \emph{equivalent} -- and we write $\cR_1 \sim \cR_2$ -- if there is a third ray containing infinitely many vertices of both $\cR_1$ and $\cR_2$.\footnote{Equivalently, $\cR_1 \sim \cR_2$ if and only if $\cR_1$ and $\cR_2$ cannot be separated by a finite vertex set, i.e., for every finite subset  $X \subset \cV$ the remaining tails of $\cR_1$ and $\cR_2$ in $\cV \setminus X$ belong to the same connected component of $\cV \setminus X$.}
 An equivalence class of rays is called a \emph{graph end of  $\cG_d$} and the set of graph ends will be denoted by $\Omega(\cG_d)$. Moreover, we will write $\cR \in \omega$ whenever $\cR$ is a ray belonging to the end $\omega \in \Omega(\cG_d)$. 
\end{definition}

An important feature of graph ends is their relation to  topological ends of a metric graph $\cG$.  

\begin{definition} \label{def:TopEnds}
Consider sequences 
$\cU = (U_n)_{n=0}^\infty$ of non-empty open connected subsets of $\cG$ with compact boundaries and such that $U_{n+1} \subseteq U_{n}$ for all $n\ge 0$ and $\bigcap_{n\ge 0} \overline{U_n} = \emptyset$. Two such sequences $\cU$ and $\cU'$ are called \emph{equivalent} if for all $n\ge0$ there exist $j$ and $k$ such that $U_n \supseteq U_j'$ and $U_n' \supseteq U_k$. An equivalence class $\gamma$
 of sequences is called a \emph{topological end} of $\cG$ and $\gC(\cG)$ denotes the set of topological ends of $\cG$. 
\end{definition}

For locally finite graphs, there is a bijection between topological ends of a metric graph $\gC(\cG)$ and graph ends $\Omega(\cG_d)$ of the underlying combinatorial graph $\cG_d$ (see~\cite[\S~21]{woe}, \cite[\S~8.6 and also~p.277--278]{die}; for the case of graphs which are not locally finite see~\cite{csw, dikue}).

\begin{theorem}\label{th:top-comb}
For every topological end $\gamma\in\gC(\cG)$ of a locally finite metric graph $\cG = (\cG_d, |\cdot|)$ there exists a unique graph end $\omega_{\gamma} \in \Omega(\cG_d)$ such that for every sequence $\cU$ representing $\gamma$, each $U_n$ contains a ray from $\omega_{\gamma}$. 
Moreover, the map $\gamma \mapsto \omega_{\gamma}$ is a bijection between $\gC(\cG)$ and $\Omega(\cG_d)$. 
\end{theorem}

Therefore, we may identify topological ends of a metric graph $\cG$ and graph ends of the underlying graph $\cG_d$. We will simply speak of the \emph{ends} of $\cG$. 
One obvious advantage of this identification is the fact that the definition of $\Omega(\cG_d)$ is purely combinatorial and does not depend on edge lengths. 

\begin{definition}
An end $\omega$ of a graph $\cG_d$ is called  \emph{free} if there is a finite set $X$ of vertices which separates $\omega$ from all other ends of the graph (i.e. the rays of all ends $\omega' \neq \omega$ end up in different connected components of $\cV \setminus X$ than the rays of $\omega$).
\end{definition}

\begin{remark}\label{rem:ends}
Let us mention several examples. 
\begin{itemize}
\item[(i)] $\Z$ has two ends both of which are free.
\item[(ii)] $\Z^N$ has one end for all $N\ge 2$.
\item[(iii)] A $k$-regular tree, $k\ge 3$, has uncountably many ends, none of which is free.
\item[(iv)] If $\cG_d$ is a Cayley graph of a finitely generated infinite group $\mG$, then the number of ends of $\cG_d$ is independent of the generating set and $\cG_d$ has either one, two, or infinitely many ends. Moreover, $\cG_d$ has exactly two ends only if $\mG$ is virtually infinite cyclic (it has a finite normal subgroup $\mN$ such that the quotient group $\mG/\mN$ is isomorphic either to $\Z$ or $\Z_2\ast\Z_2$).
These results are due to Freudenthal~\cite{f44} and Hopf~\cite{hop} (see also~\cite{wall}). The classification of finitely generated groups with infinitely many ends is due to Stallings~\cite{s71}. Let us mention that if $\mG$ has infinitely many ends, then the result of Stallings implies that it contains a non-Abelian free subgroup and hence is non-amenable. For further details we refer to, e.g.,~\cite[Chapter~13]{geog}.
\item[(v)] Let us also mention that by Halin's theorem~\cite{hal} every locally finite graph $\cG_d$ with infinitely many ends contains at least one end which is not free.
\end{itemize}
\end{remark}

One of the main features of graph ends is that they provide a rather refined way of compactifying graphs (see~\cite{f31} and~\cite[\S~8.6]{die}, \cite{woe}).
Namely, we introduce a topology on $\wh{\cG} := \cG \cup \gC (\cG)$\label{p:FreudenthalCompactification} as follows.
For an open subset $U \subseteq \cG$, denote its extension $\wh U$ to $\wh{\cG}$ by 
\begin{align} \label{eq:top1}
	\wh{U}:= U \cup \{\gamma \in \gC (\cG)\, |  \; \exists\, \cU = (U_n)\in\gamma\ \text{such that}\ U_0 \subset U \}.
\end{align}
Now we can introduce a neighborhood basis of $\gamma\in \gC(\cG)$ as follows
\begin{align} \label{eq:top2}
\{\wh{U}\,|\, U \subseteq \cG\ \text{is open}, \gamma\in \wh{U} \}.
\end{align}
This turns $\wh{\cG}$  into a compact topological space, called the \emph{end (or Freudenthal) compactification} of $\cG$.

\begin{remark} \label{rem:freetop} 
Notice that an end $\gamma \in \gC(\cG)$ is free exactly when  $\{\gamma \}$ is open as a subset of $\gC(\cG)$ (here $\gC(\cG)$ carries the induced topology from $\wh{\cG}$). This is further equivalent to the existence of a connected subgraph $\cG^\gamma$ with compact boundary\footnote{Notice that for a subgraph $\wt\cG$ of $\cG$ its boundary is  $\partial\wt\cG = \{v\in \cV(\wt\cG)| \deg_{\wt\cG}(v) < \deg_\cG (v)\}$ and hence  $\partial \wt \cG$ is compact only if $\# \partial \wt \cG < \infty$.} $\partial \cG^\gamma$ such that 
$U_n \subseteq \cG^\gamma$ eventually for any sequence $ \cU = (U_n)$ representing $ \gamma$ and $U_n' \cap \cG^\gamma = \varnothing$ eventually for all sequences $\cU' = (U_n') $ representing an end $\gamma' \neq \gamma$.  
\end{remark}

Let us mention that topological ends can be obtained in a constructive way by means of compact exhaustions. Namely, 
a sequence of connected subgraphs $(\mathcal{F}_n)$ of $\cG$ such that each $\mathcal{F}_n$ has finitely many vertices and edges, $\mathcal{F}_n \subseteq \mathcal{F}_{n+1}$ for all $n\ge 0$ and $\bigcup_n \mathcal{F}_n = \cG$ is called a \emph{compact exhaustion} of $\cG$. Clearly, each $\mathcal{F}_n$ may be identified with a compact subset of $\cG$. Now iteratively construct a sequence $(U_n)$ by choosing in each step a non-compact, connected component $U_n$ of $\cG \setminus \mathcal{F}_n$ satisfying $U_{n}\subseteq U_{n-1}$. It is easy to check that each such sequence $(U_n)$ defines a topological end $\gamma \in \gC(\cG)$ and in fact all ends $\gamma \in \gC(\cG)$ are obtained by this construction. Notice also that the open subsets $U_n$ of such representations $\gamma \sim (U_n)$ (actually, their topological closures, since we need to add endpoints of edges which also belong to $\cV(\mathcal{F}_n)$) can again be identified with connected subgraphs $\mathcal{G}_n (\gamma) := \overline{ U_n}$ and we will frequently use this fact.

Let us finish this section with a few more notations. Suppose $\cR$ is a ray or a finite path without self-intersections in $\cG_d$. We may identify $\cR$ with a subgraph of $\cG_d$ and hence with a subset of $\cG$, i.e., we can consider it as the union of all edges of $\cR$. The latter can further be identified with the interval $I_\cR := [0,|\cR|)$ of length $|\cR|$, where
\[
	|\cR| := \sum_{e\in\cR} |e|.
\]
Also, we need to consider paths -- and in particular rays -- in $\cG$ starting or ending at a non-vertex point. In particular, given a path $ (v_0,v_1,\dots,v_N)$ and a point $x$ in the interior of some edge $e$ attached to $v_0$, $e \neq e_{v_0, v_1}$, we add the interval $[x, v_0] \subseteq e$ to $ (v_0,v_1,\dots,v_N)$.
 For the resulting set, we shall write $(x, v_0, v_1,\dots,v_N)$ and call it a \emph{non-vertex path}; and likewise for rays. The set of all non-vertex rays will be denoted by $\mR(\cG)$.

%%%%%%%%%%%%%%%%%%%%%%%%%%%%%%%%%%%%%%%%%%%%%%%%%%%%%%%%%%%%
\subsection{Kirchhoff Laplacian}\label{ss:II.03}
%%%%%%%%%%%%%%%%%%%%%%%%%%%%%%%%%%%%%%%%%%%%%%%%%%%%%%%%%%%%

Let $\cG$ be a metric graph satisfying Hypothesis~\ref{hyp:locfin}. Upon identifying every $e\in\cE$ with a copy of the interval $\cI_e = [0,|e|]$, we denote by
\[
L^2(e) := L^2(\cI_e; d{x_e})
\] 
the $L^2$-space for the (unweighted) Lebesgue measure $dx_e$ on ${\cI_e}$ 
and introduce the Hilbert space $L^2(\cG)$ of functions $f\colon \cG\to \C$ such that 
\[
L^2(\cG) := \bigoplus_{e\in\cE} L^2(e) = \Big\{f=\{f_e\}_{e\in\cE}\big|\, f_e\in L^2(e),\ \sum_{e\in\cE}\|f_e\|^2_{L^2(e)} <\infty\Big\}.
\]
The subspace of compactly supported $L^2(\cG)$ functions will be denoted by
\begin{equation*}
	L^2_c(\cG) := \big\{f \in L^2(\cG)| \; f \neq 0 \text{ only on finitely many edges } e \in \cE\big\}.
\end{equation*}
For every $e\in\cE$ consider the maximal operator $\rH_{e,\max}$ acting on functions $f\in H^2(e)$ as a negative second derivative. Here and below $H^s(e)$ for $s \ge 0$ denotes the usual Sobolev space on $e$ (see, e.g.,~\cite[Chapter~8]{bre}). In particular, $H^0(e)= L^2(e)$ and 
\[
H^1(e) = \{f\in AC(e)|\, f'\in L^2(e)\},\quad H^2(e) = \{f\in H^1(e)|\,  f'\in H^1(e)\}.
\]
This defines the maximal operator on $L^2(\cG)$ by 
\be\label{eq:Hmax}
\bH_{\max} := \bigoplus_{e\in \cE} \rH_{e,\max},\qquad \rH_{e,\max} = -\frac{\rD^2}{\rD x_e^2},\quad \dom(\rH_{e,\max}) = H^2(e).
\ee
If $v$ is a vertex of the edge $e \in \cE$, then for every $f\in H^2(e)$ the following quantities 
%(the trace and inner normal derivative at $v$) 
\begin{align}\label{eq:tr_fe}
f_e(v) & := \lim_{x_e\to v} f(x_e), &  f_e'(v) & := \lim_{x_e\to v} \frac{f(x_e) - f(v)}{|x_e - v|},
\end{align}
are well defined. Considering $\cG$ as the union of all edges glued together at certain endpoints, let us equip a metric graph with the Laplace operator.  The Kirchhoff (also called \emph{standard} or \emph{Kirchhoff--Neumann}) boundary conditions at every vertex $v\in\cV$ are then given by
\begin{align}\label{eq:kirchhoff}
\begin{cases} f\ \text{is continuous at}\ v,\\[1mm] 
\sum\limits_{e\in \cE_v}f_e'(v) =0. \end{cases}
\end{align}
Imposing these boundary conditions on the maximal domain $\dom(\bH_{\max})$ yields the \emph{maximal Kirchhoff Laplacian}
\begin{align}\label{eq:H}
\begin{split}
	\bH & :=  \bH_{\max}\upharpoonright {\dom(\bH)},\\ 
	 \dom(\bH ) &= \{f\in \dom(\bH_{\max})\,|\, f\ \text{satisfies}\ \eqref{eq:kirchhoff} \text{ for any } v\in\cV\}.
\end{split}
\end{align}
Restricting further to compactly supported functions we end up with the pre-minimal operator
\be\label{eq:H00}
\begin{split}
	\bH_{0}^0 & :=  \bH_{\max}\upharpoonright {\dom(\bH_{0}^0)},\\ 
	 \dom(\bH_{0}^0) & = \{f\in \dom(\bH_{\max})\cap L^2_{c}(\cG)|\, f\ \text{satisfies}\ \eqref{eq:kirchhoff} \text{ for any } v\in\cV\}.
\end{split}
\ee
Integrating by parts one obtains
\be \label{eq:integrationbp}
	\langle\bH_0^0 f, f\rangle_{L^2(\cG)} = \int_\cG |f'(x)|^2 \; dx , \qquad f \in \dom(\bH_0^0),
\ee
and hence $\bH_0^0$ is a non-negative symmetric operator. We call its closure $\bH_0 := \overline{ \bH_{0}^{0}}$ in $L^2(\cG)$ \emph{the minimal Kirchhoff Laplacian}. 
The following result is well-known (see, e.g.,~\cite[Lemma~3.9]{car08}).

\begin{lemma}\label{lem:H0*=H}
	Let $\cG$ be a metric graph. Then
\be\label{eq:H0*=H}
		\bH_0^\ast = \bH.
\ee
\end{lemma}

\iflong{}
\begin{proof}
We give the proof for the sake of completeness. Integration by parts shows that
	\[
		\langle\bH_0^0 f, g\rangle_{L^2(\cG)} = \langle f, \bH g \rangle_{L^2(\cG)}
	\] 
for all $f\in \dom(\bH_0^0)$ and $g\in \dom(\bH)$. Therefore, $g \in \dom(\bH_0^\ast)$ and hence $\bH \subseteq \bH_0^\ast$. 
	
	To prove the converse  inclusion, suppose that $g \in \dom (\bH_0^\ast)$. Fix an edge $e \in \cE$ and consider a test function $f\in \dom(\bH_0^0)$ such that $f$ equals zero everywhere except $e$. Then clearly $f \in H^2_0(e)$ and, moreover,
	\[
		 - \int_e f''(x) g(x)^\ast dx %= (\bH_0 f, g) =  (f, \bH_0^\ast g) 
		 = \int_e f(x)  g''(x)^\ast  dx.
	\]
	Thus, $g \in H^2(e)$ and the restriction of $\bH_0^\ast g$ on $e$ is simply given by $- g_e''$. Since $e\in\cE$ is arbitrary, this implies that $g \in \dom(\bH_{\max})$ and $\bH_0^\ast g = \bH_{\max} g$.
	
It remains to show that $g$ satisfies the Kirchhoff conditions \eqref{eq:kirchhoff}. Pick a vertex $v \in \cV$. If $\deg(v)=1$, then the claim is trivial. So, suppose $\deg(v)>1$. Let $e_1$ and $e_2$ be two distinct edges attached to $v$, $e_1,e_2\in \cE_v$. Then choose $f\in \dom(\bH_0^0)$ such that $f\equiv 0$ on $\cE\setminus\{e_1,e_2\}$, $f (v) = 0$ and $f_{e_1}' (v) = - f_{e_2}' (v) = 1$. Hence
	\[
		0 = \langle\bH_0 f, g\rangle - \langle f, \bH_0^\ast g\rangle = \sum_{n=1}^2  \left(f_{e_n}'(v) g_{e_n}(v)^\ast - f_{e_n}(v)  g_{e_n}'(v)^\ast\right) =  (g_{e_1} (v) -  g_{e_2}(v))^\ast,
	\]
which implies that $g_e(v)$ is independent of $e \in \cE_v$. Thus, $g $ is continuous.  

Next, choose $f\in \dom(\bH_0^0)$ such that $f \equiv 0$ on $\cE\setminus\cE_v$. Moreover, for every $e\in\cE_v$ we assume that $f(x_e) = 1$ if $x_e \in e$ and $|x_e - v| < |e|/4$ and $f(x_e) = 0$ if $|x_e - v| > |e|/2$. Thus we get
	\begin{align*}
		0 = \langle\bH_0 f, g\rangle - \langle f, \bH_0^\ast g\rangle = \sum_{e \in \cE_v} \left( f_e'(v) g_e(v)^\ast  - f_e(v)  g_e'(v)^\ast \right)=-\sum_{e \in \cE_v} g_e'(v)^\ast.
	\end{align*}
Hence $g$ satisfies \eqref{eq:kirchhoff} at every $v \in \cV$ and the proof is complete.
\end{proof}\fi

%%%%%%%%%%%%%%%%%%%%%%%%%%
\subsection{Deficiency indices}\label{ss:II.04}
%%%%%%%%%%%%%%%%%%%%%%%%%%
In the following we are interested in the question whether $\bH_0$ is self-adjoint, or equivalently whether the equality $\bH_0 = \bH$ holds true. Let us recall one sufficient condition. 
Define the \emph{star weight} $m(v)$ of a vertex $v \in \cV$ by
\be\label{def:m}
		m(v) := \sum_{e \in \cE_v}|e| = \vol(\cE_v),
\ee
and also introduce the \emph{star path metric} on $\cV$ by 
\be\label{def:rho_m}
\varrho_m(u,v) := \inf_{\substack{\cP=(v_0,\dots,v_n)\\ u=v_0,\ v=v_n}}\sum_{v_k\in\cP} m(v_k).
\ee

\begin{theorem}[\cite{ekmn}]\label{th:ekmn}
If $(\cV,\varrho_m)$ is complete as a metric space, 
then $\bH_0^0$ is essentially self-adjoint and $\overline{\bH_0^0} = \bH_0 = \bH$.
\end{theorem}

If a symmetric operator is not (essentially) self-adjoint, then the degree of its non-self-adjointness is determined by its \emph{deficiency indices}. 
Recall that the \emph{deficiency subspace} $\cN_z(\bH_0)$ of $\bH_0$ is defined by
\be\label{def:Nz}
\cN_z(\bH_0) := \ker(\bH_0^\ast - z)= \ker(\bH - z),\quad z\in\C.
\ee 
The numbers 
\be\label{def:n_pm}
\Nr_\pm(\bH_0) := \dim\cN_{\pm\I}(\bH_0) = \dim\ker(\bH \mp \I)
\ee
are called the deficiency indices of $\bH_0$. Notice that $\Nr_+(\bH_0)=\Nr_-(\bH_0)$ since $\bH_0$ is non-negative.
Moreover, $\bH_0$ is self-adjoint exactly when $\Nr_+ (\bH_0) =\Nr_- (\bH_0) = 0$.

\begin{lemma}\label{lem:npm=ker}
If $0$ is a point of regular type for $\bH_0$, then\footnote{For an operator $T$ with dense domain in a Hilbert space $\cH$,  $\lambda\in\C$ is called a \textit{point of regular type} of $T$ if there exists $c=c_\lambda>0$ such that $\|(T-\lambda)f\|\ge c\|f\|$ for all $f\in \dom(T)$.} 
\be\label{eq:npm=ker}
\Nr_\pm(\bH_0) = \dim\ker(\bH).
\ee
\end{lemma}

\begin{proof}
The claim immediately follows from~\cite[\S~78]{ag} or~\cite[Prop.~3.3]{schm}. Indeed, the set of regular points of $\bH_0$ is an open subset of $\C$. Moreover, by the Krasnoselskii--Krein theorem (see, e.g.,~\cite[\S~78]{ag} or~\cite[Prop.~2.4]{schm}), $\dim\cN_z(\bH_0)$ is constant on each connected component of the set of regular type points of $\bH_0$. Since $\bH_0$ is symmetric, each $z\in\C\setminus\R$ is a point of regular type for $\bH_0$. Therefore, if $0$ is a point of regular type for $\bH_0$, we immediately get 
\[
\dim\ker(\bH) = \dim\cN_{0}(\bH_0) = \Nr_+(\bH_0) = \Nr_- (\bH_0).\qedhere
\]
\end{proof}

Using the Rayleigh quotient, define
\be\label{eq:lowerbound}
\lambda_0(\cG):= \mathop{\inf_{f\in \dom(\bH_0)}}_{\|f\|=1} \big\langle\bH_0f,f\big\rangle_{L^2(\cG)} = \mathop{\inf_{f\in \dom(\bH_0)}}_{\|f\|=1}\int_{\cG}|f'|^2 dx.
\ee
Noting that the operator $\bH_0$ is non-negative, $0$ is a point of regular type for $\bH_0$ if $\lambda_0(\cG)>0$. Thus, we arrive at the following result. 

\begin{corollary}\label{cor:npm=ker}
If $\lambda_0(\cG)>0$, 
then \eqref{eq:npm=ker} holds true.
\end{corollary}

The positivity of $\lambda_0(\cG)$ is known in the following simple situation.

\begin{corollary}\label{cor:finvol=ker}
If $\cG$ has \textit{finite total volume}, 
\be\label{hyp:finitevol}
		\vol(\cG) := \sum_{e \in \cE} |e| < \infty,
\ee
then $\bH_0$ is not self-adjoint and \eqref{eq:npm=ker} holds true.
\end{corollary}

\begin{proof}
Indeed, by the Cheeger-type estimate~\cite[Corollary~3.5(iv)]{kn17}, we have
\begin{equation} \label{eq:cheeg}
	\lambda_0(\cG) \ge \frac{1}{4\, \vol(\cG)^2} >0,
\end{equation}
and hence \eqref{eq:npm=ker} holds true by Corollary~\ref{cor:npm=ker}. 
Moreover,  $\id_\cG \in \ker(\bH)$, where $\id_{\cG}$ denotes the constant function on $\cG$, and hence
\[
	\Nr_\pm (\bH_0) = \dim (\ker  \bH) \ge 1.
\]
It remains to notice that $\bH_0$ is self-adjoint exactly when $\Nr_\pm (\bH_0) = 0$.
\end{proof}

\begin{remark}
By~\cite[Theorem~3.4]{kn17}, $\lambda_0(\cG)>0$ holds true if the isoperimetric constant $\alpha(\cG)$ of the metric graph $\cG$ is positive. For antitrees, the isoperimetric constant is tightly related to the structure of its combinatorial spheres (see~\cite[Theorem~7.1]{kn19}). If $\cG_d$ is the edge graph of a tessellation of $\R^2$, then positivity of $\alpha(\cG)$ can be deduced from certain curvature-type quantities~\cite{noe20}.

On the other hand, by~\cite[Corollary~4.5(i)]{kn17}, $\lambda_0(\cG)>0$ holds true if the combinatorial isoperimetric constant of $\cG_d$ is positive and $\ell^\ast(\cG) := \sup_{e\in\cE}|e|<\infty$. 
For example, this holds true if $\cG_d$ is an infinite tree without leaves~\cite[Lemma~8.1]{kn17} or if $\cG_d$ is a Cayley graph of a non-amenable finitely generated group~\cite[Lemma~8.12(i)]{kn17}. 
\end{remark}

Finally, let us remark that $\ker(\bH) = \HH(\cG) \cap L^2( \cG)$, where $\HH(\cG)$ denotes the space of \emph{harmonic functions} on $ \cG$, that is, the set of all ``edgewise" affine functions satisfying Kirchhoff conditions \eqref{eq:kirchhoff} at each vertex $v \in \cV$. Notice that every function $f \in \HH(\cG)$ is uniquely determined by its vertex values $\f := f|_\cV =(f(v))_{v \in \cV}$. Recall also the following result (see, e.g.,~\cite[eq.~(2.32)]{kn17}).

\begin{lemma}\label{lem:Harmcont=Harmdiscr}
Let $\cG$ be a metric graph satisfying the assumptions in Hypothesis~\ref{hyp:locfin}. If $f\in \HH(\cG)$, then  $f \in L^2(\cG)$ if and only if $\f \in \ell^2(\cV;m)$, that is,
	\be \label{eq:g02}
			\sum_{v \in \cV} |f(v)|^2 m(v) < \infty.
	\ee
\end{lemma}

\begin{remark}\label{remSA=Harm}
The above considerations indicate that in order to understand the deficiency indices of the Kirchhoff Laplacian one needs to find the dimension of the space of $L^2$ harmonic (or, more carefully, $\lambda$-harmonic) functions. Moreover, in order to describe self-adjoint extensions one has to understand the behavior of $L^2$ harmonic functions at ``infinity", that is, near a ``boundary" of a given metric graph. However, graphs admit a lot of different notions of boundary (ends, Poisson and Martin boundaries, Royden and Kuramochi boundary etc.) and search for a suitable notion in this context (namely, $L^2$ harmonic functions) is a highly nontrivial problem, which seems to be not very well studied neither in the context of incomplete manifolds nor in the case of weighted graphs.

Let us also mention that recently there has been a tremendous amount of work devoted to the study of harmonic functions and self-adjoint extensions of Laplacians on weighted graph (we only refer to a brief selection of articles~\cite{cdvtht, ghklw, hklw, hu19, huke14, hkmw, jp, kl12}).
\end{remark}

%%%%%%%%%%%%%%%%%%%%%%%%
%%%%%%%%%%%%%%%%%%%%%%%%
\section{Graph ends and $H^1(\cG)$} \label{sec:Ends}
%%%%%%%%%%%%%%%%%%%%%%%%
%%%%%%%%%%%%%%%%%%%%%%%%

This section deals with the Sobolev space $H^1$ on metric graphs. Its importance stems, in particular, from the fact that it serves as a form domain for a large class of self-adjoint extensions of $\bH_0$.

\subsection{$H^1(\cG)$ and boundary values} \label{ss:III:01}
First recall that
\begin{equation}
	H^1(\cG) = \big\{f \in L^2(\cG)\cap C(\cG) | \;  f_e \in H^1(e) \ \text{for all}\ e \in \cE, \  \| f'\|^2_{L^2(\cG)} < \infty \big\},
\end{equation}
where $C(\cG)$ is the space of continuous complex-valued functions on $\cG$ and 
\[
\| f'\|^2_{L^2(\cG)} := \sum_{e \in \cE} \|f_e' \|^2_{L^2(e)}.
\]
Notice that $(H^1(\cG),  \| \cdot \|_{H^1})$ is a Hilbert space when equipped with the standard norm
\[
\| f \|^2_{H^1(\cG)} := \| f \|^2_{L^2(\cG)} +  \| f '\|^2_{L^2(\cG)} = \sum_{e\in\cE}\|f_e \|^2_{H^1(e)},\quad f \in H^1(\cG).
\]
Moreover, $\dom(\bH_0^0) \subset H^1(\cG)$ and we define $H^1_0( \cG)$ as the closure of $\dom(\bH_0^0)$ with respect to $\| \cdot \|_{H^1(\cG)}$. 

\begin{remark}\label{rem:Friedrichs}
If $\bH_0^0$ is essentially self-adjoint, then $H^1(\cG) = H^1_0(\cG)$. However, the converse is not true in general. In fact this equality is tightly connected to the uniqueness of Markovian extensions of $\bH_0$ and, as we shall see, it is possible to characterize it in terms of topological ends of $\cG$ (see Corollary~\ref{cor:Markovian1} below).

Notice also that $H^1_0( \cG)$ is the form domain of the Friedrichs extension $\bH_F$\label{p:Friedrichsextension} of $\bH_0^0$ and $\lambda_0(\cG)$ defined by \eqref{eq:lowerbound} is the bottom of the spectrum of $\bH_F$.
\end{remark}

By definition, $H^1(\cG)$ is densely and continuously embedded in $L^2(\cG)$. 

\begin{lemma} \label{lem:inftyh1}
$H^1 (\cG)$ is continuously embedded in $C_b(\cG) = C(\cG)\cap L^\infty(\cG)$ and 
\begin{equation} \label{eq:inftyh1}
	\| f \|_{\infty} := \sup_{x \in \cG} |f(x)| \le C_\cG \| f \|_{H^1 (\cG)} 
\end{equation}
holds for all $ f \in H^1(\cG)$ with $C_\cG := \sqrt{\coth\big(\frac{1}{2}\sup_\cR |\cR|\big)}$, where the supremum is taken over all non-vertex paths without self-intersections.
\end{lemma}

\begin{proof} 
For every interval $\cI\subseteq \R$ the embedding of $H^1(\cI)$ into $L^\infty(\cI)$ is bounded and 
\be \label{eq:linftyinterval}
	\sup_{x \in \cI} | f(x)| \le C_{|\cI|} \| f \|_{H^1(\cI)} 
\ee
holds for all $ f \in H^1(\cI)$ with $C_{|\cI|}= \sqrt{\coth(|\cI|)}$ (see~\cite{mar83}). Notice that we may identify the restriction $f{|_{\cR}}$ of $f \in H^1(\cG)$ to a (non-vertex) path without self-intersections $\cR$ with a function on $\cI_\cR = [0, |\cR|)$. It is easy to check that upon this identification $f|_\cR \in H^1(\cI_\cR)$ and $(f{|_{\cR}})' = f'{|_{\cR}}$.

Suppose now that $\cR$ is a fixed non-vertex path without self-intersections in $\cG$. Then for every $x \in \cG$, connecting $x$ and $\cR$  by some finite non-vertex path $\cR_0$, we see that there exists a non-vertex path without self-intersections $\cR_x$ such that $x\in\cR_x$ and $|\cR_x| \ge |\cR|/2$ (if $x$ lies on $\cR$ already, then the connecting argument is superfluous and we can simply take a portion of $\cR$). Applying \eqref{eq:linftyinterval} to $\cR_x$, we easily deduce the estimate \eqref{eq:inftyh1}.
%If $\cR$ is a non-vertex path without self-intersections,  then considering it as an interval $I_\cR = [0,|\cR|)$ of length $|\cR|$, we immediately get 
%\be\label{eq:Rx}
%|f(x)| \le C_{|\cR| /2 }\|f\|_{H^1(\cR)} \le C_{|\cR| /2 }\|f\|_{H^1(\cG)}
%\ee
%for all $f\in H^1(\cG)$. 
%If $x \not \in \cR$, then connecting $x$ and $\cR$  by some finite non-vertex path $\cR_0$, we conclude that there is a non-vertex path without self-intersections $\cR_x$ such that $x\in\cR_x$ and $|\cR_x| \ge |\cR|/2$. Applying the same argument, we conclude that \eqref{eq:Rx} holds for all $x\in\cG$. 
\end{proof}

\begin{remark}
The diameter of $\cG$ (as a metric space $(\cG,\varrho)$) is defined by  
\begin{align} \label{eq:diameter}
\diam(\cG) := \sup_{x,y\in \cG}\varrho(x,y).  % = \sup_{u,v\in \cV}\varrho(u,v) 
\end{align}
Therefore, $\diam(\cG) \le \sup_\cR |\cR|$ and hence $C_\cG \le  \sqrt{\coth\big(\frac{1}{2}\diam(\cG)\big)}$.
\end{remark}

The above considerations, in particular, imply the following crucial property of $H^1$-functions: if $ \cR = (v_n)$ is a ray, then 
\[
	f(\gamma_\cR) := \lim_{ n \to \infty} f(v_n)
\]
exists. Indeed, upon the identification of $\cR$ with the interval $\cI_\cR = [0,|\cR|)$, the latter is an immediate corollary of the description of a Sobolev space $H^1$ in one dimension --- for a bounded interval this follows from~\cite[Theorem~8.2]{bre} and in the unbounded case see~\cite[Corollary~8.9]{bre}. Moreover, this limit only depends on the equivalence class of $\cR$  (indeed, for any two equivalent rays $\cR$ and $\cR'$  there exists a third ray $\cR''$ containing infinitely many vertices of both $\cR$ and $\cR'$, which immediately implies that $f(\gamma_\cR) = f(\gamma_{\cR''}) = f(\gamma_{\cR'})$). This enables us to introduce the following notion.

\begin{definition}\label{def:context}
For every $f \in H^1(\cG)$ and a (topological) end $\gamma \in \gC(\cG)$, we  define
\begin{equation} \label{eq:defvalueend}
	f( \gamma)  := f(\gamma_\cR),
\end{equation}
where $\cR \in \omega_\gamma$ is any ray belonging to the corresponding graph end $\omega_\gamma$ (see Theorem~\ref{th:top-comb}). Sometimes we shall also write $ f(\omega_\gamma) := f( \gamma)$.
\end{definition}

It turns out that \eqref{eq:defvalueend} enables us to obtain an extension  by continuity of every function $f\in\ H^1(\cG)$ to the end compactification $\wh{\cG}$ of $\cG$ (see Section~\ref{ss:II.02}). 

\begin{lemma} \label{lem:uniformends}
	Let $\cG $ be a metric graph and $\gamma \in \gC(\cG)$. If $f \in H^1(\cG)$, then
	\be\label{eq:continuity}
		\lim_{n \to \infty} \sup_{x \in U_n} |f(x) - f(\gamma) | =0
	\ee
	for every sequence $\cU= (U_n)$ representing $\gamma$. 
\end{lemma}

\begin{proof}
Let $\gamma \in \gC(\cG)$ and let $\cU = (U_n)$ be a sequence representing $\gamma$. Let also 
	\[
	\mR_n(\gamma) := \{ \cR  \in \mR(\cG) | \;  \cR \subseteq U_n \}
	\]
be the set of all non-vertex rays contained in $U_n$, $n\ge 0$. 

We proceed by case distinction. 	
First, assume that for $n$ sufficiently large, all rays in $\mR_n(\gamma)$ have length at most one. If $x \in U_n$, then there exists a (non-vertex) ray $\cR_x\in \mR_n(\gamma)$ such that $\cR_x = (x,v_0,\dots)$ and its tail $\cR:=(v_0,v_1,\dots)$ (see Definition~\ref{def:rays}) belong to 
$\omega_\gamma$.

	By our assumption, $|\cR_x| \le 1$ and hence
\begin{align*}
		|f(\gamma) - f(x)| = |f(\gamma_{\cR_x}) - f(x)| =\Big | \int_{\cR_x} f'(y) \; dy \Big | \le  \| f' \|_{L^2(\cR_x)} \le  \| f' \|_{L^2(U_n)}.
	\end{align*}
Since $x \in U_n$ is arbitrary, this implies
\[
	\sup_{x \in U_n} |f(\gamma) - f(x)| \le \| f' \|_{L^2(U_n)}.
\]
Since $\cU=(U_n)$ represents $\gamma$,  $\bigcap_n \overline{U_n} = \varnothing$ and hence $\lim_{n \to \infty}   \| f' \|_{L^2(U_n)} =0$. This implies \eqref{eq:continuity}.

Assume now that for every $n \in \Z_{\ge 0}$ there is a ray $\cR \in \mR_n(\gamma)$ with $|\cR| > 1$. Take $n \ge 0$ and choose an  $x \in U_n$. We can find a finite (non-vertex) path without self-intersections $\cR_x\subseteq U_n$ such that $x \in \cR_x$ and 
  $|\cR_x| = 1/2$ (take into account that $U_n$ contains at least one ray of length greater than $1$). 
Hence we get  
 \[
 	|f(x)| \le \sup_{y \in \cR_x} |f(y)| \le C_{1/2} \| f\|_{H^1( \cR_x)} \le C_{1/2}  \| f\|_{H^1(U_n)},
 \]
 where $C_{1/2} = \sqrt{\coth(1/2)}$ is the constant from \eqref{eq:linftyinterval}.  
Since $x \in U_n$ is arbitrary, 
\begin{align*}
\sup_{x \in U_n} |f(x)| \le C_{1/2}   \| f\|_{H^1(U_n)}.
\end{align*}
However, $\bigcap_n \overline{U_n} = \varnothing$ and hence $\sup_{x \in U_n} |f(x)| =o(1)$ as $n \to \infty$. It remains to notice that $f(\gamma)= 0$. Indeed, by Theorem~\ref{th:top-comb}, for every $n \ge 0$ there is a ray $\wt \cR_n\in \omega_\gamma$ such that $\wt \cR_n \subseteq U_n$ and hence
 \[
 	|f(\gamma)| = |f(\gamma_{\wt \cR_n}) | \le \sup_{x \in U_n} |f(x)|  = o(1)
 \]
as $n \to \infty$. This finishes the proof.
 \end{proof}

Taking into account the topology on $\wh{\cG} = \cG \cup \gC (\cG)$, the next result  is a direct consequence of Lemma~\ref{lem:inftyh1} and Lemma~\ref{lem:uniformends}.

\begin{proposition} \label{prop:h1trace}
	Each $f \in H^1(\cG)$ has a unique continuous extension to the end compactification $\wh{\cG}$ of $\cG$ and this extension is given by \eqref{eq:defvalueend}.  Moreover, 
	\[
		\|  f \|_{\infty} = \sup_{x \in \wh{\cG}} |f(x)| \le C_\cG \| f \|_{H^1(\cG)}.
	\]
\end{proposition}

%%%%%%%%%%%%%%%%%%%%%%%
\subsection{Nontrivial and finite volume ends}
%%%%%%%%%%%%%%%%%%%%%%%

Observe that some ends lead to trivial boundary values for $H^1$ functions. For example,  $f(\gamma)  = 0$ for all $f  \in H^1(\cG)$ if $\omega_\gamma \in \Omega(\cG_d)$ contains a ray $\cR$ with infinite length $|\cR| =  \infty$. 
On the other hand, it might happen that all rays have finite length, however,  $f(\gamma) = 0$ for all $f  \in H^1(\cG)$ (see, e.g., the second step in the proof of Lemma~\ref{lem:uniformends}). 

\begin{definition}\label{def:endnontr}
A topological end $\gamma \in \gC(\cG)$ is called \emph{nontrivial} if $f(\gamma)\neq 0$ for some $f\in H^1(\cG)$.
\end{definition}

We also need the following notion.

\begin{definition} \label{def:finvol}
A topological end $\gamma \in \gC(\cG)$ has \emph{finite volume} (or, more precisely, \emph{finite volume neighborhood)} if there is a sequence $\cU = (U_n)$  representing $\gamma$ such that $\vol(U_n) < \infty$ for some $n$. Otherwise $\gamma$ has \emph{infinite volume}. The set of all finite volume ends is denoted by $\gC_0(\cG)$. 
Here and below, $\vol(A)$ is the Lebesgue measure of a measurable set $A \subseteq \cG$.  
\end{definition}

\begin{remark}\label{rem:finvolends}
If $\gC(\cG)$ contains only one end, then this end has finite volume exactly when $\vol(\cG)<\infty$. Analogously, if $\gamma \in \gC(\cG)$ is a free end, then there is a finite set of vertices $X$ separating $\omega_\gamma$ from all other ends and hence this end has finite volume exactly when the corresponding connected component $\cG_\gamma$ has finite total volume. 

If $\gamma$ is not free, then the situation is more complicated. For example, for a rooted tree $\cG=\cT_o$ the ends are in one-to-one correspondence with the rays from the root $o$ and hence one may possibly confuse the notion of a finite/infinite volume of an end with the finite/infinite length of the corresponding ray. More specifically, let $\gamma$ be an end of $\cT_o$ and let $\cR_\gamma = (o,v_1,v_2,\dots)$ be the corresponding ray. For each $n\ge 1$, let $\cT_n$ be the subtree of $\cT_o$ having its root at $v_n$ and containing all the ``descendant" vertices of $v_n$. Then  by definition $\gamma$ has finite volume (neighborhood) if and only if there is $n\ge 1$ such that the corresponding subtree $\cT_n$ has finite total volume. In particular, this implies that $\cG$ would have uncountably many finite volume ends in this case (here we assume for simplicity that all vertices are essential, that is, $\deg(v)>2$ for all $v \in \cV$). In particular,  $|\cR_\gamma|<\infty$ is a necessary but not sufficient condition for $\gamma$ to have finite volume.
\end{remark}

It turns out that nontrivial and finite volume ends are closely connected.

\begin{theorem} \label{thm:nontr=finitevol}
	Let $\cG $ be a metric graph. Then $\gamma \in \gC(\cG)$ is nontrivial if and only if $\gamma$ has finite volume.
	Moreover, for any finite collection of distinct nontrivial ends $\{\gamma_j\}_{j=1}^N$ there exists $f \in H^1(\cG) \cap \dom(\bH)$ such that $f(\gamma_1) = 1$ and $f(\gamma_2) = \dots = f(\gamma_N) = 0$.
\end{theorem}

\begin{proof}
It is not difficult to see that $f(\gamma) = 0$ for all $f \in H^1(\cG)$ if $\gamma$ has infinite volume. Indeed, assuming that there is $f\in H^1(\cG)$ such that $f(\gamma)\neq 0$, Lemma~\ref{lem:uniformends} would imply that there exists $\cU = (U_n)$ representing $\gamma$ such that 
\[
|f(x)| \ge |f(\gamma)|/2>0
\]
for all $x\in U_N$ and some $N \in \Z_{\ge 0}$. However, since $\vol(U_N) = \infty$, we conclude that $f$ is not in $L^2(\cG)$ and this gives a contradiction.

Suppose now that $\gamma\in \gC(\cG)$ has finite volume. Take a sequence $\cU = (U_n)$ representing $\gamma$ with $\vol(U_0)<\infty$.  Pick a function $\phi \in H^2(0,1)$ such that $\phi(0) = \phi'(0)= \phi'(1) = 0$ and $\phi(1) = 1$ and then define $f\colon \cG\to \C$ by
\begin{align*}
	f(x_e) = \begin{cases} 1, & x_e\in e\ \text{and both vertices of } e \text{ are in } U_0, \\
					  0, &  x_e\in e\ \text{and both vertices of } e \text{ are not in }  U_0, \\
					\phi \Big ( \frac{|x_e- u|}{|e|} \Big ), & x_e\in e = e_{u,v} \text{ and } u \in \cV \setminus U_0, v \in U_0.
		    \end{cases}
\end{align*}
 Clearly, $f \in H^2(e)$ for every $e\in \cE$. Moreover, it is straightforward to check that $f$ satisfies Kirchhoff conditions \eqref{eq:kirchhoff} at every $v \in \cV$. By assumption, $\partial  U_0$ is compact and hence it is contained in finitely many edges. Thus there are only finitely many edges $e\in\cE$ such that one of its vertices belongs to $U_0$ and the other one does not belong to $U_0$. This implies that $f\in L^2(\cG)$ and, moreover, $f' \not\equiv 0$ only on finitely many edges, which proves the inclusion $ f \in \dom(\bH)\cap H^1(\cG)$. Taking into account that $f\equiv 1$ on $U_n$ for large enough $n$, we conclude that $f(\gamma)=1$ and hence $\gamma$ is nontrivial. 

It remains to prove the second claim. Suppose that $\gamma_1, \dots, \gamma_N \in \gC(\cG)$ are distinct nontrivial ends. Then we can find $\cU^j = (U_n^j)$, sequences representing $\gamma_j$, $j\in \{1,\dots,N\}$, such that $\vol(U_0^1) < \infty$ and $U_0^1 \cap U_0^j = \emptyset $ for all $j=2,\dots,N$ (see~\cite[Satz~3]{f31} or~\cite[Lemma~3.1]{dikue}).  Using the above procedure, we can construct a function $f\in \dom(\bH)\cap H^1(\cG)$ such that $\supp(f)\subseteq U_0$ and $f(\gamma)=1$. The latter also implies that $f(\gamma_2) = \dots = f(\gamma_N) = 0$.
\end{proof}

\begin{remark}\label{rem:finvol}
 If $\vol(\cG) = \sum_{e \in \cE} |e| < \infty$, then all ends have finite volume and the end compactification $\widehat \cG$ of $\cG$ coincides with several other spaces, among them the \emph{metric completion} of $\cG$ and the \emph{Royden compactification} of a related discrete graph (see~\cite[Corollary~4.22]{ghklw} and also~\cite[p.~1526]{g11}). 
 Notice that the natural path metric $\varrhoo$ can be extended to $\wh{\cG} = \cG \cup \gC(\cG)$ (see~\cite{g11}). That is, the distance $\varrhoo(x, \gamma)$ between a point $x \in \cG$ and an end $\gamma \in \gC(\cG)$  is the infimum over all lengths of rays starting at $x$ and belonging to $\gamma$. Similarly, the distance $\varrhoo(\gamma, \gamma')$ between two ends is the infimum over the lengths of all double rays with one tail part in $\gamma$ and the other one in $\gamma'$. Then $(\wh{\cG}, \varrhoo)$ is a metric completion of $\cG$ and $\wh{\cG}$ is compact and homeomorphic to the end compactification of $\cG$ (see~\cite{g11} for further details). 

 The metric completion was considered in connection with quantum graphs in~\cite{car08, car12}; however, it can have a rather complicated structure if $\vol(\cG)=\infty$ and a further analysis usually requires additional assumptions. Moreover, there are clear indications that metric completion is not a good candidate for these purposes. 
\end{remark}

%%%%%%%%%%%%%%%%%%%%%%%%%%%%%%%%
\subsection{Description of $H^1_0 (\mathcal{G})$} \label{ss:H10}
%%%%%%%%%%%%%%%%%%%%%%%%%%%%%%%%

Recall that the space $H^1_0( \cG)$ is defined as the closure of $\dom(\bH_0^0) \subset H^1(\cG)$ with respect to $\| \cdot \|_{H^1(\cG)}$. 
One can naturally conjecture that $H^1_0(\cG)$ consists of those $H^1$-functions which vanish on $\gC(\cG)$. In fact, the results of the previous two sections enable us to show that this is indeed the case.
  
\begin{theorem}\label{th:H10}
Let $\cG$ be a metric graph and $\gC(\cG)$ be its ends. Then
\begin{equation} \label{eq:h10} 
	H^1_0 (\cG) =\{ f \in H^1(\cG) \,| \; f(\gamma) = 0 \text{ for all } \gamma \in \gC(\cG) \}.
\end{equation}
\end{theorem}

\begin{proof}
First of all, it immediately follows from Proposition~\ref{prop:h1trace} that $f \in H^1_0(\cG)$ vanishes at every  end $\gamma \in \gC (\cG)$ (since this holds for each $f \in \dom(\bH_0^0)$). 

To prove the converse inclusion, we will follow the arguments of the proof of~\cite[Theorem~4.14]{ghklw}. Namely, suppose that $f \in H^1(\cG)$ and $f(\gamma) = 0$ for all $\gamma \in \gC(\cG)$. Without loss of generality, we may assume that $f$ is real-valued and $f \ge 0$. To prove that $f \in H^1_0(\cG)$,  it suffices to construct a sequence of compactly supported functions $f_n\in H^1(\cG)$ which converges to $f$ in $H^1(\cG)$. 
 Define $\phi_n \colon [0,\infty) \to  [0,\infty)$ by
\be\label{eq:phin}
	\phi_n (s) := \begin{cases} s - \frac{1}{n}, & \text{ if } s \ge \frac{1}{n}, \\ 0, &  \text{ if } s < \frac{1}{n}, \end{cases}
\ee
and then let $f_n\colon \cG \to  [0,\infty)$ be the composition $f_n := \phi_n \circ f$, $n\ge0$. Since $\phi_n (s) \le s$ for all $s\ge0$ and $|\phi_n(s) - \phi_n(t)| \le |s-t|$ for all $s, t \ge 0$, $|f_n(x)| \le |f(x)|$ and $|f_n' (x)| \le |f'(x)|$ for almost every $x \in \cG$. Hence $f_n \in H^1(\cG)$ and 
\be
\|f_n\|_{H^1(\cG)} \le \|f\|_{H^1(\cG)}
\ee
for all $n$. Let us now show that $f_n$ has compact support. Indeed, assuming the converse, there exist infinitely many distinct edges $e_k$ in $\cE$ such that $f_{n}$ is  non-zero on each $e_k$. Taking into account \eqref{eq:phin}, for each $k$ we can find a non-vertex point $x_k$ on $e_k$ such that $f_n(x_k) > \frac{1}{n}$. Since $ \wh\cG$ is compact, the sequence $(x_k)$ has an accumulation point  $x \in \wh\cG$. By construction each edge $e \in \cE$ contains at most one of the $x_k$'s. It follows that $x \notin \cG$ and hence $x \in \wh\cG$ is an end. On the other hand, $f$ is continuous on $ \wh\cG $ by Proposition~\ref{prop:h1trace} and thus $f(x) \ge \frac{1}{n}$, which contradicts our assumptions on $f$. 

It remains to show that $f_n$ converges to $f$ in ${H^1(\cG)}$ as $n \to \infty$. Taking into account the above properties of $f_n$, we get
\[
	\|f - f_{n}\|^2_{L^2} + \|f' - f_{n}'\|^2_{L^2} \le 2 ( \|f\|_{L^2}^2 +  \|f_n\|_{L^2}^2 +  \|f'\|_{L^2}^2 +  \|f_n'\|_{L^2}^2) \le 4\|f\|^2_{H^1},
\]
and hence by dominated convergence it is enough to show that $f_{n} \to f$ and $f_{n}' \to f'$ pointwise a.e.\ on $\cG$. The first claim is clearly true since $\lim_{n \to \infty} \phi_n(s) = s$ for all $s \in  [0,\infty)$. To prove the second claim, suppose that $f$ is differentiable at a non-vertex point $x \in \cG$. If  $f(x) > 0 $, then by continuity of $f$, there is a neighborhood $U$ of $x$ such that $f_{n} = f - \frac{1}{n}$ holds on $U$ for all sufficiently large $n>0$. Hence $f_{n}$ is differentiable at $x$ with $f_{n}'(x) = f'(x)$ for all large enough $n$. Finally, if $f(x) = 0$, then for each $n$ there is a neighborhood $U_n$ of $x$ such that $f \le\frac{1}{n}$ on $U_n$. Hence $f_n \equiv 0$ on $U_n$ and, in particular, $f_n$ is differentiable at $x$ with $f_{n}'(x) = 0$. However, since $f \ge 0$ on $\cG$ and $f$ is differentiable at $x$, it follows that $f'(x) = 0$ as well. This finishes the proof.
\end{proof}

Combining Theorem~\ref{th:H10} with Theorem~\ref{thm:nontr=finitevol}, we arrive at the following fact.

\begin{corollary}\label{cor:H=H}
The equality $H^1(\cG) = H^1_0(\cG)$ holds true if and only if all topological ends of $\cG$ have infinite volume. 
\end{corollary}

\begin{remark}\label{rem:IdealBdry}
In the related setting of (weighted) discrete graphs, an important concept is the construction of boundaries by employing $C^\ast$-algebra techniques (this includes both {\em Royden} and {\em Kuramochi boundaries}, see~\cite{ghklw, kasue10, klmw, my, soardi} for further details and references). 
Finite volume graph ends can also be constructed by using this method. Indeed, $\mathcal{A} := H^1(\cG) \subset C_b (\cG)$ is a subalgebra by Lemma~\ref{lem:inftyh1} and hence its $\|\cdot\|_\infty$-closure $\widetilde{\cA} := \overline{\mathcal A}^{\|\cdot\|_\infty}$ is isomorphic to $C_0(\wt{X})$, where $\wt{X}$ is the space of characters equipped with the weak$^\ast$-topology with respect to $\widetilde{\cA}$. In general, finding $\wt{X}$ for some concrete $C^\ast$-algebra is a rather complicated task. However, it turns out that in our situation $\wt{X}$ coincides with $\wt \cG := \cG \cup \gC_0(\cG)$. Indeed, $\wt \cG = \cG \cup \gC_0(\cG)$ equipped with the induced topology of the end compactification $\widehat \cG$ is a locally compact Hausdorff space. Proposition~\ref{prop:h1trace} together with Theorem~\ref{thm:nontr=finitevol} shows that each function $f \in H^1(\cG)$ has a unique continuous extension to $\wt \cG$ and this extension belongs to $C_0(\wt \cG)$. Moreover, by Theorem~\ref{thm:nontr=finitevol}, $H^1(\cG)$ is point-separating and nowhere vanishing on $\wt \cG$ and hence $\widetilde{\cA} = C_0(\wt \cG)$ by the Stone--Weierstrass theorem. Thus the resulting boundary notion is precisely the space of finite volume graph ends. 

Let us also mention that $\wt\cG$ is compact only if $\vol(\cG)<\infty$ and in this case one can show that the Royden compactification of $\cG$ as well as its Kuramochi compactification coincide with the end compactification $\wh{\cG}$ (see~\cite{ghklw}, \cite[Theorem~7.11]{kasue10}, \cite[p.215]{kasue17} and also~\cite[p.2]{hs} for the discrete case).
\end{remark}

%%%%%%%%%%%%%%%%%%%%%%%%%%
%%%%%%%%%%%%%%%%%%%%%%%%%%
\section{Deficiency indices}\label{sec:DefInd}
%%%%%%%%%%%%%%%%%%%%%%%%%%
%%%%%%%%%%%%%%%%%%%%%%%%%%

Intuitively, deficiency indices should be linked to \emph{boundary notions} for underlying combinatorial graphs. However, spectral properties of the operator $\bH_0$ also depend on the edge lengths and this suggests that it is difficult to expect a purely combinatorial formula for the deficiency indices $\Nr_\pm (\bH_0)$ of $\bH_0$. Recall that throughout the paper we always assume that $\cG$ satisfies Hypothesis~\ref{hyp:locfin}.

%%%%%%%%%%%%%%%%%%%%%%%%%%
\subsection{Deficiency indices and graph ends}
%%%%%%%%%%%%%%%%%%%%%%%%%%

The main result of this section provides criteria which allow to connect $\Nr_\pm (\bH_0)$ with the number of graph ends.

\begin{theorem} \label{th:indicesends}
	Let $\cG$ be a metric graph and let $\bH_0$ be the corresponding minimal Kirchhoff Laplacian. Then
\be\label{eq:npm>ends}
		\Nr_\pm (\bH_0) \ge \# \gC_0(\cG).
\ee
Moreover, the equality 	
	\be\label{eq:npm=ends}
		\Nr_\pm (\bH_0) = \# \gC_0(\cG)
	\ee
holds true if and only if either $\# \gC_0(\cG) = \infty$ or $\dom(\bH) \subset H^1(\cG)$.
\end{theorem}

\begin{remark} \label{rem:sobolev}
 Since the map
\[
\begin{array}{cccc}
D \colon & H^1(\cG) & \to & L^2(\cG) \\
& f &\mapsto &f'
\end{array}
\]
 is bounded, the inclusion $\dom(\bH) \subset H^1(\cG)$ holds true if and only if there is a positive constant $C>0$ such that
\be \label{eq:sobolev}
	\| f' \|_{L^2(\cG)}^2 \le C ( \| f \|_{L^2(\cG)}^2 + \| f'' \|_{L^2(\cG)}^2 )
\ee
holds for all $f \in \dom(\bH)$. It can be shown by examples that \eqref{eq:sobolev} may fail. 
\end{remark}

Before proving Theorem~\ref{th:indicesends}, let us first comment on some of its immediate consequences. 

\begin{corollary}\label{cor:indicesends}
	If $\cG$ is a metric graph with finite total volume $\vol(\cG) < \infty$, then 
	\be \label{eq:npm>graphends}
		\Nr_\pm (\bH_0) \ge \# \Omega(\cG_d).
	\ee
Moreover, 
\be \label{eq:npm=graphends}
		\Nr_\pm (\bH_0) = \# \Omega(\cG_d)
	\ee
	 if and only  if either $\cG$ contains a non-free end (and hence $\# \Omega(\cG_d) = \infty$ in this case) or 
	 $\ker( \bH) \subset H^1(\cG)$.
\end{corollary}

In fact, we only need to mention that by Halin's theorem~\cite{hal} (see Remark~\ref{rem:ends}(v)) and the finite total volume of $\cG$, $\# \gC_0(\cG) = \infty$ only if $\cG$ contains a non-free end.

Recall that for a finitely generated group $\mG$, the number of graph ends of a Cayley graph is independent of the generating set (see, e.g.,~\cite{geog}). Combining this fact with the above statement, we obtain the following result.

\begin{corollary}\label{cor:stallings}
Let $\cG_d$ be a Cayley graph of a finitely generated group $\mG$ with infinitely many ends.\footnote{A classification of groups having infinitely many ends is given in Stallings's ends theorem~\cite{s71} (see also~\cite[Theorem~13.5.10]{geog} and Remark~\ref{rem:ends}(iv)).} If $\vol(\cG)<\infty$, then
$\Nr_\pm(\bH_0) = \infty$.
\end{corollary}

%%%%%%%%%%%%%%%%%%%%%%%%%%
\subsection{Proof of Theorem~\ref{th:indicesends}}
%%%%%%%%%%%%%%%%%%%%%%%%%%

The proof of Theorem~\ref{th:indicesends} is based on the following observation. Let $\bH_F$ be the Friedrichs extension of $\bH_0$.  Then $\dom(\bH)$ admits the following decomposition
\be \label{eq:dombHsum01}
	\dom(\bH) = \dom(\bH_F) \dotplus \ker(\bH - z) = \dom(\bH_F) \dotplus \cN_z(\bH_0),
\ee
for every $z$ in the resolvent set $\rho(\bH_F)$ of $\bH_F$ (see, e.g., \cite[Proposition~14.11]{schm}).
In particular, \eqref{eq:dombHsum01} holds for all $z\in (-\infty,\lambda_0(\cG))$, where $\lambda_0(\cG)\ge 0$ is defined by \eqref{eq:lowerbound}. 
Moreover, $\dom(\bH_F)\subset H^1_0(\cG)$ and hence the inclusion $\dom(\bH) \subset H^1(\cG)$ depends only on the inclusion $\ker(\bH - z) \subset H^1(\cG)$ for some (and hence for all) $z\in \rho(\bH_F)$.  
Let us stress that $\cN_0(\bH_0) = \ker(\bH) = \HH(\cG)\cap L^2(\cG)$ and hence in the case $\lambda_0(\cG)>0$, one is interested in whether all $L^2$ harmonic functions belong to $H^1(\cG)$ or not, which is known to depend on the geometry of the underlying metric graph. 

We also need the fact that functions in $\cN_{\lambda}(\bH_0)$ with $\lambda \in (-\infty,0)$ can be considered as subharmonic functions and hence they should satisfy a maximum principle.

\begin{lemma} \label{lem:maxprinciple}
	Suppose $\cG $ is a metric graph and let $\lambda \in(-\infty,0)$. 
	\begin{itemize} 
	\item[(i)] If $f \in \cN_{ \lambda}(\bH_0)= \ker(\bH-\lambda)$ is real-valued and $f(x_0) > 0$ for some $x_0 \in \cG$, then
	\be\label{eq:sup=sup}
\sup_{x \in \cG} f(x) = \sup_{v \in \cV} f(v).
\ee
	\item[(ii)] If additionally  $f\in H^1(\cG)$, then
	\be \label{eq:maxprinc01}
		\sup_{x \in \cG} f(x) = \sup_{\gamma \in \gC(\cG)} f(\gamma).
	\ee
	\item[(iii)] If  (not necessarily real-valued) $f \in \cN_{ \lambda}(\bH_0)  \cap H^1(\cG) $ satisfies
	\be \label{eq:maxprinc02}
		f(\gamma) = 0
	\ee
for all $\gamma \in \gC(\cG)$,	then $f \equiv 0$.
	\end{itemize}
\end{lemma}

\begin{proof}
(i) Let $f \in \cN_{ \lambda}(\bH_0)$ be real-valued. If $x\in \cG$ is such that $f(x)>0$ and $e\in\cE$ is an edge with $x\in e$, then upon identifying $e$ with the interval $\cI_e=[0,|e|]$ and taking into account that 
$- f'' = \lambda  f$ on $e$,
we get
\be\label{eq:finN}
	f(y) = f (x) \cosh\big( \sqrt{-\lambda}(y -x)\big) +  \frac{f'(x)}{\sqrt{-\lambda}}    \sinh\big( \sqrt{-\lambda}(y -x)\big)
\ee
for all $y \in e$. If $f'(x) \ge 0$, then obviously $f(e_i) \ge f(x)$, where $e_i$ is the vertex of $e$ identified with the right endpoint of $\cI_e$. Similarly, $f(e_o) \ge f(x)$ for the other vertex $e_o$ of $e$ if $f'(x) <0$. Hence  $f$ attains its maximum on $e$ at the vertices of $e$, which clearly implies \eqref{eq:sup=sup}.

(ii) Now let $v \in \cV$ be a vertex with $f(v) > 0$. By  \eqref{eq:kirchhoff}, there is an edge $e \in \cE_{v}$ such that $f_{e}'(v) \geq 0$. If $u \in \cV$ is the other vertex of $e$, then by \eqref{eq:finN} we get
\[
	f (u) = f(v) \cosh\big( \sqrt{-\lambda} |e| \big) + \frac{f_e'(v)}{\sqrt{-\lambda}} \sinh\big(\sqrt{-\lambda} |e|\big) > f(v).
\]
Observe that $f_e'(u)< 0$. Hence, setting $v_0=v$ and $v_1=u$ and using induction, we can construct a ray $\cR = (v_n) $ such that $f(v_{n+1} ) > f(v_n)$ for all $n\ge0$. Since $f\in H^1(\cG)$, we get
\[
	0<f(v) < \lim_{n \to \infty} f(v_n)  = f(\gamma_\cR) \le \sup_{\gamma \in \gC(\cG)} f(\gamma),
\]
which proves \eqref{eq:maxprinc01}.

(iii) By considering $\pm f$ (and splitting into real and imaginary part, if necessary), \eqref{eq:maxprinc02} clearly follows from \eqref{eq:maxprinc01}. 
\end{proof}

\begin{remark}\label{rem:4.7}
Notice that the arguments used in the proof of Lemma~\ref{lem:maxprinciple}(ii) in fact show that functions in $\cN_{\lambda}(\bH_0)$ with $\lambda\in(-\infty,0)$ admitting positive values on $\cG$ cannot attain global maxima in $\cG$, that is, if $f$ attains a positive value at some $x\in \cG$, then for every compact subgraph $\wt\cG\subset \cG$ the following holds
\[
\sup_{x \in \cG} f(x) = \sup_{x \in \cG\setminus\wt\cG} f(x).
\]
 Clearly, analogous statements hold true for functions admitting negative values, however, then $\sup$ must be replaced with $\inf$.
\end{remark}

\begin{lemma}\label{lem:NcapH1}
Suppose $\cG $ is a metric graph and let $\lambda \in(-\infty,0)$. Then 
\be\label{eq:NcapH1}
\dim(\cN_\lambda(\bH_0) \cap H^1(\cG)) =  \# \gC_0(\cG).
\ee
\end{lemma}

\begin{proof}
Using \eqref{eq:dombHsum01} with $z = \lambda \in(-\infty,0)$ and noting that $\dom(\bH_F)\subset H^1_0(\cG)$, Theorem~\ref{thm:nontr=finitevol} and Theorem~\ref{th:H10} imply that $\dim(\cN_\lambda(\bH_0) \cap H^1(\cG)) \ge  \# \gC_0(\cG) $. The converse inequality follows from Lemma~\ref{lem:maxprinciple}(iii), which shows that the mapping $f \mapsto (f(\gamma))_{\gamma \in \gC_0(\cG)}$ is injective on $\cN_\lambda(\bH_0) \cap H^1(\cG)$.
\end{proof}

After all these preparations, we are now in position to complete the proof of Theorem~\ref{th:indicesends}.

\begin{proof}[Proof of Theorem~\ref{th:indicesends}]
Observe that the inequality \eqref{eq:npm>ends} immediately follows from \eqref{eq:dombHsum01} and \eqref{eq:NcapH1} since $\Nr_\pm(\bH) = \dim(\cN_\lambda(\bH_0))$.

Clearly,  the second claim is trivial if $\# \gC_0(\cG) = \infty$. Hence it remains to show that in the case $\# \gC_0(\cG) < \infty$ equality \eqref{eq:npm=ends} holds exactly when $\dom(\bH)\subset H^1(\cG)$. Applying \eqref{eq:dombHsum01} once again, the inclusion $\dom(\bH)\subset H^1(\cG)$ holds true exactly when $\cN_\lambda(\bH_0) \subset H^1(\cG)$. Taking into account once again that $\Nr_\pm(\bH) = \dim \cN_\lambda(\bH_0)$ and using \eqref{eq:NcapH1}, we arrive at the conclusion.
\end{proof}

\begin{remark}
Let us mention that one can prove the second claim of Theorem~\ref{th:indicesends} in a different way. Namely, if $\# \gC_0(\cG) < \infty$, then it is possible to reduce the problem to the study of a finite volume graph with a single end. 
\end{remark}

Let us stress that in the proof of Theorem~\ref{th:indicesends} the equivalence of equality \eqref{eq:npm=ends} and the inclusion $\dom(\bH)\subset H^1(\cG)$ was proved in the case when all finite volume ends are free. The next result shows that the inclusion never holds if there is a finite volume end which is not free.

\begin{proposition} \label{prop:H1InfEnds}
Let $\cG$ be  a metric graph having a finite volume end which is not free. Then there exists a function $f \in \dom(\bH)$ which does not belong to $H^1(\cG)$. 
\end{proposition}

\begin{proof}
First observe that we can restrict our considerations to the case of a metric graph $\cG$ having finite total volume. Indeed, if $\gamma $ is a non-free finite volume end of $\cG$, then there exists a sequence $\cU = (U_n)$  representing $\gamma$ such that $\vol(U_n) < \infty$ for all $n$. By definition, each $U_n$ is open and has compact boundary. Choosing $\cG_0\subset \cG$ as the subgraph with vertex set $\cV(\cG_0) = \cV\cap U_0$ and edge set $\cE(\cG_0) = \{e\in \cE\,|\, e\subset U_0 \}$, it is easy to see that $\cG_0$ is a connected finite volume subgraph and $\gamma$ is a non-free end of $\cG_0$ (see also the notion of graph representation of an end in Section~\ref{ss:VI:i}). Moreover, by construction the set $\partial \cG_0$ of boundary points (here, $\cG_0$ is seen as a closed subset of $\cG$) is finite.

Let $\wt \cG \subset \cG$ be a connected, compact subgraph and consider the finitely many connected components of $\cG \setminus \wt \cG$. Since $\cG$ has infinitely many ends, there is a connected component $U$ which contains at least two distinct graph ends $\gamma, \gamma' \in \gC(\cG)$. 
Following the proof of Theorem~\ref{thm:nontr=finitevol}, we readily construct a real-valued function $f=f_U \in \dom(\bH) \cap H^1(\cG)$ with $f(\gamma) = 0$, $f(\gamma') = 1$ and $0 \le f \le 1$ on $\gC(\cG)$ (in fact, it suffices to choose the corresponding function $\phi$ with $0 \le \phi \le 1$). Taking into account Theorem~\ref{th:H10} and decomposition \eqref{eq:dombHsum01}, we can assume that $f$ belongs to $H^1(\cG) \cap \cN_\lambda(\bH_0)$ for some (fixed) $\lambda \in (- \infty, 0)$. However, Lemma~\ref{lem:maxprinciple} (iii) implies that
\[
	\| f \|_\infty = \sup_{x \in \cG} |f(x)| = \sup_{x \in \cG} f(x) = 1.
\]
On the other hand, there exist two rays $\cR$, $\cR' \in \mR(\cG_d)$ representing the ends $\gamma$ and, respectively, $\gamma'$ such that both $\cR$, $\cR'$ are contained in $U$ and have the same initial vertex $v_0$. This leads to another estimate  
\begin{align*}
1 &= \big |f(\gamma) - f(\gamma') \big |= \big |f(\gamma) - f(v_0) + f(v_0)  - f(\gamma') \big | \\
&= \Big |  \int_\cR f'(x)dx - \int_{\cR'} f'(x)dx \Big  | \le 2 \sqrt{\vol(U)}\, \|f'\|_{L^2(U)} \le 2 \sqrt{\vol(U)}\, \|f'\|_{L^2(\cG)}.
\end{align*}

Assume now that \eqref{eq:sobolev} holds for all functions $g \in \cN_\lambda(\bH_0)$. Then $\| \cdot \|_\infty$ and $\| \cdot \|_{H^1}$ are in fact equivalent norms on $\cN_\lambda(\bH_0)$. Indeed, combining \eqref{eq:sobolev} and the finite volume property, we get
\[
	\| g \|_{H^1}^2 \le C (\| g\|^2_{L^2} + \| \bH g \|_{L^2}^2) = C (1 + \lambda^2) \| g \|_{L^2}^2 \le C (1 + \lambda^2) \vol(\cG) \| g \|_{\infty}^2
\]
for all $g \in \cN_\lambda(\bH_0)$, whereas $\| g \|_{\infty} \le C_\cG \|g\|_{H^1}$ by Lemma~\ref{lem:inftyh1}.
Choosing compact subgraphs $\wt \cG_\varepsilon$ with $\vol(\cG \setminus \wt \cG_\varepsilon) \le \varepsilon^2$ (which is possible since $\cG$ has finite volume), we clearly get $\vol(U_\varepsilon) \le \varepsilon^2$ and hence the above constructed function $f_\varepsilon := f_{U_\varepsilon}\in H^1(\cG) \cap \cN_\lambda(\bH_0)$ satisfies
\[
\|f_\varepsilon'\|_{L^2(\cG)} \ge \|f_\varepsilon'\|_{L^2(U_\varepsilon)} \ge \frac{1}{2 \sqrt{\vol(U_\varepsilon)}} \ge \frac{1}{2 \varepsilon}.
\]
However, by construction, $\|f_\varepsilon\|_\infty = 1$, which obviously contradicts to the equivalence of norms $\|\cdot\|_\infty$ and $\|\cdot\|_{H^1}$ on $\cN_\lambda(\bH_0)$ since $\varepsilon >0$ is arbitrary. 
\end{proof}

We conclude this section by mentioning some explicit examples. 

\begin{example}[Radially symmetric trees] \label{ex:rst}
Let $\cG= \cT$ be a \emph{radially symmetric (metric) tree}: that is, a rooted tree $\cT$ such that for each $n \ge 0$, all vertices in the combinatorial sphere $S_n$ have the same number of descendants $d_n \ge 2$ and all edges between the combinatorial spheres $S_n$ and $S_{n+1}$ have the same length.
It is well-known that in this case $\bH$ is self-adjoint if and only if $\vol(\cT) = \infty$ and deficiency indices are infinite, $\Nr_\pm(\bH_0) = \infty$, otherwise (see, e.g.,~\cite{car00, sol04}). Moreover, due to the symmetry assumptions, all graph ends are of finite volume simultaneously. Hence we arrive at the equality
\[
	\Nr_\pm(\bH_0) = \# \gC_0(\cG) = \begin{cases}\infty, &\text{if } \vol(\cT) < \infty, \\ 0, & \text{if } \vol(\cT) = \infty. \end{cases}
\]
Moreover, by Theorem~\ref{th:indicesends} and Proposition~\ref{prop:H1InfEnds}, the inclusion $\dom(\bH) \subset H^1(\cG)$ holds true if and only if $\vol(\cT) = \infty$.
\end{example}

\begin{example}[Radially symmetric antitrees] \label{ex:rsat}
Consider a metric antitree $\cG = \cA$ (see Section~\ref{ss:AT01} for definitions) and additionally suppose that $\cA$ is \emph{radially symmetric}, that is, for each $n\ge 0$, all edges between the combinatorial spheres $S_n$ and $S_{n+1}$ have the same length. Combining~\cite[Theorem~4.1]{kn19} (see also Corollary~\ref{lem:ATrad} below) with the fact that antitrees have exactly one graph end, $\# \gC(\cA) = 1$, we conclude that
\[
	\Nr_\pm(\bH_0) = \# \gC_0(\cG) = \begin{cases}1, &\text{if } \vol(\cA) < \infty, \\ 0, & \text{if } \vol(\cA) = \infty. \end{cases}
\]
In particular, $\bH$ is self-adjoint if and only if $\vol(\cA) = \infty$. Moreover, the inclusion $\dom(\bH) \subset H^1(\cG)$ holds true for all radially symmetric antitrees by Theorem~\ref{th:indicesends}.
\end{example}

\begin{remark}\label{rem:fampreserv}
Both radially symmetric trees and antitrees are particular examples of the so-called \emph{family preserving metric graphs} (see~\cite{bl19} and also~\cite{brke13}) . 
 Employing the results from~\cite{bl19}, it is in fact possible to extend the conclusions in Example~\ref{ex:rst} and Example~\ref{ex:rsat} to this general setting. 
More precisely, for each \emph{family preserving metric graph $\cG$ without horizontal edges}, the Kirchhoff Laplacian $\bH$ is self-adjoint if and only if $\vol(\cG) = \infty$ and moreover
\[
	\Nr_\pm(\bH_0) =  \# \gC_0(\cG) = \begin{cases}\# \gC (\cG), &\text{if } \vol(\cG) < \infty, \\ 0, & \text{if }  \vol(\cG) = \infty. \end{cases}
\]
If in addition $\cG$ has finitely many ends, then the inclusion $\dom(\bH) \subset H^1(\cG)$ holds true. On the other hand, if $\cG$ has infinitely many ends, then $\dom(\bH) \subset H^1(\cG)$ holds true if and only if $\vol(\cG) = \infty$. The last two statements are again immediate consequences of Theorem~\ref{th:indicesends} and Proposition~\ref{prop:H1InfEnds}.

In conclusion, let us also emphasize that the example of the rope ladder graph in Appendix~\ref{sec:rope} shows that the assumption on horizontal edges cannot be omitted. More precisely, the rope ladder graph is a \emph{family preserving graph} in the sense of~\cite{brke13} with exactly one graph end. However, it possesses infinitely many horizontal edges (i.e., edges connecting vertices in the same combinatorial sphere) and Example~\ref{ex:polynomialrl} shows that in general $\Nr_\pm(\bH_0) > \#\gC_0(\cG)$, even if the edge lengths are chosen symmetrically to the root, $|e_n^+| = |e_n^-|$ for all $n \in \Z_{\ge 0}$.
\end{remark}

%%%%%%%%%%%%%%%%%%%%%%%%%%%%%%%%%%%%%%%%%%%%%%%%%%%%%%%%%%%%
%%%%%%%%%%%%%%%%%%%%%%%%%%%%%%%%%%%%%%%%%%%%%%%%%%%%%%%%%%%%
\section{Properties of self-adjoint extensions} \label{sec:V}
%%%%%%%%%%%%%%%%%%%%%%%%%%%%%%%%%%%%%%%%%%%%%%%%%%%%%%%%%%%%
%%%%%%%%%%%%%%%%%%%%%%%%%%%%%%%%%%%%%%%%%%%%%%%%%%%%%%%%%%%%

The Sobolev space $H^1(\cG)$ plays a distinctive role in the study of self-adjoint extensions of the minimal operator $\bH_0$. A self-adjoint extension $\wt\bH$ of $\bH_0$ is called a \emph{finite energy extension} if its domain is contained in $H^1(\cG)$, that is, every function $f\in \dom(\wt\bH)$ has finite energy, $\|f'\|_{L^2(\cG)} <\infty$. 
The main result of this section already indicates that finite energy self-adjoint extensions of the minimal operator (notice that among those are the Friedrichs extension and, as we will see later in this section, all Markovian extensions) possess a number of important properties. 

\begin{theorem}\label{prop:resolvprop}
 Let $\widetilde{\bH}$ be a self-adjoint lower semi-bounded extension of ${\bH}_0$. Assume that $z$ belongs to its resolvent set $\rho(\widetilde{\bH})$. Then the following assertions hold.
\begin{enumerate}[(i)]
\item If the form domain of $\widetilde{\bH}$ is contained in $H^1(\cG)$, then the resolvent $\cR(z,\widetilde{\bH})$ of $\widetilde{\bH}$ is an integral operator whose kernel $\cK_z$ is both of class $L^\infty(\cG\times \cG)$ and jointly H\"older continuous of exponent $\beta = 1/2$.
\item If additionally $\cG$ has finite total volume, then 
$\cR(z,\widetilde{\bH})$ is of trace class.
\end{enumerate}
\end{theorem}

\begin{proof}
(i) Let $\widetilde{\bH}$ be a self-adjoint lower semi-bounded extension of ${\bH}_0$, $\wt\bH \ge c$ for some $c \in \R$. Without loss of generality we may assume $c=0$. Then we can consider its positive semi-definite square root $\widetilde{\bH}^{1/2}$, which is again self-adjoint and whose domain agrees with the form domain of $\widetilde{\bH}$. Accordingly, for all $z\in \C\setminus [0,\infty)$ and $\lambda = \sqrt{z}$ we get
\[
\big(\widetilde{\bH}^{1/2}  - \lambda\big)\big(\widetilde{\bH}^{1/2} + \lambda\big)=\widetilde{\bH} - z,
\]
and hence
\begin{eqnarray}\label{eq:factorz}
\cR(z,\widetilde{\bH})=\cR(\lambda,\widetilde{\bH}^{1/2})\cR(-\lambda,\widetilde{\bH}^{1/2}).
\end{eqnarray}
If the form domain of $\widetilde{\bH}$ is contained in $H^1(\cG)$, and hence by Lemma~\ref{lem:inftyh1} in $C_b(\cG)$, then $\cR(\pm \lambda,\widetilde{\bH}^{1/2})$ maps $L^2(\cG)$ into $L^\infty(\cG)$, and  hence by duality also maps $L^1(\cG)$ into $L^2(\cG)$. Thus \eqref{eq:factorz} implies that $\cR(z,\widetilde{\bH})$ maps $L^1(\cG)$ into $L^\infty(\cG)$ and hence, by the Kantorovich--Vulikh theorem (see, e.g.,~\cite[Theorem~1.3]{Arebuh94} or~\cite[Theorem~1.1]{MugNit11}), $\cR(z,\widetilde{\bH})$ is an integral operator with the $L^\infty$-kernel  $\cK(z;\cdot,\cdot)$. 

In order to prove the assertion about joint H\"older continuity, we need to take a closer look at the kernel $\cK$ by adapting the proof of~\cite[Prop.~2.1]{AreEls19}:
as noticed before, the resolvent $\cR(\lambda,\widetilde{\bH}^{1/2})$ is bounded from $L^2(\cG)$ to $L^\infty(\cG)$ by Lemma~\ref{lem:inftyh1} for any $\lambda$ in the resolvent set of $\widetilde{\bH}^{1/2}$. Applying the Kantorovich--Vulikh theorem (see, e.g.,~\cite[page 113]{Arebuh94}) once again, we see that
\[
\cR( \lambda,\widetilde{\bH}^{1/2})u(x)=\int_\cG u(y) \kappa(\lambda;x,y) dy = \langle u ,  \kappa(\lambda;x,\cdot)^\ast \rangle_{L^2(\cG)}
\]
for all $x\in \cG$ and some $\kappa(\lambda;x,\cdot) \in L^2(\cG)$ such that $\sup_{x \in \cG} \| \kappa(\lambda;x, \cdot) \|_{L^2(\cG)} < \infty$. Moreover, observe that there exists $C = C(\lambda) > 0$ such that
\begin{align} \label{eq:jointcont}
 \| \kappa(\lambda;x,\cdot) - \kappa(\lambda;x',\cdot) \|_{L^2(\cG)} \le C \sqrt {\varrhoo (x, x') }
\end{align}
for all $x, x' \in \cG$, where $\varrhoo (x,x')$ denotes the distance in the natural path metric on $\cG$. Indeed, for any function $u \in L^2(\cG)$,
\begin{align}\label{eq:bddh1l2}
\begin{split}
	\Big| \int_\cG u(y)   (\kappa(\lambda;x,y) - \kappa(\lambda;x',y))    dy \Big|  & =  
	\big| \cR(\lambda, \widetilde{\bH}^{1/2}  )u(x) -  \cR(\lambda, \widetilde{\bH}^{1/2} ) u(x')\big| \\
	&\le \sqrt{\varrhoo (x, x')} \| \cR(\lambda,\widetilde{\bH}^{1/2} ) u \|_{H^1}\\
	& \le C \sqrt{\varrhoo (x, x')} \| u \|_{L^2},
	\end{split}
\end{align}
where we have used the Cauchy--Schwarz inequality and the fact that the resolvent $\cR(\lambda, \wt{\bH}^{1/2})$ is a bounded operator from $L^2$ to the domain of $\wt{\bH}^{1/2}$ equipped with the graph norm, and \eqref{eq:jointcont} immediately follows. Now, taking into account the equalities \eqref{eq:factorz} and $\cR( \lambda,\widetilde{\bH}^{1/2})^\ast = \cR( \lambda^\ast,\widetilde{\bH}^{1/2})$, we conclude that
\begin{align*}
\cR(z,\widetilde{\bH})u(x)&=\cR( \lambda,\widetilde{\bH}^{1/2})\big(\cR( -\lambda,\widetilde{\bH}^{1/2})u\big)(x) \\
&= \big\langle \cR( -\lambda,\widetilde{\bH}^{1/2})u,  \kappa(\lambda;x,\cdot)^\ast\big\rangle_{L^2(\cG)} \\
&= \big\langle u,  \cR( -\lambda^\ast,\widetilde{\bH}^{1/2})\kappa(\lambda;x,\cdot)^\ast\big\rangle_{L^2(\cG)} \\
&=\int_\cG u(y)\int_\cG\kappa(\lambda;x,s)\kappa(-\lambda^\ast;y,s)^\ast ds\, dy\\
&=:\int_\cG u(y)\cK(z;x,y)\, dy,
\end{align*}
 for all $u\in L^2(\cG)$.
 It remains to prove that the mapping
\[
	\cK:\cG\times \cG\ni (x,y)\mapsto \int_\cG\kappa(\lambda;x,s)\kappa(-\lambda^\ast;y,s)^\ast ds \in \C
\]
is jointly H\"older continuous. However, recalling that $\sup_{x \in \cG} \| \kappa(\lambda, x; \cdot) \|_{L^2(\cG)} < \infty$, this immediately follows from \eqref{eq:jointcont}, since
\begin{align*}
	|\cK(x, y) - \cK(x', y')| & \le \|\kappa(\lambda;x,\cdot) (\kappa(-\lambda^\ast;y,\cdot)^\ast - \kappa(-\lambda^\ast;y',\cdot)^\ast) \|_{L^1} \\
	&\quad + \| \kappa(-\lambda^\ast;y',\cdot)^\ast (\kappa(\lambda;x,\cdot) - \kappa(\lambda;x',\cdot) ) \|_{L^1}.
\end{align*}
for all pairs $(x, y), (x',  y') \in \cG \times \cG$.
 
(ii) If $\cG$ has finite total volume, then $L^\infty(\cG\times \cG)\hookrightarrow L^2(\cG\times \cG)$ and hence the resolvents $\cR(\pm \lambda,\widetilde{\bH}^{1/2})$ are Hilbert--Schmidt operators. Thus,  by~\eqref{eq:factorz} we conclude that $\cR(z,\widetilde{\bH})$ is of trace class.
\end{proof}

Observe that the first step in the proof of Theorem~\ref{prop:resolvprop} is the factorization~\eqref{eq:factorz}, which has the natural counterpart for semigroups
\[
\E^{-z\widetilde{\bH}}\,\E^{-z\widetilde{\bH}}=\E^{-2z\widetilde{\bH}},\qquad \re z>0.
\]
Because the semigroup generated by a self-adjoint semi-bounded extension $\widetilde{\bH}$ is analytic, it is a bounded operator from the Hilbert space into its generator's form domain whenever $\re z>0$. A careful look at the proof of Theorem~\ref{prop:resolvprop} shows that this is sufficient to establish that $\E^{-z\widetilde{\bH}}$ is an integral operator;  all further steps in the proof of Theorem~\ref{prop:resolvprop}  carry over almost verbatim to the study of semigroups. We can hence easily deduce the following result.
\begin{theorem}\label{prop:semigprop}
 Let $\widetilde{\bH}$ be a self-adjoint lower semi-bounded extension of ${\bH}_0$ and let $z\in \C$ with $\re z>0$. Then the following assertions hold.
\begin{enumerate}[(i)]
\item If the domain of $\widetilde{\bH}$ is contained in $H^1(\cG)$, then the semigroup $\E^{-z\widetilde{\bH}}$ generated by $\widetilde{\bH}$ is an integral operator whose kernel is both of class $L^\infty(\cG\times \cG)$ and jointly H\"older continuous of exponent $\beta = 1/2$. 
\item If additionally $\cG$ has finite total volume, then 
$\E^{-z\widetilde{\bH}}$ is of trace class.
\end{enumerate}
\end{theorem}

Estimating as in~\eqref{eq:bddh1l2} and using analyticity of $\E^{-z\widetilde{\bH}}$ yields the inequality 
\begin{align} \label{eq:jointcont-2}
 |p_t(x,y) - p_t(x',y) | \le \frac{C}{\sqrt{t}} \sqrt {\varrhoo (x,x')},\qquad t> 0,\ x,y,x'\in \cG,
\end{align}
for the heat kernel $p_t(x,y)$ of a nonnegative extension $\widetilde{\bH}$, where in contrast to~\eqref{eq:bddh1l2} the constant $C>0$ is independent of $t>0$. 
 Such H\"older estimates are known to be related to Sobolev-type inequalities and also important for upper and lower Gaussian bounds~(cf., e.g.,~\cite{Cou03}, \cite[Chapter~6]{ouh}). However, we do not pursue this line of study here and this will be done elsewhere.

\begin{remark}
A few remarks are in order.
\begin{itemize}
\item[(i)] If $\sup_\cR |\cR| <\infty$, where the supremum is taken over all non-vertex paths without self-intersections, then the path metric $\varrho$ has a natural extension $\widehat \varrho$ to the end compactification $\widehat \cG$. Moreover, in this case $(\widehat \cG, \widehat \varrho)$ coincides with the metric completion of $(\cG, \varrho)$. Indeed, the metric completion of $(\cG, \varrho)$ is obtained by adding to $\cG$ equivalence classes of rays of finite length (see~\cite[Section~2.3]{g11} for details) and
 the distance of $x \in \cG$ to a boundary point is defined as the ``shortest length" of a path in the corresponding equivalence class starting at $x$. 
 
Therefore, Theorem~\ref{prop:resolvprop} and Theorem~\ref{prop:semigprop} imply that in this case the corresponding resolvent and semigroup kernels have a bounded and uniformly continuous extension to $(\widehat \cG, \widehat \varrho)$. However, we stress that in contrast to the case $\vol(\cG) < \infty$ (see Remark~\ref{rem:finvol}), the topology generated by $\widehat \varrho$ on $\widehat \cG$ can differ from the end compactification topology if $\vol(\cG) =\infty$. 
\item[(ii)]
Discreteness of the spectrum of the Friedrichs extension $\bH_F$ is a standard fact in the case of finite total volume (see, e.g.,~\cite[Prop.~3.11]{car08} or~\cite[Corollary~3.5(iv)]{kn19}). However, Theorem~\ref{prop:resolvprop}(ii) implies  the stronger assertion that the resolvent of $\bH_F$ belongs to the trace class if $\vol(\cG)<\infty$. Let us also stress that it is not true in general that every self-adjoint extension of $\bH$ will have a discrete spectrum if $\vol(\cG)<\infty$, since in case of infinite deficiency indices such a self-adjoint extension could have a domain large enough to make  compactness of the embedding of $H^1(\cG)$ into $L^2(\cG)$ irrelevant.
\end{itemize}
\end{remark}

Recall that a self-adjoint extension $\wt\bH$ of $\bH_0$ is called \emph{Markovian} if $\wt\bH$ is a non-negative self-adjoint extension and the corresponding quadratic form is a \emph{Dirichlet form} (for definitions and further details we refer to~\cite[Chapter~1]{fuk10}). 
 Hence the associated semigroup $\E^{-t\wt\bH}$, $t>0$, as well as resolvents $\cR(-\lambda,\wt\bH)$, $\lambda>0$, are Markovian: i.e., are both \emph{positivity preserving} (map non-negative functions to non-negative functions) and \emph{$L^\infty$-contractive} (map the unit ball of $L^\infty(\mathcal G)$, and then by duality of $L^p(\mathcal G)$ for all $p\in[1,\infty]$, into itself). Let us stress that the Friedrichs extension $\bH_F$ of $\bH_0$ is a Markovian extension. Consider also the following quadratic form in $L^2(\cG)$
\begin{align} \label{eq:NMform}
 	\gt_N[f] = \int_{\cG} |f'(x)|^2 dx, \qquad \dom(\gt_N) = H^1(\cG).
\end{align}
This form is non-negative and closed, hence we can associate in $L^2(\cG)$ a self-adjoint operator with it, let us denote it by $\bH_N$. We will refer to it as the \emph{Neumann extension}\label{p:NeumannExt}. It is straightforward to check that $\gt_N$ is a Dirichlet form and $\bH_N$ is also a Markovian extension of $\bH_0$. 

It turns out that Theorems~\ref{prop:resolvprop} and~\ref{prop:semigprop} apply to all Markovian extensions of $\bH_0$. More specifically, the analog of the results for discrete Laplacians~\cite[Theorem~5.2]{hklw} and Laplacians in Euclidean domains~\cite[Chapter~3]{fuk10} and Riemannian manifolds~\cite[Theorem~1.7]{gm} holds true for quantum graphs as well. 

\begin{theorem}\label{th:Markovian1}
If $\wt\bH$ is a Markovian extension of $\bH_0$, then $\dom(\wt\bH)\subset H^1(\cG)$ and, moreover, 
\be
\bH_N \le \wt\bH \le \bH_F,
\ee
where the inequalities are understood in the sense of forms.\footnote{We shall write $A\le B$ for two non-negative self-adjoint operators $A$ and $B$ if their quadratic forms $\gt_A$ and $\gt_B$ satisfy $\dom(\gt_B)\subseteq\dom(\gt_A)$ and $\gt_A[f]\le \gt_B[f]$ for every $f\in \dom(\gt_B)$.}
\end{theorem}

We omit the proof of Theorem~\ref{th:Markovian1} since the proofs of either~\cite[Theorem~5.2]{hklw} or~\cite[Lemma~3.6]{gm} carry over verbatim to our setting (see also the proof of~\cite[Theorem~3.3.1]{fuk10}).

Let us finish this section with the following observation. 

\begin{corollary}\label{cor:Markovian1}
The following are equivalent:
\begin{itemize}
\item[(i)] $\bH_0$ has a unique Markovian extension,
\item[(ii)] $H^1_0(\cG) = H^1(\cG)$,
\item[(iii)] all topological ends of $\cG$ have infinite volume, $\gC_0(\cG) = \emptyset$.
\end{itemize}
\end{corollary}

\begin{proof}
The claimed equivalences follow from Theorem~\ref{th:Markovian1} and Corollary~\ref{cor:H=H}.
\end{proof}

\begin{remark}\label{rem:H=H=uniq}
Let us finish this section with a few comments.
\begin{itemize}
\item[(i)]
The equivalence $(i)\Leftrightarrow (ii)$ in Corollary~\ref{cor:Markovian1} is known for Riemannian manifolds~\cite[Theorem~1.7]{gm} (see also~\cite[Chapter~3]{fuk10},~\cite[Theorem~1]{ma05}) as well as for weighted Laplacians on graphs~\cite[Corollary~5.6]{hklw}. However, to the best of our knowledge these settings do not admit any further geometric characterization.
\item[(ii)] 
 The list of equivalences in Corollary~\ref{cor:Markovian1} can be extended by adding a claim on the self-adjointness of the so-called \emph{Gaffney Laplacian}. Namely, since $H^1_0(\cG)$ and $H^1(\cG)$ are Hilbert spaces, the operators denoted by $\nabla_D$ and $\nabla_N$ and defined in $L^2(\cG)$ on the domains, respectively, $H^1_0(\cG)$ and $H^1(\cG)$ by $f\mapsto f'$ are closed. Notice that with this notation at hand we have $\bH_F = \nabla_D^\ast \nabla_D$ and $\bH_N = \nabla_N^\ast \nabla_N$. Now we can introduce the Gaffney Laplacian $\bH_G:= \nabla_D^\ast\nabla_N$ as the restriction of the maximal operator $\bH$ onto the domain (compare with~\cite[p. 610]{gm} for the definition in the manifolds case)
\begin{align}
\dom(\bH_G) := \{f\in H^1(\cG)|\, \nabla_N f\in \dom(\nabla_D^\ast)\}.
\end{align}
Clearly, $\dom(\bH_F) \subseteq \dom(\bH_G)$, $\dom( \bH_N) \subseteq \dom(\bH_G)$, and $\bH_G$ is not necessarily symmetric. It turns out that $\bH_G$ is symmetric (and hence self-adjoint) if and only if the Kirchhoff Laplacian $\bH_0$ has a unique Markovian extension. Moreover, in this case $\bH_F = \bH_N = \bH_G$ (cf.~\cite[Theorem~1.7(ii)]{gm} in the manifold setting). Let us mention that the Markovian/finite energy extensions of $\bH_0$ are exactly the Markovian/self-adjoint restrictions of $\bH_G$ and hence the deficiency indices of $\bH_G^\ast = \nabla_N^\ast\nabla_D$ are equal to $\#\gC_0(\cG)$. \end{itemize}
\end{remark}

%%%%%%%%%%%%%%%%%%%%%%%%%%%%%%%%%%%%%%%%%%%%%%%%%%%%%%%%%%%%
%%%%%%%%%%%%%%%%%%%%%%%%%%%%%%%%%%%%%%%%%%%%%%%%%%%%%%%%%%%%
\section{Finite energy self-adjoint extensions} \label{sec:VI}
%%%%%%%%%%%%%%%%%%%%%%%%%%%%%%%%%%%%%%%%%%%%%%%%%%%%%%%%%%%%
%%%%%%%%%%%%%%%%%%%%%%%%%%%%%%%%%%%%%%%%%%%%%%%%%%%%%%%%%%%%
It turns out that finite volume (topological) ends  provide the right notion of the boundary for metric graphs to deal with finite energy and also with Markovian extensions of the minimal Kirchhoff Laplacian $\bH_0$. In particular, we are going to show that this end space is well-behaved as concerns the introduction of both traces and normal derivatives.
More specifically, the goal of this section is to give a description of finite energy self-adjoint extensions of $\bH_0$ in the case when the number of finite volume ends of $\cG$ is finite, that is, $\#\gC_0(\cG) < \infty$.   
Notice that  in this case all finite volume ends are free. 

\subsection{Normal derivatives at graph ends} \label{ss:VI:i}
%%%%%%%%%%%%%%%%%%%%%%%%%%%%%%%%%%%%%%%%%%%%%%%%%%%%%%%%%%%%

Let $\wt \cG = (\wt\cV,\wt\cE)$ be a (possibly infinite) connected subgraph of $\cG$.
 Recall that 
  its boundary $\partial\wt \cG$ (w.r.t.\ the natural topology on $\cG$, see Section~\ref{ss:II.01}) is given by 
\begin{align} \label{eq:bdsubgraph}
\partial \wt\cG = \big\{ v \in \wt \cV | \; \deg_{\wt \cG} (v) < \deg_\cG (v) \big\}.
\end{align}
For a function $f \in \dom (\bH)$, we define its (inward) \emph{normal derivative} at  $v \in \partial \wt \cG$ by 
\begin{align} \label{eq:Dervsubgraph}
	\frac{\partial f}{\partial n_{\wt \cG}} (v) :=  \sum_{e \in \cE_v \cap \wt \cE} f_e'(v).
\end{align}
With this definition at hand, we end up with the following useful integration by parts formula.

\begin{lemma} \label{lem:ibp}
Let $\wt \cG$ be a compact (not necessarily connected) subgraph of the metric graph $\cG$. Then
\be\label{eq:IntByPart}
- \int_{\wt \cG} f''(x) g(x) dx = \int_{\wt \cG}  f'(x) g'(x) dx + \sum_{v \in \partial \wt \cG} g(v) \frac{\partial f}{\partial n_{\wt \cG}} (v)  
\ee
for all $f \in \dom(\bH)$ and $g \in H^1(\wt \cG)$. In particular, 
\be\label{eq:IntByPart2}
- \int_{\wt \cG} f''(x)  dx =  \sum_{v \in \partial \wt \cG} \frac{\partial f}{\partial n_{\wt \cG}} (v). 
\ee
\end{lemma}

\begin{proof}
The claim follows immediately from integrating by parts, taking into account that $f$ satisfies \eqref{eq:kirchhoff}. Setting $g\equiv 1$ in \eqref{eq:IntByPart}, we arrive at \eqref{eq:IntByPart2}.
\end{proof}

In order to simplify our considerations, we need to introduce the following notion.
Let $\gamma \in \gC(\cG)$ be a (topological) end of $\cG$.  Consider a sequence $(\cG_n)$ of connected subgraphs of $\cG$ such that $\cG_{n} \supseteq \cG_{n+1}$ and $\# \partial  \cG_n < \infty$ for all $n$. We say that the sequence $(\cG_n)$ is a \emph{graph representation of the end } $\gamma \in \gC(\cG)$ if there is a sequence of open sets $\cU = (U_n)$  representing $\gamma$ such that for each $n \ge 0$ there exist $j$ and $k$ such that $ \cG_n \supseteq U_j$ and $U_n \supseteq  \cG_k$.  It is easily seen that all graphs $\cG_n$ are infinite (they have infinitely many edges). Moreover, graph representations $(\cG_n)$ of an end can be constructed with the help of compact exhaustions; in particular each graph end $\gamma \in \gC (\cG)$ has a representation by subgraphs (see Section~\ref{ss:II.02}). 

\begin{proposition} \label{prop:NMLim}
Let $\cG$ be a metric graph and let $\gamma \in \gC(\cG)$ be a free end of finite volume. Then for every function $f \in \dom(\bH)$ and any sequence $(\cG_k)$ of subgraphs representing $\gamma$, the  limit
\begin{align} \label{eq:NMLim}
	\lim_{k \to \infty} \sum_{v \in \partial \cG_k} \frac{\partial f}{\partial n_{\cG_k}} (v)
\end{align}
exists and is independent of the choice of $(\cG_k)$.
\end{proposition}

\begin{proof}
First of all, notice that uniqueness of the limit follows from the inclusion property in the definition of the graph representations of $\gamma$. Hence we only need to show that the limit in \eqref{eq:NMLim} indeed exists.

Let $(\cG_k)$ be a graph representation of a free finite volume end $\gamma\in \gC_0(\cG)$. Since $\gamma$ is free, we can assume that $\vol(\cG_0) < \infty$ and that $\cG_0 \cap U_k = \varnothing$ eventually for every sequence $\cU = (U_k) $ representing an end $\gamma' \neq \gamma$.
First observe that $\wt \cG= \cG_k  \setminus \cG_j$ can again be identified with a compact subgraph of $\cG$ whenever $k\le j$. Indeed, if $\wt \cG$ has infinitely many edges $\{e_n\} \subset \cE$, choose for each $n$ a point $x_n$ in the interior of the edge $e_n$. Since $\wh \cG = \cG \cup \gC(\cG)$ is compact, the set $\{x_n\}$ has an accumulation point $x \in \wh \cG$. By construction, $x \notin \cG$ and hence $x \in \wh \cG \setminus  \cG = \gC(\cG)$ is an end. However, we have that $x_n \notin { \cG_j}$ and recalling \eqref{eq:top1} and \eqref{eq:top2}, this implies that $x = \gamma'$ for a topological end $\gamma' \neq \gamma$. On the other hand, $x_n \in \cG_0 $ for all $n$ and using the properties of $\cG_0$ and \eqref{eq:top1}--\eqref{eq:top2} once again, we arrive at a contradiction.

Now, using \eqref {eq:bdsubgraph} it is straightforward to verify that
\[
	 \sum_{v \in \partial \cG_k} \frac{\partial f}{\partial n_{\cG_k}} (v) -  \sum_{v \in \partial \cG_j} \frac{\partial f}{\partial n_{\cG_j}} (v) = \sum_{v \in \partial \wt\cG} \frac{\partial f}{\partial n_{\wt\cG}} (v).
\]
Hence by \eqref{eq:IntByPart2} and the Cauchy--Schwarz inequality, we get
\begin{align} \label{eq:NMLimproof}
	\Big|\sum_{v \in \partial \cG_k} \frac{\partial f}{\partial n_{\cG_k}} (v) -  \sum_{v \in \partial \cG_j} \frac{\partial f}{\partial n_{\cG_j}} (v)\Big| = \Big|\int_{\cG_k  \setminus \cG_j} f''(x) dx\Big| \le \sqrt{\vol(\cG_k)}\, \| \bH f\|_{L^2(\cG)},
\end{align}
whenever $k \le j$. This implies the existence of the limit in \eqref{eq:NMLim} since $\vol(\cG_k) = o(1)$ as $k\to \infty$.
\end{proof}

Proposition~\ref{prop:NMLim} now enables us to introduce a normal derivative at graph ends.

\begin{definition} \label{def:NMEnd}
Let $\gamma \in \gC(\cG)$ be a free end of finite volume and let $(\cG_k)$ be a graph representation of $\gamma$. Then for every $f\in\dom(\bH)$
\begin{align} \label{eq:NMEnd}
		\partial_n f  (\gamma) := \frac{\partial f}{\partial n} (\gamma) := 
		\lim_{k \to \infty} \sum_{v \in \partial \cG_k} \frac{\partial f}{\partial n_{\cG_k}} (v)
\end{align}
is called the \emph{normal derivative} of $f$ at $\gamma$.
\end{definition}

\begin{remark}
In fact, it is not difficult to extend the definitions \eqref{eq:Dervsubgraph} and \eqref{eq:NMEnd} to general sequences $\cU = (U_n)$ of open sets representing the free end $\gamma \in \gC_0(\cG)$. However, while the idea of the proof of Proposition~\ref{prop:NMLim} naturally carries over, the analysis becomes more technical and we restrict to the case of subgraphs for the sake of a clear exposition. \end{remark}

Let us mention that the normal derivative can also be expressed in terms of compact exhaustions. 

\begin{lemma} \label{cor:coex}
Let  $\cG$ be a metric graph having finite total volume and only one end $\gamma$, $\gC(\cG)= \{ \gamma \}$. If $(\mathcal{F}_k)$ is a compact exhaustion of $\cG$ and $f \in \dom(\bH)$, then
\be \label{eq:DerExh}
	 \partial_n f (\gamma) =  - \lim_{k \to \infty}  \sum_{v \in \partial \mathcal{F}_k} \frac{\partial f}{\partial n_{\mathcal{F}_k}} (v). 
\ee
\end{lemma}

The fact that we are not approximating $\gamma$ by its neighborhoods, but rather  by compact subgraphs, is responsible for the different sign in~\eqref{eq:NMEnd} and~\eqref{eq:DerExh}.

\begin{proof}
First of all, notice that $\cG \setminus \mathcal{F}_k$ can be identified with a subgraph of $\cG$ and 
\begin{align*}
- \sum_{v \in \partial \mathcal{F}_k} \frac{\partial f}{\partial n_{\mathcal{F}_k}} (v) = \sum_{v \in \partial(\cG \setminus \mathcal{F}_k)} \frac{\partial f}{\partial n_{\cG \setminus \mathcal{F}_k}} (v)
\end{align*}
for all $f \in \dom(\bH)$. If, moreover, $\cG \setminus \mathcal{F}_k$ is a connected subgraph for all $k\ge 0$, then it is clear that $(\cG_k)$ with $\cG_k:=\cG \setminus \mathcal{F}_k$ for all $k\ge0$, is a graph representation of $\gamma$ and this proves \eqref{eq:DerExh} in this case.

If $\cG \setminus \mathcal{F}_k$ is not connected, then it has only one infinite connected component $\cG_k^\gamma$ and finitely many compact components (since $\gC(\cG) = \{\gamma\}$). Adding these compact components to $\mathcal{F}_k$, we obtain a compact exhaustion $({\mathcal{F}}_k')$ with $\cG \setminus {\mathcal{F}}_k' = \cG_k^\gamma$. Arguing as in the proof of Proposition~\ref{prop:NMLim} (see \eqref{eq:NMLimproof}), we get
\begin{align*}
\Big|\sum_{v \in \partial {\mathcal{F}}_k'} \frac{\partial f}{\partial n_{{\mathcal{F}}_k'}} (v) -  \sum_{v \in \partial{\mathcal{F}}_k} \frac{\partial f}{\partial n_{{\mathcal{F}}_k}} (v)\Big| = \Big|\int_{{\mathcal{F}}_k' \setminus {\mathcal{F}}_k} f''(x) dx\Big| = o(1)
\end{align*}
as $k \to \infty$. Hence \eqref{eq:DerExh} holds true also in the general case.
\end{proof}

%%%%%%%%%%%%%%%%%%%%%%%%%%%%%%%%
\subsection{Properties of the trace and normal derivatives}\label{sec:VI.02}
%%%%%%%%%%%%%%%%%%%%%%%%%%%%%%%%
In this section, we collect some basic properties of the trace maps. We shall adopt the following notation. Since we shall always assume throughout this section that $\# \gC_0(\cG) < \infty$, we set $\cH:=\ell^2(\gC_0(\cG))$, which can be further identified with $\C^{\#\gC_0(\cG)}$. Next, we introduce the maps
$\Gamma_0\colon H^1(\cG) \to \cH$ and $\Gamma_1\colon \dom(\bH)\cap H^1(\cG) \to \cH$ by
\begin{align}\label{def:Gamma}
\Gamma_0\colon & f\mapsto  \big(f(\gamma)\big)_{\gamma \in \gC_0 (\cG)}, & 
\Gamma_1\colon & f\mapsto \big(\partial_n f (\gamma)\big)_{\gamma \in \gC_0 (\cG)},
\end{align}
where the boundary values and normal derivative of $f$ are defined by \eqref{def:context} and \eqref{eq:NMEnd}, respectively.

\begin{proposition}\label{lem:GG}
Let $\cG$ be a metric graph with $\# \gC_0(\cG) < \infty$. Then:

\begin{enumerate}[(i)]
\item For every $\wh{f}\in \cH$, there exists $f\in \dom(\bH)\cap H^1(\cG)$ such that
\begin{align*}
\Gamma_0 f & = \wh{f}, & \Gamma_1f & = 0.
\end{align*}
\item Moreover, the Gauss--Green formula
\begin{align} \label{eq:GG}
	\langle\bH f, g \rangle_{L^2(\cG)} = \langle f', g'\rangle_{L^2(\cG)} - \langle\Gamma_1f,\Gamma_0g\rangle_\cH
\end{align}
holds true for every $f \in \dom(\bH) \cap H^1(\cG)$ and $g \in H^1(\cG)$.
\end{enumerate}
\end{proposition}

\begin{proof}
(i)  Since $\#\gC_0(\cG) < \infty$, each finite volume end $\gamma\in \gC_0(\cG)$ is free. For every $\gamma\in \gC_0(\cG)$, let $\cG^\gamma$ be a subgraph with the properties as in Remark~\ref{rem:freetop}. We can also assume that $\vol(\cG^\gamma) < \infty$. Following the proof of Theorem~\ref{thm:nontr=finitevol}, we can construct for each end $\gamma \in \gC_0(\cG)$ a function $f_\gamma \in \dom(\bH) \cap H^1(\cG)$ such that $f_\gamma$ is non-constant only on finitely many edges (since $\#\partial \cG^\gamma<\infty$), $f_\gamma (\gamma) = 1$ and $f_\gamma (\gamma') = 0$ for all other ends $\gamma' \in \gC_0(\cG)\setminus\{\gamma\}$. 
Clearly, $\Gamma_1 f_\gamma = 0$ for every $\gamma\in \gC_0(\cG)$. Thus, setting 
\[
f = \sum_{\gamma\in \gC_0(\cG)} \wh{f}(\gamma) f_\gamma
\]
for a given $\wh{f}\in\cH$, we clearly have $\Gamma_0 f  = \wh{f}$ and $\Gamma_1f = 0$.

(ii) Let us first show that \eqref{eq:GG} holds true for all $f \in \dom(\bH) \cap H^1(\cG)$ if $g = f_\gamma \in H^1(\cG)$. 
Take a compact exhaustion $(\mathcal{F}_k)$ of $\cG$. Then by Lemma~\ref{lem:ibp},
\begin{align*}
	\langle\bH f, f_\gamma\rangle_{L^2(\cG)} & - \langle f', f_\gamma'\rangle_{L^2(\cG)}  
	= \lim_{k\to \infty} \langle\bH f, f_\gamma\rangle_{L^2(\mathcal{F}_k)}  - \langle f', f_\gamma'\rangle_{L^2(\mathcal{F}_k)} \\
	& = \lim_{k \to \infty} \sum_{v \in \partial \mathcal{F}_k} \frac{\partial f}{\partial n_{\mathcal{F}_k}} (v) f_\gamma(v)^\ast = \lim_{k \to \infty} \sum_{v \in \partial\mathcal{F}_k \cap \cV^\gamma} \frac{\partial f}{\partial n_{\mathcal{F}_k}} (v),
\end{align*}
where $\cV^\gamma$ is the set of vertices of $\cG^\gamma$. 
Notice that the subgraph $\cG^\gamma$ itself is a connected infinite graph having finite total volume and exactly one end, which can be identified with $\gamma$ in an obvious way. Moreover, setting $\mathcal{F}^\gamma_k := \mathcal{F}_k \cap  \cG^\gamma$ for all $k\ge 0$ and noting that $\mathcal{F}^\gamma_k$ is connected for all sufficiently large $k$, the sequence $(\mathcal{F}^\gamma_k)$ provides a compact exhaustion of $\cG^\gamma$. Since $\partial_{\cG^\gamma} \mathcal{F}^\gamma_k = \partial \mathcal{F}_k \cap \cV^\gamma$ and
\[
	\frac{\partial f}{ \partial n_{ \mathcal{F}^\gamma_k}} (v) =  \frac{\partial f}{  \partial n_{ \mathcal{F}_k}} (v) , 
	\qquad v \in \partial_{\cG^\gamma}  \mathcal{F}^\gamma_k,
	\]
for all large enough $k\ge 0$, we get by applying Lemma~\ref{cor:coex}
\[
	\langle\bH f, f_\gamma\rangle_{L^2(\cG)}  - \langle f', f_\gamma'\rangle_{L^2(\cG)} 
	= \lim_{k \to \infty}\sum_{v \in  \mathcal{F}_k \cap \cV^\gamma } \frac{\partial f}{\partial n_{ \mathcal{F}_k^\gamma}} (v) 
	= - \frac{\partial f}{\partial n} (\gamma).
\]
Hence \eqref{eq:GG} holds true if $g = f_\gamma \in H^1(\cG)$.

Now observe that a simple integration by parts implies that \eqref{eq:GG} is valid for all compactly supported $g \in H^1(\cG)$. By continuity  and Theorem~\ref{th:H10} this extends further to all $g\in H^1_0(\cG)$. 
Finally, setting $\ti g := g - \sum_{\gamma\in \gC_0(\cG)} {g}(\gamma) f_\gamma$ for $g\in H^1(\cG)$, it is immediate to check that, by Theorem~\ref{th:H10}, $\ti{g}\in H^1_0(\cG)$. It remains to use the linearity of $\Gamma_0$.
\end{proof}

 It turns out that the domain of the Neumann extension admits a simple description.
 
\begin{corollary} \label{prop:domNM}
Let $\cG$ be a metric graph with $\# \gC_0(\cG) < \infty$. Then the Neumann extension $\bH_N$ is given as the restriction $\bH_N = \bH|_{\dom(\bH_N)}$ to the domain
\begin{align} \label{eq:HN}
	\dom(\bH_N) = \big \{f \in \dom(\bH) \cap H^1(\cG) | \; \Gamma_1 f =0 \big \}.
\end{align}
\end{corollary}

\begin{proof}
By the first representation theorem~\cite[Chapter~VI.2.1]{kato}, $\dom(\bH_N)$ consists of all functions $f \in H^1(\cG)$ such that there exists $h \in L^2(\cG)$ with
\[
	\langle f', g'\rangle_{L^2(\cG)} = \langle h, g\rangle_{L^2(\cG)}, \qquad \text{for all } g \in H^1(\cG).
\]
Moreover, in this case  $\bH_N f:= h$. Taking into account Proposition~\ref{lem:GG} and the fact that $\bH_N$ is a restriction of $\bH$, we immediately arrive at \eqref{eq:HN}.
\end{proof}

Our next goal is to prove surjectivity of the normal derivative map.

\begin{proposition} \label{prop:surjder}
If  $\cG$ is a metric graph with $\# \gC_0(\cG) < \infty$, then the mapping $\Gamma_1$
is surjective.
\end{proposition}

In fact, Proposition~\ref{prop:surjder} will follow from the following lemma.

\begin{lemma}  \label{lem:surjder0}
Suppose $\cG$ is a metric graph with $\vol (\cG) < \infty$ and only one end, $\gC(\cG) = \{\gamma\}$.  Then there exists $f\in \dom(\bH) \cap H^1(\cG)$ such that
\begin{align*}
\partial_n f (\gamma ) \neq 0.
\end{align*}
\end{lemma}

\begin{proof}
We will proceed by contradiction. Suppose that $\partial_n g (\gamma) = 0$ for all $g \in  \dom(\bH) \cap H^1(\cG)$. Then,  by Corollary~\ref{prop:domNM}, $\dom(\bH_F) \subseteq \dom(\bH_N) = \dom(\bH) \cap H^1(\cG)$. However, both $\bH_F$ and $\bH_N$ are self-adjoint restrictions of $\bH$ and hence $\dom(\bH_F) = \dom(\bH_N)$. Therefore, $\bH_F=\bH_N$ and their quadratic forms also coincide, which implies that $H^1_0 (\cG) = H^1(\cG)$. This contradicts Corollary~\ref{cor:H=H} and hence completes the proof.
\end{proof} 

\begin{proof}[Proof of Proposition~\ref{prop:surjder}]
Let $\cG^\gamma$, $\gamma\in \gC_0 (\cG)$ be the subgraphs of $\cG$ constructed in the proof of Proposition~\ref{lem:GG}(i). Every $\cG^\gamma$ is a connected graph with $\vol(\cG^\gamma) < \infty$ and only one end, which can be identified with $\gamma$. Hence we can apply Lemma~\ref{lem:surjder0} to obtain a function $\ti{g}_\gamma \in \dom (\bH^\gamma) \cap H^1(\cG^\gamma)$ such that $\partial_n \ti{g}_\gamma (\gamma) = 1$. Here $\bH^\gamma$ denotes the Kirchhoff Laplacian on $\cG^\gamma$. 

Since $\# \partial \cG^\gamma  < \infty$, we can obviously extend $\ti g_\gamma$ to a function $g_\gamma$ on $\cG$ such that $g_\gamma \in \dom(\bH) \cap H^1(\cG)$ and $g_\gamma$ is identically zero on a neighborhood of each end $\gamma' \neq \gamma$ (see also the proof of Theorem~\ref{thm:nontr=finitevol}). In particular, this implies that $\partial_n g_\gamma (\gamma')  = 0$ for all $\gamma' \in \gC_0(\cG)\setminus\{\gamma\}$. Upon identification of $\gamma$ with the single end of $\cG^\gamma$ we also have that
\[
	\partial_n g_\gamma(\gamma ) = \partial_n \ti{g}_\gamma (\gamma ) =1.
\]
This immediately implies surjectivity.
\end{proof}

%%%%%%%%%%%%%%%%%%%%%%%%%
\subsection{Description of self-adjoint extensions}\label{sec:VI.03}
%%%%%%%%%%%%%%%%%%%%%%%%%

Our next goal is a description of all finite energy self-adjoint extensions of $\bH_0$, that is, self-adjoint extensions $\wt \bH$ satisfying the inclusion $\dom (\wt \bH) \subset H^1(\cG)$. We will be able to do this under the additional assumption that $\cG$ has finitely many finite volume ends.  Recall that in this case $\cH = \ell^2(\gC_0(\cG))$ is a finite dimensional Hilbert space. 

Let $C$, $D$ be  two linear operators on $\cH$ satisfying \emph{Rofe-Beketov conditions}~\cite{RB}:
\begin{align}\label{eq:RB}
CD^\ast & = DC^\ast, & {\rm rank}(C|D) & = \dim \cH = \#\gC_0(\cG). %\in \rho(C^\ast C + D^\ast D).
\end{align}
Consider the quadratic form $\gt_{C,D}$ defined by
\begin{align} \label{eq:extqf0}
\gt_{C,D} [f]:=\int_\cG |f'(x)|^2dx + \langle D^{-1} C\Gamma_0f, \Gamma_0f\rangle_\cH
\end{align}
on the domain
\begin{align} \label{eq:extqf1}
\dom(\gt_{C,D}) :=\{f\in H^1(\cG)|\, \Gamma_0f \in\ran(D^\ast) \}.
\end{align}
Here and in the following the mappings $\Gamma_0$ and $\Gamma_1$ are given by \eqref{def:Gamma} and $D^{-1}\colon \ran (D) \to \ran(D^\ast)$ denotes the inverse of the restriction $D|_{\ker(D)^\perp} \colon \ran(D^\ast) \to \ran(D)$. 
In particular, \eqref{eq:RB} implies that $\gt_{C,D}[f]$ is well-defined for all $f \in \dom(\gt_{C,D})$ (see also \eqref{eq:domCD}).

\begin{remark}
It is straightforward to check that $\gt_{I,0} = \gt_F$ and $\gt_{0,I} = \gt_N$ are the quadratic forms corresponding to the Friedrichs extension $\bH_F$ and, respectively, Neumann extension $\bH_N$ (see Remark~\ref{rem:Friedrichs} and \eqref{eq:NMform}). 
\end{remark}

Now we are in position to state the main result of this section. 

\begin{theorem}\label{th:ThetaCD}
Let $\cG$ be a metric graph with finitely many finite volume ends, $\# \gC_0(\cG) < \infty$. Let also $C$, $D$ be linear operators on $\cH$ satisfying Rofe-Beketov conditions \eqref{eq:RB}. Then:
\begin{enumerate}[(i)]
\item The form $\gt_{C,D}$ given by \eqref{eq:extqf0}, \eqref{eq:extqf1} is closed and lower semi-bounded in $L^2(\cG)$.
\item  The self-adjoint operator $\bH_{C,D}$ associated with the form $\gt_{C,D}$ is a self-adjoint extension of $\bH_0$ and its domain is explicitly given by 
\begin{align} \label{eq:domext}
\dom({\bH}_{C,D}) = \{f\in \dom({\bH}) \cap H^1(\cG)|\ C\Gamma_0f + D\Gamma_1 f  = 0  \}.
\end{align}
\item Conversely, if $\wt \bH$ is a self-adjoint extension of ${\bH_0}$ such that $\dom(\wt \bH) \subset H^1(\cG)$, then there are $C,D$ satisfying \eqref{eq:RB} such that $\wt\bH = \bH_{C,D}$.
\item Moreover, $\wt \bH = \bH_{C,D}$ is a Markovian extension if and only if the corresponding quadratic form $\wh{\gt}_{C,D}[y] = \langle D^{-1} C y, y\rangle_\cH$, $\dom(\wh\gt) = \ran(D^\ast)$ is a Dirichlet form on $\cH$ in the wide sense.\footnote{Here we do not assume that $\wh\gt$ is densely defined, see~\cite[p.29]{fuk10}. We stress that in order for $\wh\gt$ to be a Dirichlet form even merely in the wide sense, it is necessary that $\dom(\wh\gt)$ is a sublattice of $\cH$, hence that the orthogonal projector onto $\ran(D^\ast)$ is a positivity preserving operator.}
\end{enumerate}
\end{theorem}

\begin{proof} 
(i) Since $\cH$ is finite dimensional, it is straightforward to see that the form $\gt_{C,D}$ is closed and lower semi-bounded in $L^2(\cG)$ whenever $C$ and $D$ satisfy \eqref{eq:RB}.

(ii) By the first representation theorem~\cite[Chapter~VI.2.1]{kato}, $\dom(\bH_{C,D})$ consists of all functions $f \in \dom(\gt_{C,D}) \subseteq H^1(\cG)$ for which there exists $h \in L^2(\cG)$ such that  
\be\label{eq:Bform}
	\langle f', g'\rangle_{L^2(\cG)} + \langle D^{-1} C\Gamma_0 f, \Gamma_0 g\rangle_\cH = \langle h, g\rangle_{L^2 (\cG)}
\ee
for all $g \in \dom(\gt_{C,D})$. Moreover, in this case $\bH_{C,D} f := h$. 

The Gauss--Green identity \eqref{eq:GG} implies that for any $f \in \dom(\bH_{C,D})$ and $g \in \dom(\gt_{C,D})$,
\[
	\langle D^{-1} C\Gamma_0 f, \Gamma_0 g\rangle_\cH  = - \langle \Gamma_1 f,\Gamma_0 g\rangle_\cH.
\]
Taking into account the surjectivity property in Proposition~\ref{lem:GG}(i), the inclusion ''$\subseteq$" in \eqref{eq:domext} follows. The converse inclusion is then an immediate consequence of the Gauss--Green identity  \eqref{eq:GG}.

(iii) To prove the claim, it suffices to show that
\[
	\Theta = \{ (\Gamma_0f,\Gamma_1f) |\,  f \in \dom(\wt \bH) \} \subseteq \cH \times \cH  
\]
is a self-adjoint linear relation (for further details we refer to Appendix~\ref{app:LR}). By definition, 
 $\Theta^\ast$ is given by
\[
	\Theta^\ast = \{ (g, h) \in \cH \times \cH |\, \langle\Gamma_1 f, g\rangle_{\cH} = \langle\Gamma_0 f, h\rangle_{\cH}  \text{ for all } f \in \dom(\wt \bH) \}.
\]
The inclusion $\Theta \subseteq \Theta^\ast$ follows immediately from the Gauss--Green identity \eqref{eq:GG} and the self-adjointness of $\wt \bH$. Indeed, we clearly have
\[
	0 = \langle\wt \bH f, \ti f \rangle_{L^2(\cG)} - \langle f, \wt \bH\ti f \rangle_{L^2(\cG)} =  
	- \langle\Gamma_1 f, \Gamma_0 \ti f \rangle_{\cH} + \langle \Gamma_0 f, \Gamma_1 \ti f \rangle_{\cH}
\]
for all functions $f, \ti f \in \dom (\wt \bH)$. On the other hand, by Proposition~\ref{prop:surjder}  and Proposition~\ref{lem:GG}, for any $(g, h) \in \Theta^\ast$ there is a function $\ti{f}\in \dom (\bH) \cap H^1(\cG)$ such that $g = \Gamma_0\ti{f}$ and $h = \Gamma_1\ti{f}$. Employing the identity \eqref{eq:GG} once again, we see that 
\begin{align*}
	\langle\wt \bH f, \ti f \rangle_{L^2(\cG)} & = \langle f', \ti f '\rangle_{L^2(\cG)} - \langle\Gamma_1 f, g \rangle_{\cH} \\
	&= \langle f', \ti f '\rangle_{L^2(\cG)} - \langle \Gamma_0f, h \rangle_\cH = \langle f, \bH \ti f \rangle_{L^2(\cG)}
\end{align*}
for all $f \in \dom(\wt \bH)$. 
Hence, $\ti{f} \in \dom (\wt \bH)$ and in particular $(g, h) \in \Theta$. Since $\Theta$ is self-adjoint, there are $C$ and $D$ in $\cH$ satisfying Rofe-Beketov conditions \eqref{eq:RB} and such that $\Theta = \{(f,g)\in\cH\times\cH|\, Cf+Dg=0\}$.

(iv) 
The first direction of the equivalence is clear: since the quadratic form $\gt_N$ associated with the Neumann extension $\bH_N$ is Markovian and 
\[
\Gamma_0 ( \varphi \circ f) =  \big((\varphi \circ f)(\gamma)\big)_{\gamma\in\gC_0(\cG)} = : \varphi\circ( \Gamma_0 f)
\]
 for all functions $f \in H^1(\cG)$ and every normal contraction $\varphi$,\footnote{A \emph{normal contraction} is a function $\varphi \colon \C \to \C$ such that $\varphi(0) = 0$ and $| \varphi(x) -  \varphi(y)| \le |x-y|$ for all $x, y \in \C$.}  the extension $\bH_{C,D}$ is Markovian if $\wh\gt_{C,D}$ is a Dirichlet form on $\cH$ in the wide sense.
 
 To prove the converse direction, let, for simplicity, $f\in \dom(\wh\gt_{C,D})$ be real-valued and fix some real-valued $\ti f \in H^1(\cG)$ with $\Gamma_0 \ti f = f$ (the existence of such an $\ti f$ follows from Proposition~\ref{lem:GG}). For any (real-valued) normal contraction $\varphi \colon \R \to \R$, we can construct a continuous and piecewise affine function $\psi \colon \R \to \R$ (i.e., $\psi$ is affine on every component of $\R \setminus \{x_1,\dots, x_M\}$ for finitely many points $x_1,\dots,x_M$) such that $\psi(0)= 0$, $\psi(f(\gamma)) = \varphi(f(\gamma))$ for all $\gamma \in \gC_0(\cG)$ and $|\psi'(x)| = 1$ for almost every $x \in \R$.\footnote{For instance, for any $s,L>0$ such that $s\le L$, the function $\psi_0(x):= \frac{L+s}{2} - \Big |x - \frac{L+s}{2} \Big|$ satisfies $\psi_0(0)= 0$, $\psi_0 (L) = s$ and $|\psi_0'| \equiv 1$. The construction in the general case follows easily from this example.} Notice that every function $\psi$ with the above properties is a normal contraction. Hence, if $\gt_{C,D}$ is Markovian, it follows that $\psi \circ \ti f \in \dom(\gt_{C,D})$. However, its boundary values are precisely given by 
 \[
	\Gamma_0(\psi \circ \ti f) = \psi \circ f = \varphi \circ f
 \]
 and we conclude that $\varphi \circ  f$ belongs to $\dom (\wh\gt_{C,D})$. Finally, the Markovian property of $\gt_{C,D}$ implies that
 \[
  \gt_{C,D} [\psi \circ \ti f ] = \int_{\cG} |(\psi \circ \ti f)'|^2 dx + \wh\gt _{C,D}[\varphi \circ f]   \le \gt_{C,D} [ \ti f ]= \int_{\cG} |\ti f'|^2 dx + \wh\gt _{C,D}[ f],
  \]
and noticing that $|( \psi \circ \ti f)'| =  |\ti f'|$ almost everywhere on $\cG$, the proof is complete.
\end{proof}

Let us demonstrate Theorem~\ref{th:ThetaCD} by applying it to Cayley graphs.% of finitely generated groups.

\begin{corollary}\label{cor:group1end}
Let $\cG_d$ be a Cayley graph of a finitely generated group $\mG$ with one end. Then the Kirchhoff Laplacian $\bH_0$ admits a unique Markovian extension if and only if the underlying metric graph $\cG = (\cG_d,|\cdot|)$ has infinite total volume,  $\vol(\cG)=\infty$. Moreover, if $\cG$ has finite total volume, then the set of all Markovian extensions of $\bH_0$ forms a one-parameter family given explicitly by 
\begin{align} \label{eq:domext-2}
\dom({\bH}_{\theta}) = \{f\in \dom({\bH}) \cap H^1(\cG)|\, \cos(\theta)\Gamma_0f + \sin(\theta)\Gamma_1 f  = 0  \},
\end{align}
where $\theta\in[0,\pi/2]$.
\end{corollary}

Taking into account that amenable groups have finitely many ends, the above result applies to amenable finitely generated groups, which are not virtually infinite cyclic (see Remark~\ref{rem:ends}(iv)). In a similar way one can obtain a complete description of Markovian extensions in the case of virtually infinite cyclic groups, however, they have two ends and the corresponding description looks a little bit more cumbersome and we leave it to the reader (cf.~\cite[p.147]{fuk10}). The case of groups with infinitely many ends remains an open highly nontrivial problem.  

\begin{remark} \label{rem:thThetaCD}
A few remarks are in order.
\begin{itemize}
\item[(i)] 
Let us mention that in the case when the domain of the maximal operator $\bH$ is contained in $H^1(\cG)$ and $\cG$ has finitely many finite volume ends (notice that by Theorem~\ref{th:indicesends} in this case $\Nr_\pm(\bH_0) = \#\gC_0(\cG)<\infty$), \emph{Proposition~\ref{th:ThetaCD} provides a complete description of all self-adjoint extensions of $\bH_0$}. Let us also mention that Proposition~\ref{th:ThetaCD} provides a complete description of all self-adjoint restrictions of the Gaffney Laplacian $\bH_G$, see Remark~\ref{rem:H=H=uniq}(ii).
\item[(ii)] 
Some of the results of this section extend (to a certain extent) to the case of infinitely many ends. Let us stress that by Proposition~\ref{prop:H1InfEnds} in the case when $\cG$ has a finite volume end which is not free the above results would lead only to  some (not all!) self-adjoint extensions of  $\bH_0$. In our opinion, even in the case of radially symmetric trees having finite total volume the description of all self-adjoint extensions of $\bH_0$ is a difficult problem. 
\item[(iii)]  
Similar relations between Markovian realizations of elliptic operators on domains or finite metric graphs (with general couplings at the vertices) on one hand, and Dirichlet property of the corresponding quadratic form's boundary term on the other hand, are of course well known in the literature (see, e.g.,~\cite[Proposition~5.1]{CarMug09},~\cite[Theorem~3.5]{KanKlaVoi09},~\cite[Theorem~6.1]{KosPotSch08}). However, the setting of infinite metric graphs additionally requires much more advanced considerations of combinatorial and topological nature. In particular, it seems noteworthy to us that the results of the previous sections provide the right notion of the boundary for metric graphs, namely, the set of finite volume ends, to deal with finite energy and also with Markovian extensions of the minimal Kirchhoff Laplacian. In particular, this end space is well-behaved as concerns the introduction of traces and normal derivatives.
\item[(iv)] 
Taking into account certain close relationships between quantum graphs and discrete Laplacians (see~\cite[\S~4]{ekmn}), one can easily obtain the results analogous to Theorem~\ref{th:indicesends} and Theorem~\ref{th:ThetaCD} for a particular class of discrete Laplacians on $\cG_d$ defined by the following expression
\begin{align}\label{eq:laplacediscr}
(\tau f)(v) := \frac{1}{m(v)}\sum_{u\sim v}\frac{f(v) - f(u)}{|e_{u,v}|},\quad v\in \cV,
\end{align}
where $m$ is the star weight \eqref{def:m}. Markovian extensions of weighted discrete Laplacians were considered also in~\cite{klss}.
On the other hand,~\cite{klss} does not contain a finiteness assumption, however, the conclusion in our setting appears to be slightly stronger than in~\cite[Theorem~3.5]{klss}, where the correspondence between Markovian extensions and Markovian forms on the boundary is in general not bijective. 
\end{itemize}
\end{remark}

%%%%%%%%%%%%%%%%%%%%%%%%%%%%%%
\section{Deficiency indices of antitrees}\label{sec:antitrees}
%%%%%%%%%%%%%%%%%%%%%%%%%%%%%%
The main aim of this section is to construct for any $N\in \Z_{\ge 1}\cup \{\infty\}$ a metric antitree such that the corresponding minimal Kirchhoff Laplacian $\bH_0$ has deficiency indices $\Nr_\pm(\bH_0) = N$. Our motivation stems from the fact that every antitree has exactly one end and hence, according to considerations in the previous sections, $\bH_0$ admits at most one-parameter family of Markovian extensions. 

\subsection{Antitrees}\label{ss:AT01}
 Let $\cG_d = (\cV, \cE)$ be a connected, simple combinatorial graph. Fix a root vertex $o \in \cV$ and then order the graph with respect to the combinatorial spheres $S_n$, $n \ge 0$ (notice that $S_0=\{o\}$). 
 $\cG_d$ is called an \emph{antitree} if 
every vertex in $S_n$, $n\ge 1$, is connected to all  vertices in $S_{n-1}$ and $S_{n+1}$ and no vertices in $S_k$ for all $|k-n|\neq 1$ (see Figure~\ref{fig:antitree}). Notice that each antitree is uniquely determined by its sequence of sphere numbers $(s_n)$, $s_n := \#S_n$ for $n \ge 0$. 

While antitrees first appeared in connection with random walks~\cite{dk,klw, woj11}, they were actively studied from various different perspectives in the last years (see ~\cite{bl19, dfs, kn19} for quantum graphs and~\cite[Section 2]{clmp} for further references).

Let us enumerate the vertices in every combinatorial sphere $S_n$ by $(v^n_i)_{i=1}^{s_n}$ and denote the edge connecting $v^n_i$ with $v^{n+1}_j$ by $e_{ij}^n$, $1\le i \le s_n$, $1 \le j \le s_{n+1}$. We shall always use $\cA$ to denote (metric) antitrees. 

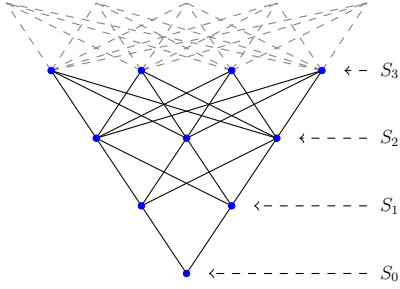
\begin{figure}
	\begin{center}
		\begin{tikzpicture}    [%
		,scale=.6
		,every node/.style={scale=.6}]
		%%%%%%%%%% EDGES %%%%%%%%%%%%%%%%%%
		%%%% FROM S0 to S1
		\draw								(0,0) -- (-1, 1.5) ;
		\draw								(0,0) -- (1, 1.5) ;
		
		%%%% FROM S1 to S2
		\draw								(-1,1.5) -- (-2, 3) ;
		\draw								(-1,1.5) -- (0, 3) ;
		\draw								(-1,1.5) -- (2, 3) ;
		\draw								(1,1.5) -- (-2, 3) ;
		\draw								(1,1.5) -- (0, 3) ;
		\draw								(1,1.5) -- (2, 3) ;
		
		%%%% FROM S2 to S3
		\draw								(-2, 3) -- (-3, 4.5);
		\draw								(-2, 3) -- (-1, 4.5);
		\draw								(-2, 3) -- (1, 4.5);
		\draw								(-2, 3) -- (3, 4.5);
		\draw								(0, 3) -- (-3, 4.5);
		\draw								(0, 3) -- (-1, 4.5);
		\draw								(0, 3) -- (1, 4.5);
		\draw								(0, 3) -- (3, 4.5);
		\draw								(2, 3) -- (-3, 4.5);
		\draw								(2, 3) -- (-1, 4.5);
		\draw								(2, 3) -- (1, 4.5);
		\draw								(2, 3) -- (3, 4.5);
		
		%%%% FROM S3 to S4
		\draw [dashed, gray]			(-3, 4.5) -- (-4, 6);
		\draw [dashed , gray]			(-3, 4.5) -- (-2, 6);
		\draw [dashed, gray]			(-3, 4.5) -- (0, 6);
		\draw [dashed , gray]			(-3, 4.5) -- (2, 6);
		\draw [dashed , gray]			(-3, 4.5) -- (4, 6);
		\draw [dashed , gray]			(-1, 4.5) -- (-4, 6);
		\draw [dashed , gray]			(-1, 4.5) -- (-2, 6);
		\draw [dashed , gray]			(-1, 4.5) -- (0, 6);
		\draw [dashed , gray]			(-1, 4.5) -- (2, 6);
		\draw [dashed , gray]			(-1, 4.5) -- (4, 6);
		\draw [dashed , gray]			(1, 4.5) -- (-4, 6);
		\draw [dashed , gray]			(1, 4.5) -- (-2, 6);
		\draw [dashed , gray]			(1, 4.5) -- (0, 6);
		\draw [dashed , gray]			(1, 4.5) -- (2, 6);
		\draw [dashed , gray]			(1, 4.5) -- (4, 6);
		\draw [dashed , gray]			(3, 4.5) -- (-4, 6);
		\draw [dashed , gray]			(3, 4.5) -- (-2, 6);
		\draw [dashed , gray]			(3, 4.5) -- (0, 6);
		\draw [dashed , gray]			(3, 4.5) -- (2, 6);
		\draw [dashed , gray]			(3, 4.5) -- (4, 6);
		
		%%%% SPHERES NUMBERING
		\draw	[thin, ->, dashed]			(4, 1.5) -- (1.5, 1.5);
		\draw	[thin, ->, dashed]			(4, 3) -- (2.5, 3);
		\draw	[thin, -> , dashed]			(4, 4.5) -- (3.5, 4.5);
		\draw	[thin, -> , dashed]			(4, 0) -- (0.5, 0);
		\node  at (4.5, 0) {\Large $S_0$};
		\node  at (4.5, 1.5) {\Large $S_1$};
		\node  at (4.5, 3) {\Large $S_2$};
		\node  at (4.5, 4.5) {\Large $S_3$};	
		%%%%%%%%%%% VERTICES %%%%%%%%%%%%%
		\filldraw[color=blue]
		(0,0) circle (2pt) 
		(-1, 1.5) circle (2pt)
		(1, 1.5) circle (2pt)
		(-2, 3) circle (2pt)
		(0, 3) circle (2pt)
		(2, 3) circle (2pt)
		(-3, 4.5) circle (2pt)
		(-1, 4.5) circle (2pt)
		(1, 4.5) circle (2pt)
		(3, 4.5) circle (2pt);
		\end{tikzpicture}
	\end{center}
	\caption{Antitree with sphere numbers $s_n = n+1$.}\label{fig:antitree}
\end{figure}

It is clear that every (infinite) antitree has exactly one end. By Theorem~\ref{th:indicesends}, the deficiency indices of the corresponding minimal Kirchhoff Laplacian are at least $1$ if $\vol(\cA)<\infty$. On the other hand, under the additional symmetry assumption that $\cA$ is radially symmetric (that is, for each $n\ge 0$, all edges connecting combinatorial spheres $S_n$ and $S_{n+1}$ have the same length), it is known that the deficiency indices are at most $1$ (see~\cite[Theorem~4.1]{kn19} and Example~\ref{ex:rsat}). It turns out that upon removing the symmetry assumption it is possible to construct antitrees such that the corresponding minimal Kirchhoff Laplacian has arbitrary finite or infinite deficiency indices. More precisely, the main aim of this section is to prove the following result.

\begin{theorem}\label{th:DefInd}
Let $\cA$ be the antitree with sphere numbers $s_n = n+1$, $n\ge 0$ (Figure~\ref{fig:antitree}). Then for each $N\in \Z_{\ge 1}\cup\{\infty\}$ there are edge lengths such that the corresponding minimal Kirchhoff Laplacian $\bH_0$ has the deficiency indices $\Nr_\pm(\bH_0) = N$. 
\end{theorem} 
\subsection{Harmonic functions}
 As it was mentioned already, every harmonic function is uniquely determined by its values at the vertices. 
 On the other hand,  $\f\in C(\cV)$ defines a function $f\in\HH (\cA)$ with $f|_{\cV} = \f$ if and only if the following conditions are satisfied:
\begin{align} \label{eq:harmrecur}
	\sum_{j=1}^{s_{n+1}} \frac{f(v_j^{n+1}) - f(v_k^n)}{ |e_{kj}^n|} + \sum_{i=1}^{s_{n-1}} \frac{ f(v_i^{n-1}) - f(v_k^n)}{  |e_{ik}^{n-1} |} = 0,
\end{align}
at each $v^n_k$,  $1 \le k \le s_n$ with $n{\ge 0}$. 
We set $s_{-1}:=0$ for notational simplicity and hence the second summand in \eqref{eq:harmrecur} is absent when $n=0$. We can put the above difference equations into the more convenient matrix form. Denote $\f_n := f|_{S_n} = (f(v_i^n))_{i=1}^{s_n}$ for all $n\in\Z_{\ge 0}$ and 
introduce matrices 
\begin{align}\label{def:M}
M_{n+1} := 
\begin{pmatrix} 
\frac{1}{|e_{11}^n|} & \frac{1}{|e_{12}^n|} & \dots & \frac{1}{|e_{1s_{n+1}}^n|} \\
\frac{1}{|e_{21}^n|} & \frac{1}{|e_{22}^n|} & \dots & \frac{1}{|e_{2s_{n+1}}^n|} \\
\dots &\dots &\dots &\dots \\
\frac{1}{|e_{s_n1}^n|} & \frac{1}{|e_{s_n2}^n|} & \dots & \frac{1}{|e_{s_ns_{n+1}}^n|} 
\end{pmatrix} \in \R^{s_n\times s_{n+1}},
\end{align}
and 
\begin{align}\label{def:D}
D_n & := {\rm diag}(d_k^n)\in \R^{s_n\times s_n}, & d_k^n & := \sum_{j=1}^{s_{n+1}}\frac{1}{|e_{kj}^n|} + \sum_{i=1}^{s_{n-1}}\frac{1}{|e_{ik}^{n-1}|},
\end{align}
for all $n\in\Z_{\ge 0}$. % with $s_{-1}:=0$ for notational simplicity. 
Notice the following useful identity
\begin{align}\label{eq:D=M}
d_1^0 & = M_1\id_{s_1}, & 
\begin{pmatrix} d^n_1\\ \vdots\\ d^n_{s_n} \end{pmatrix} = D_n\id_{s_n} & = (M_{n+1}\ M_n^\ast)\begin{pmatrix}\id_{s_{n+1}}\\ \id_{s_{n-1}}\end{pmatrix},\quad n{\ge 1},
\end{align}
where $\id_{s_n} := (1,\dots,1)^\top \in \R^{s_n}$.
Hence \eqref{eq:harmrecur} can be written as follows
\begin{align}
M_{1}\f_{1} &= \sum_{j=1}^{s_1} \frac{1}{|e_{1j}^0|} \f_0 = d_1^0 \f_0, \label{eq:harmrec1}\\
M_{n+1}\f_{n+1} & = D_n\f_n - M_n^\ast \f_{n-1},\quad n\ge 1.\label{eq:harmrec2}
\end{align}
Since $D_n$ is invertible, we get
\begin{align}\label{eq:recurA}
 \f_n & = D_n^{-1}(M_{n+1}\  M_n^\ast)\begin{pmatrix}\f_{n+1}\\ \f_{n-1}\end{pmatrix}
\end{align}
for all $n\ge 1$. In particular, $\f_n\in \ran\big(D_n^{-1}(M_{n+1}\  M_n^\ast)\big)$ for all $n\ge 1$, which implies that the number of linearly independent solutions to the above difference equations (and hence the number of linearly independent harmonic functions) depends on the ranks of the matrices $(M_{n+1}\  M_n^\ast)$, $n\ge 1$. Let us demonstrate this by considering the following example.

\begin{lemma}
Let $\cA$ be a radially symmetric antitree. Then 
\be
\HH(\cA) = \Span\{\id_\cG\}.
\ee
\end{lemma}

\begin{proof}
Let for each $n\ge 0$, all edges connecting combinatorial spheres $S_n$ and $S_{n+1}$ have the same length, say $\ell_n >0$. Clearly, in this case 
\[
\ran(M_{n+1}) = \ran(M_{n}^\ast) =\Span\{\id_{s_n}\},
\] 
for all $n\ge 1$. 
Moreover, each $D_n$ is a scalar multiple of the identity matrix $I_{s_n}$ and hence \eqref{eq:recurA} implies that $\f_n = c_n \id_{s_n}$ with some $c_n\in\C$ for all $n\ge 0$. Plugging this into \eqref{eq:harmrec1}--\eqref{eq:harmrec2}, we get
\begin{align*}
c_1 & = c_0, & 
c_{n+1} & = c_n + \frac{s_{n-1}\ell_{n}}{s_{n+1}\ell_{n-1}}(c_n - c_{n-1}),\ \ n\ge 1.
\end{align*}
Hence $c_n=c_0 = f(o)$ for all $n\ge 0$, which proves the claim.
\end{proof}

The latter in particular implies the following statement (cf.~\cite[Theorem~4.1]{kn19}).

\begin{corollary}\label{lem:ATrad}
If $\cA$ is a radial antitree with finite total volume, then $\Nr_\pm(\bH_0) = 1$.
\end{corollary}

\begin{proof}
By Corollary~\ref{cor:finvol=ker}, we only need to show that $\Nr_\pm (\bH_0)\le 1$. However, 
\[
\Nr_\pm (\bH_0) = \dim(\ker(\bH)) \le \dim(\HH(\cA)) = 1.\qedhere
\]
\end{proof}

\subsection{Finite deficiency indices}\label{ss:AT02}
We restrict our further considerations to a special case of polynomially growing antitrees.
Namely, for every $N\in \Z_{\ge 1}$, the antitree $\cA_N$ has sphere numbers $s_0=1$ and $s_n:=n+N$ for all $n\in\Z_{\ge1}$. 
To define its lengths, pick a sequence of positive numbers $(\ell_n)$ and set
\begin{align}\label{eq:lengthN}
|e_{ij}^n| := \begin{cases} 2\ell_n, & \hbox{if }1\le i=j \le N,\\ \ell_n, & \text{otherwise}, \end{cases}
\end{align}
for all $n\in\Z_{\ge0}$. 

\begin{lemma}\label{lem:harmN}
If a metric antitree $\cA_N$ has lengths given by \eqref{eq:lengthN}, then
\be\label{eq:harmN}
\dim \HH(\cA_N) = N+1.
\ee
\end{lemma}

\begin{proof}
Denoting 
\begin{align}
B_{n,m} & := \begin{pmatrix} 
1 & 1 & \dots & 1  \\
1 & 1 & \dots & 1  \\
\dots &\dots &\dots &\dots \\
1 & 1 & \dots & 1  
\end{pmatrix} \in \R^{n\times m}, & B_n & := B_{n,n} \in \R^{n\times n},
\end{align}
we get the following block-matrix form of the matrices $M_{n+1}$:
\begin{align}
%M_1 & = \frac{1}{2\ell_{0}}(1,1,\dots,1,2), & 
M_{n+1} = \frac{1}{\ell_{n}}\begin{pmatrix} 
B_N -\frac{1}{2}I_N & B_{N,n+1} \\
B_{n,N} & B_{n,n+1}  
\end{pmatrix} %,\quad n\ge 1.
\end{align}
for all $n\ge 1$. 
%It is immediate to see that
%\be
%{\rm rank}(M_{n+1}) = {\rm rank}(M_{n+1}^\ast)= {\rm rank}(M_{n+1}\ M_n^\ast) = N+1,
%\begin{cases} 1 , & 1\le n\le N-1,\\ N+1, & n\ge N.\end{cases}
%\ee
%for all $n\ge1$. 
Taking into account \eqref{def:D} and denoting
\[
d_{n}^1 := \frac{n+N-3/2}{\ell_{n -1}} + \frac{n+N +1/2}{\ell_{n }},
\quad d_{n }^2 := \frac{n+N-1 }{\ell_{n-1}} + \frac{n +N+1}{\ell_{n }},
\]
we get
\be
D_{n} = \begin{pmatrix} d_{n}^1I_N & 0\\  0& d_{n}^2I_{n} \end{pmatrix},
\ee
for all $n\ge 2$. 
Since $M_1 \in \C^{1\times (N+1)}$ and 
\be
\ran(M_{n+1}) = \ran(M_{n}^\ast) = \Span\left\{\begin{pmatrix} \f_N \\ \id_{n}\end{pmatrix}|\, \f_N\in\C^N\right\}, 
\ee
for all $n\ge 2$, \eqref{eq:recurA} implies that every $\f$ solving \eqref{eq:harmrec1}--\eqref{eq:harmrec2} must be of the form
\be\label{eq:f_form}
\f_{n} = \begin{pmatrix} \f^N_{n} \\ c_{n}\id_{n}\end{pmatrix}\in \C^{N+n},\qquad \f^N_{n}\in\C^N,\ \ c_{n}\in\C,
\ee 
for all $n\ge 1$.
Plugging \eqref{eq:f_form} into \eqref{eq:harmrec2} and taking into account that
\begin{align*}
B_N \f^N_{n} & = \overline{\f}^N_{n}\id_N, & \overline{\f}^N_{n} & := \langle \f^N_{n},\id_N \rangle = B_{1,N}\f^N_{n},
\end{align*}
we get after straightforward calculations
%%%%%calculations%%%%%%%%%%%
%\begin{align*}
%M_{n+1}\f_{n+1} & = \frac{1}{\ell_n}\begin{pmatrix} A_N\f_{n+1}^N + c_{n+1}(n+1)\id_N \\ \overline{\f}_{n+1}^N\id_n + c_{n+1}(n+1)\id_n\end{pmatrix},\\
%M_{n}^*\f_{n-1} & = \frac{1}{\ell_{n-1}}\begin{pmatrix} A_N\f_{n-1}^N + c_{n-1}(n-1)\id_N \\ \overline{\f}_{n-1}^N\id_n + c_{n-1}(n-1)\id_n\end{pmatrix},\\
%D_n\f_n &= \begin{pmatrix} d_{n}^1\f_n^N \\ c_nd_n^2 \id_n\end{pmatrix},
%\end{align*}
%which implies the following recursion formulas
%%%%%calculations%%%%%%%%%%%
\begin{align}
\frac{\overline{\f}_{n+1}^N  + c_{n+1}(n+1)}{\ell_{n}}\id_N - \frac{1}{2\ell_n}\f_{n+1}^N & = d_{n}^1 \f_{n}^N - \frac{\overline{\f}_{n-1}^N + c_{n-1}(n-1)}{\ell_{n-1}} \id_N +\frac{1}{2\ell_{n-1}}\f_{n-1}^N,\label{eq:sys1}\\
\frac{\overline{\f}_{n+1}^N + c_{n+1}(n+1)}{\ell_{n}} & = c_{n}d_{n}^2  - \frac{\overline{\f}_{n-1}^N + c_{n-1}(n-1)}{\ell_{n-1}},\label{eq:sys2}
\end{align}
for all $n\ge 2$.
Multiplying \eqref{eq:sys2} with $\id_N$ and then subtracting \eqref{eq:sys1}, we end up 
with
\begin{align}\label{eq:sys3}
\f_{n+1}^N = 2\ell_{n}(c_{n}d_{n}^2 \id_{N} - d_{n}^1 \f_{n}^N) - \frac{\ell_{n}}{\ell_{n-1}}\f_{n-1}^N,\quad n\ge 2.
\end{align}
Next taking the inner product in \eqref{eq:sys1} with $\id_N$ 
%%%%%calculations%%%%%%%%%%%
%\[
%\frac{N-1/2}{\ell_{n}}\overline{\f}_{n+1}^N + c_{n+1}\frac{(n+1)N}{\ell_{n}} = d_{n}^1 \overline{\f}_{n}^N - \frac{N-1/2}{\ell_{n-1}}%\overline{\f}_{n-1}^N - c_{n-1}\frac{N(n-1)}{\ell_{n-1}}
%\]
%%%%%calculations%%%%%%%%%%%
and then subtracting \eqref{eq:sys2} multiplied by $N-1/2$, we finally get
\begin{align}\label{eq:sys4}
 c_{n+1} = \frac{\ell_n}{n+1}(2d_{n}^1 \overline{\f}_{n}^N - (2N-1)d_n^2 c_n) - c_{n-1}\frac{(n-1)\ell_n}{(n+1)\ell_{n-1}},\quad n\ge 2.
\end{align}
Taking into account that the value of $f$ at the root $o$ is determined by $\f_1$ via
\be\label{eq:recur0}
f(o) = \f_0 = \frac{2\ell_0}{2N+1} M_1\f_1,
\ee
and noting that $\f_2^N$ and $c_2$ are also determined by $\f_1$, we conclude that \eqref{eq:sys3}--\eqref{eq:sys4} define $\f$ uniquely once $\f_1 \in\C^{N+1}$  is given. 
\end{proof}
 
 Lemma~\ref{lem:harmN} immediately implies that $\Nr_\pm(\bH_0) \le N+1$ if $\vol(\cA_N)<\infty$, where $\bH_0$ is the associated minimal operator. The next result shows that it can happen that $\Nr_\pm(\bH_0) = N+1$ upon choosing lengths $\ell_n$ with a sufficiently fast decay.
 
 \begin{proposition}\label{prop:ANindices}
 Let $\cA_N$ be the antitree as in Lemma~\ref{lem:harmN}. If $(\ell_n)$ is decreasing and
 \be\label{eq:ellNdecay}
 \sqrt{\ell_n} = \OO\left(\frac{1}{(6\sqrt{N})^n(n+N+3)!}\right)
 \ee
 as $n\to \infty$, then $\Nr_\pm(\bH_0) = N+1$.
 \end{proposition}

\begin{proof}
It is immediate to see that $\vol(\cA_N)<\infty$ if \eqref{eq:ellNdecay} is satisfied. Next, taking into account \eqref{eq:lengthN}, observe that 
\[
m(v) = \sum_{v\in\cE_v} |e| \le (n+N) \ell_{n-1} + (n+N+2) \ell_n \lesssim n\ell_{n-1},\quad v\in S_n,
\]
as $n\to \infty$. Suppose $f\in\HH(\cA)$ and set $\f = f|_\cV$. Then $\f$ has the form \eqref{eq:f_form} and hence
\[
\|\f_n\|^2 = \sum_{v\in S_n} |f(v)|^2  = \|\f_n^N\|^2 + n|c_n|^2, 
\]
for all $n\ge 1$. This implies the following estimate
\be\label{eq:est7}
\sum_{v\in\cV} |f(v)|^2m(v) = \sum_{n\ge 0}\sum_{v\in S_n} |f(v)|^2m(v) %\lesssim \sum_{n\ge 1} n\ell_{n-1}(\|\f_n^N\|^2 + n|c_n|^2)
\lesssim \sum_{n\ge 1} n^2\ell_{n-1}(\|\f_n^N\|^2 + |c_n|^2).
\ee
Next,  \eqref{eq:sys3}--\eqref{eq:sys4} can be written as follows
\begin{align}
\begin{pmatrix} \f^N_{n+1} \\ c_{n+1}\end{pmatrix} = A_{1,n}\begin{pmatrix} \f^N_{n} \\ c_{n}\end{pmatrix} + A_{2,n}\begin{pmatrix} \f^N_{n-1} \\ c_{n-1}\end{pmatrix},
\end{align}
where the matrices $A_{1,n},A_{2,n} \in \R^{(N+1)\times(N+1)}$ are given explicitly by
\begin{align}
A_{1,n} & := \begin{pmatrix} -2\ell_nd_n^1 I_N & 2\ell_nd_n^2 B_{N,1} \\ \frac{2\ell_n d_n^1}{n+1}B_{1,N} & -\frac{(2N-1)\ell_n d_n^2}{n+1}I_1 \end{pmatrix}, & 
A_{2,n} & := -\frac{\ell_n}{\ell_{n-1}}\begin{pmatrix} I_N & 0\\ 0& \frac{n-1}{n+1}I_1 \end{pmatrix},
\end{align}
for all $n\ge 2$. Since $\ell_{n-1} \le \ell_{n}$ and
\be\label{eq:est_d}
d_n^{1}<d_n^{2} = \frac{n+N-1 }{\ell_{n-1}} + \frac{n +N+1}{\ell_{n }} \le \frac{2(n+N)}{\ell_n}
\ee
for all $n\ge2$, it is not difficult to get the following rough bounds 
\footnote{Here and below to estimate norms, we use the equality  $\|A\| = \sqrt{\|A^\ast A\|}$ and the following simple estimate for non-negative $2\times 2$ block-matrices $A = \begin{pmatrix} A_{11} & A_{12} \\ A_{12}^\ast & A_{22}\end{pmatrix}$:\ 
$\|A\| \le \|A_{11}\| + \|A_{22}\|$.
There are other estimates (e.g.,~\cite[ineq. (2.3.8)]{GL12}), however, they do not seem to work as good as the above approach.}
%%%%%%%%%%%for calculations use
%\[
%\|A_{1,N}\|^2 \le \|A_{1,N}^\ast A_{1,N}\| 
%\]
%%%%%%%%%and then take max of its diagonal
\begin{align}\label{eq:A12est}
\|A_{1,n}\| & \le 6\sqrt{N}(n+N), &  \|A_{2,n}\| &  = \frac{\ell_n}{\ell_{n-1}}\le 1,
\end{align}
for all $n \ge 2N$. 
Denoting 
\[
F_n := \begin{pmatrix} \f_n^N \\ c_n \end{pmatrix},\quad n\ge 1,
\]
the recurrence relations \eqref{eq:sys3}--\eqref{eq:sys4} can be written in the following matrix form
\begin{align}
\begin{pmatrix} F_{n+1} \\ F_n \end{pmatrix} = \begin{pmatrix} A_{1,n} & A_{2,n} \\ I_{N+1} & 0_{N+1} \end{pmatrix} \begin{pmatrix} F_{n} \\ F_{n-1} \end{pmatrix} 
= A_n \begin{pmatrix} F_{n} \\ F_{n-1} \end{pmatrix}.
\end{align}
Taking into account \eqref{eq:A12est}, we get $\|A_n\| \le 6\sqrt{N}(n+N+1)$ for all $n\ge 2N$, which implies the estimate
\be
\sqrt{\|\f_n^N\|^2 + |c_n|^2} = \|F_n\| \le C \prod_{k=1}^{n-1} \|A_k\| \lesssim (6\sqrt{N})^n (n+N)! 
\ee
for all $n\ge 2$. Combining this bound with \eqref{eq:ellNdecay}, it is easy to see that the series on the right hand side in \eqref{eq:est7} converges and hence by Lemma~\ref{lem:Harmcont=Harmdiscr} we conclude that $\HH(\cA_N)\subset L^2(\cA)$. Thus $\ker(\bH) = \HH(\cA_N)$ and the use of  Corollary~\ref{cor:finvol=ker} finishes the proof.
\end{proof}

\subsection{Infinite deficiency indices}
Consider the antitree $\cA$ with sphere numbers $s_n := n+1$, $n\ge 0$. Next pick a sequence of positive numbers $(\ell_n)$ and define lengths as follows
\begin{align}\label{eq:lengthInf}
|e_{ij}^n| = \begin{cases} 2\ell_n, & 1\le i=j \le n+1,\\ \ell_n, & \text{otherwise}, \end{cases}
\end{align}
for all $n\in\Z_{\ge0}$. 
Thus, the corresponding matrix $M_{n+1}$ given by \eqref{def:M} has the form 
\be\label{eq:Minf}
M_{n+1} = \frac{1}{\ell_n}\begin{pmatrix} B_{n+1} - \frac{1}{2}I_{n+1} & B_{n+1,1} \end{pmatrix} \in \R^{(n+1)\times(n+2)}
\ee
for all $n\ge 0$. Let us denote this antitree by $\cA_\infty$.

\begin{lemma}\label{lem:HATinf}
$ \dim(\HH (\cA_\infty)) = \infty$.
\end{lemma}

\begin{proof}
Consider the difference equations \eqref{eq:harmrec1}--\eqref{eq:harmrec2}. Clearly, the matrix $M_{n+1}$ has the maximal rank $n+1$ for every $n\ge 0$. Taking into account that 
\[
\Big(B_{n+1} - \frac{1}{2}I_{n+1} \Big)^{-1} = %-\frac{4}{2N-1}A_N - 2I_N =  
\frac{4}{2n+1}B_{n+1} - 2 I_{n+1} = : C_n, \quad n\ge0,
\]
\eqref{eq:harmrec2} then reads
\begin{align}\label{eq:harmrec2inf}
\begin{pmatrix} I_{n+1} & \frac{2}{2n+1}B_{n+1,1} \end{pmatrix}\f_{n+1} & = \ell_nC_n(D_n\f_n -  M_n^\ast \f_{n-1})
\end{align}
for all $n\ge 1$. Observe that 
\[
\begin{pmatrix} I_{n+1} & \frac{2}{2n+1}B_{n+1,1} \end{pmatrix}\begin{pmatrix} f_1\\ \vdots\\ f_{n+1}\\ 0 \end{pmatrix} = \begin{pmatrix} f_1\\ \vdots\\ f_{n+1} \end{pmatrix}
\]
and hence for any $\f_n\in \C^{n+1}$ and $\f_{n-1}\in\C^n$ there always exists a unique $\f_{n+1} = (f_1,\dots,f_{n+1},0)^\top$ satisfying \eqref{eq:harmrec2inf}.
Now pick a natural number $N$ and define $\f^N\in C(\cA_\infty)$ by setting $\f_n^N := (0,\dots,0)^\top \in \C^{n+1}$ for all $n\in \{0,\dots,N\}$, 
\[
\f_{N+1}^N := (1,\dots,1,-N-1/2)^\top,
\]
and
\begin{align}\label{eq:finfrecur}
\f_{n+1}^N := \begin{pmatrix} \ell_nC_n(D_n\f_n^N -  M_n^\ast \f_{n-1}^N) \\ 0 \end{pmatrix} \in \C^{n+2}
\end{align}
for all $n\ge N+1$. Clearly, $\f^N$ satisfies  \eqref{eq:harmrec1}--\eqref{eq:harmrec2} and hence defines a harmonic function $f^N\in \HH(\cA_\infty)$. Moreover, it is easy to see that $\Span\{\f^N\}_{N\ge 1}$ is infinite dimensional, which proves the claim.
\end{proof}

 \begin{proposition}\label{prop:ANindicesInf}
 Let $\bH_0$ be the minimal Kirchhoff Laplacian associated with the antitree $\cA_\infty$. If $\ell_n$ is decreasing and
 \be\label{eq:elldecayInf}
 \sqrt{\ell_n} = \OO\left(\frac{1}{6^n(n+3)!}\right)
 \ee
 as $n\to \infty$, then $\Nr_\pm(\bH_0) = \infty$.
 \end{proposition}
 
\begin{proof}
Clearly, it suffices to show that every $f^N$ constructed in the proof of Lemma~\ref{lem:HATinf} belongs to $L^2(\cG)$ if $\ell_n$ decays as in \eqref{eq:elldecayInf}. To prove this we shall proceed as in the proof of Proposition~\ref{prop:ANindices}. First, taking into account \eqref{eq:lengthInf}, observe that 
\[
m(v) \lesssim n\ell_{n-1},\quad v\in S_n,
\]
as $n\to \infty$. 
Since $\|\f^N_{n}\|^2 = \sum_{v\in S_n} |f^N(v)|^2$ 
for all $n\ge 0$, we get the estimate
\begin{align}\label{eq:est10}
\sum_{v\in\cV} |f^N(v)|^2m(v) \lesssim \sum_{n\ge N+1}\sum_{v\in S_n} |f^N(v)|^2m(v)
%\lesssim \sum_{n\ge 1} n\ell_{n-1}(\|\f_n^N\|^2 + n|c_n|^2)
\lesssim \sum_{n\ge N+1} n\ell_{n-1}\|\f^N_{n}\|^2. 
%\|\ell_nC_n(D_n\f_n -  M_n^\ast \f_{n-1})\|^2.\label{eq:est10}
\end{align}
Denoting $F_{n} := \f^N_{n}$ for all $n\ge 1$, we can put \eqref{eq:harmrec2inf} into the matrix form
\be
\begin{pmatrix} F_{n+1} \\ F_n \end{pmatrix} = \begin{pmatrix}  A_{1,n} & A_{2,n} \\ I_{n+1} & 0_{n+1,n} \end{pmatrix} \begin{pmatrix} F_{n} \\ F_{n-1} \end{pmatrix}= A_n \begin{pmatrix} F_{n} \\ F_{n-1} \end{pmatrix}
\ee
for all $n\ge N+1$, where
\begin{align}
A_{1,n} & := \begin{pmatrix} \ell_n C_nD_n \\ 0_{1,n+1} \end{pmatrix}\in \R^{(n+2)\times(n+1)}, & A_{2,n} & := \begin{pmatrix} -\ell_n C_nM_n^\ast \\ 0_{1,n} \end{pmatrix}\in \R^{(n+2)\times n}.
\end{align}
Now observe that $\|C_n\| = 2$ and $\|\ell_nD_n\| \le 2(n+1)$  for all $n\ge 1$. 
Moreover, $\|\ell_n M_n^\ast\|\le n+1$ for all $n\ge 1$, which immediately implies the following estimate
\be
\|A_n\| \le \sqrt{\|\ell_n C_nD_n\|^2 + 1 + \|\ell_n C_nM_n^\ast\|^2} \le 6(n+1),\quad n\ge N+1.
\ee
Hence we get
\[
\|\f^N_{n+1}\| \le C \prod_{k=N+1}^n \|A_k\| \le C6^{n-N} \frac{(n+1)!}{(N+1)!}\lesssim 6^n(n+1)!
\]
for all $n\ge N+1$. Combining this estimate with \eqref{eq:est10} and \eqref{eq:elldecayInf} and using Lemma~\ref{lem:Harmcont=Harmdiscr}, we conclude that $f^N\in L^2(\cA_\infty)$ for each $N\ge 1$. 
\end{proof}

\begin{remark}
It is not difficult to show that $f^N$ does not belong to $H^1(\cA_\infty)$ for the above choices of edge lengths. In fact, it follows from the maximum principle for $\HH(\cA)$ that if $\vol(\cA) < \infty$, then $\HH(\cA) \cap H^1(\cA)$ consists only of constant functions.
\end{remark}

\subsection{Proof of Theorem~\ref{th:DefInd}} 
Clearly, the case of infinite deficiency indices follows from Proposition~\ref{prop:ANindicesInf}. On the other hand, since adding and/or removing finitely many edges and vertices to a graph does not change the deficiency indices of the minimal Kirchhoff Laplacian, Proposition~\ref{prop:ANindices} completes the proof of Theorem~\ref{th:DefInd}. Indeed, every antitree $\cA_N$ can be obtained from $\cA$ by first removing all the edges between combinatorial spheres $S_0$ and $S_{N}$ and then adding $N+1$ edges connecting the root $o$ with the vertices in $S_N$. \qed 

\begin{remark}\label{rem:ATends}
Since every infinite antitree has exactly one end, Theorem~\ref{th:ThetaCD}(iv) implies that the Kirchhoff Laplacian $\bH_0$ in Theorem~\ref{th:DefInd} has a unique Markovian extension exactly when $\vol(\cA)=\infty$. If $\vol(\cA)<\infty$, then Markovian extensions of $\bH_0$ form a one-parameter family explicitly given by \eqref{eq:domext-2}. Notice that \eqref{eq:domext-2} looks similar to the description of self-adjoint extensions of the minimal Kirchhoff Laplacian on radially symmetric antitrees obtained recently in~\cite{kn19}.

Let us also emphasize that the antitree constructed in Proposition~\ref{prop:ANindicesInf} has finite total volume and $\bH_0$ has infinite deficiency indices, however, the set of  Markovian extensions of $\bH_0$ forms a one-parameter family.
\end{remark}

Let us finish this section with one more comment. As it was proved, the dimension of the space of Markovian extensions depends only on the space of graph ends and, moreover, it is equal to the number of finite volume ends. However, deficiency indices (dimension of the space of self-adjoint extensions) are in general independent of graph ends and we can only provide a lower bound. Moreover, the above example of a polynomially growing antitree shows that the space of non-constant harmonic functions heavily depends on the choice of edge lengths (in particular, its dimension may vary between zero and infinity). In this respect let us also emphasize that in the case of Cayley graphs of finitely generated groups the end space is independent of the choice of a generating set, however, simple examples show that the space of harmonic functions does depend on this choice. 

%%%%%%%%%%%%%%%%%%%%%%%
%%%%%%%%%%%%%%%%%%%%%%%
\appendix
%%%%%%%%%%%%%%%%%%%%%%%
%%%%%%%%%%%%%%%%%%%%%%%

%%%%%%%%%%%%%%%%%%%%%%%%%%%%%%%%%%%%%%%%%%%%%%%%%%%%%%%%%%%%%%%
%%%%%%%%%%%%%%%%%%%%%%%%%%%%%%%%%%%%%%%%%%%%%%%%%%%%%%%%%%%%%%%
\section{Linear relations in Hilbert spaces}\label{app:LR}
%%%%%%%%%%%%%%%%%%%%%%%%%%%%%%%%%%%%%%%%%%%%%%%%%%%%%%%%%%%%%%%
%%%%%%%%%%%%%%%%%%%%%%%%%%%%%%%%%%%%%%%%%%%%%%%%%%%%%%%%%%%%%%%

In this section we collect basic notions and facts on linear relations in Hilbert spaces, a very convenient concept of multi-valued linear operators. For simplicity, we shall assume that \emph{$\cH$ is a finite dimensional Hilbert space},  $N:= \dim(\cH) <\infty$. 

A \emph{linear relation} $\Theta$ in $\cH$ is a linear subspace  in $\cH\times \cH$. Linear operators become special linear relations (single valued) after identifying them with their graphs in $\cH\times \cH$. Consider linear relations in $\cH$ having the form
\be\label{eq:ThetaCD}
\Theta_{C,D} = \{(f,g)\in \cH\times \cH\,|\, Cf = Dg \},
\ee
where $C,D$ are linear operators on $\cH$. Notice that different $C$ and $D$ may define the same linear relation. The \emph{domain} and the \emph{multi-valued part} of $\Theta_{C,D}$ are given by
\begin{align*}
\dom(\Theta_{C,D}) &= \{f\in \cH\,|\, \exists g\in\cH, Cf=Dg \} = \{f\in \cH \,|\, Cf \in \ran(D)\},\\
\mul(\Theta_{C,D}) & = \{g\in \cH \,|\, Dg = 0\} = \ker(D).
\end{align*}
In particular, $\Theta_{C,D}$ is a graph of a linear operator only if $\ker(D) = \{0\}$. 

The adjoint relation $\Theta_{C,D}^\ast$ to $\Theta_{C,D}$ is given by 
\begin{align}
\Theta_{C,D}^\ast & = \{(f,g)\in \cH\times\cH \,| \, \langle \wt{g},f\rangle_{\cH} = \langle\wt{f},g \rangle_{\cH}\ \forall (\wt{f},\wt{g})\in\Theta_{C,D}\} \nn\\
&= \big\{ (D^\ast f, C^\ast f)\,|\, f\in \cH \big\}.\label{eq:ThetaCD*}
\end{align}
Thus, a linear relation $\Theta_{C,D}$ is self-adjoint, $\Theta_{C,D}  = \Theta_{C,D}^\ast $, if and only if $C$ and $D$ satisfy the \emph{Rofe-Beketov conditions}~\cite{RB} (see also~\cite[Exercises 14.9.3-4]{schm}):
\begin{align}\label{eq:appRB}
CD^\ast & = DC^\ast, & 0 & \in \rho(C^\ast C + D^\ast D).
\end{align}
Taking into account that every linear relation in $\cH$ admits one of the forms \eqref{eq:ThetaCD} or \eqref{eq:ThetaCD*}, this provides a description of self-adjoint linear relations in $\cH$. Notice also that the second condition in \eqref{eq:appRB} is equivalent to the fact that the matrix $(C| D)\in \C^{N\times 2N}$ has the maximal rank $N$. 

Recall also that every self-adjoint linear relation admits the representation $\Theta = \Theta_{\rm op}\oplus\Theta_{\mul}$, where $\Theta_{\mul} := \{0\}\times \mul(\Theta)$ and $\Theta_{\rm op}$, called the operator part of $\Theta$, is a graph of a linear operator. In particular, for a self-adjoint linear relation $\Theta_{C,D}$ one has
\be\label{eq:domCD}
\dom(\Theta_{C,D}) = \mul(\Theta_{C,D})^\perp = \ker(D)^\perp = \ran(D^\ast).
\ee
For further details on linear relations we refer the reader to, e.g.,~\cite[Chapter~14.1]{schm}. 

%%%%%%%%%%%%%%%%%%%%%%%%%%%%%%%%%%%%%%%%%%%%%%%%%%%%%%%%%%%%%%%%%
%%%%%%%%%%%%%%%%%%%%%%%%%%%%%%%%%%%%%%%%%%%%%%%%%%%%%%%%%%%%%%%%%
\section{A rope ladder graph}\label{sec:rope}
%%%%%%%%%%%%%%%%%%%%%%%%%%%%%%%%%%%%%%%%%%%%%%%%%%%%%%%%%%%%%%%%%
%%%%%%%%%%%%%%%%%%%%%%%%%%%%%%%%%%%%%%%%%%%%%%%%%%%%%%%%%%%%%%%%%
Let us introduce a rope ladder graph depicted on Figure~\ref{fig:ropeladder}. Let $\cG_d=(\cV, \cE)$ be a simple graph with the vertex set $\cV := \{o\} \cup \cV^+ \cup \cV^-$, where $o = v_0$ is a root, $\cV^+ = (v_n^+)_{n\ge 1}$ and $\cV^- = (v_n^-)_{n\ge 1}$ are two disjoint countably infinite sets of vertices. The edge set $\cE$ is defined as follows:
\begin{itemize}
\item  $o$ is connected to  $v^+_1$ and  $v^-_1$ by the ``diagonal" edges $e^+_0$ and $e^-_0$, respectively;
\item for each $n \ge 1$, $v_n^\pm$ is connected to $v_{n+1}^\pm$ by the vertical edge $e_n^\pm$;
\item for each $n \ge 1$, $v^+_n$ and $v^-_n$ are connected by the horizontal edge $e_n$. 
\end{itemize}
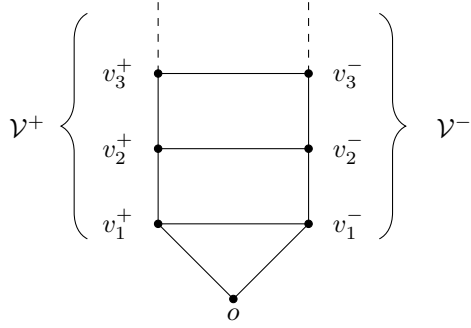
\begin{figure}[ht]
	\begin{center}
		\begin{tikzpicture}

		%VERTICES
		% ORIGIN
		\filldraw (0,0) circle (1.4 pt) node [below] {$o$}  ;

		%OTHER VERTICES
		\foreach \x in {-1,1} {
	
		%EDGE TO ORIGIN
		\draw		(0,0) -- (\x,1);
		
		%OTHER VERTICES
		
		\foreach \y in {1,...,3}{
			\filldraw 
			(\x,\y) circle (1.4pt);   }
			}

		%HORIZONTAL EDGES
		\foreach \y in {1,...,3}{
		\draw (-1, \y) -- (1, \y);}
			
		%VERTICAL EDGES
		\foreach \x in {-1, 1}{ 
		
		\draw  [thin, -, dashed] (\x, 3) -- (\x, 4); 

		\foreach \y in {1,...,2}{
			\draw (\x, \y) -- (\x, \y+1);}
		}

		%LABELLING

		\foreach \y in {1,...,3}{
			\filldraw 
			(-1.2,\y) node [left] {$v_\y^+$}  ;
			
		\filldraw 
			(1.2,\y) node [right] {$v_\y^-$}  ;
			}

   	\draw [decorate,decoration={brace,amplitude=10pt},xshift=-4pt,yshift=0pt]
	(-1.8, 0.8) -- (-1.8, 3.8) node  [black,midway, xshift = - 0.8 cm] {$\cV^+$};
	  \draw [decorate,decoration={brace,amplitude=10pt},xshift=4pt,yshift=0pt]
	(1.8, 3.8) -- (1.8, 0.8)   node  [black,midway, xshift =  1 cm] {$\cV^-$};
   		%\filldraw 
			%(-2,2.5) node [left] {$\cV^1$}  ;
			
		%\filldraw 
			%(2,2.5) node [right] {$\cV^2$}  ;
		
		\end{tikzpicture}
	\end{center}
	\caption{The rope ladder graph. }\label{fig:ropeladder}
\end{figure}
By construction, $\deg(o) = 2$ and $\deg(v_n^+) =\deg(v_n^-) = 3$ for all $n\ge 1$. Moreover, an infinite rope ladder graph has exactly one end. Notice also that a similar example was studied in~\cite[Section 7]{jp} (see also~\cite[\S~5]{g10}) in context with the construction of non-constant harmonic functions of finite energy.

Equip now $\cG_d$ with edge lengths $|\cdot|\colon \cE\to (0,\infty)$ and consider the corresponding minimal Kirchhoff Laplacian $\bH_0$ on the metric graph $\cG = (\cG_d, |\cdot|)$. 
The next result immediately follows from Theorem~\ref{th:ekmn} and Corollary~\ref{cor:finvol=ker}.

\begin{corollary}\label{cor:ladderSA}
If 
\begin{align}
\sum_{n\ge 1} |e_n^+| + |e_n|  = \infty, \quad \text{and} \quad \sum_{n\ge 1} |e_n^-| + |e_n|  = \infty,
\end{align}
then the Kirchhoff Laplacian $\bH_0$ is self-adjoint.
If 
\be
\vol(\cG) = \sum_{n\ge 1} |e_n^+| + |e_n^-|+ |e_n| < \infty,
\ee
then $\Nr_\pm(\bH_0)\ge 1$.
\end{corollary}

We omit the proof since it is easy to check that the first condition is equivalent to the geodesic completeness of $(\cV,\varrho_m)$ (cf. Theorem~\ref{th:ekmn}). Due to the symmetry of the underlying combinatorial graph, the gap between the above two conditions is equivalent to the fact that the corresponding lengths satisfy
\begin{align}\label{eq:4.7}
\sum_{n\ge 1}|e_n^+| & =\infty, & \sum_{n\ge 1} |e_n^-| + |e_n| < \infty.
\end{align}

Next, let us describe the space of harmonic functions $\HH(\cG)$.

\begin{lemma}\label{lem:rlbasis}
	Let $a, b \in \C$. Then there is exactly one $f \in \HH(\cG)$ such that
	\begin{align}\label{eq:f=ab}
		f(v^+_1) & = a, & f(v^-_1) & = b.
	\end{align}
	Moreover, this function $f$ is recursively given by
	\begin{align} \label{eq:rl01}
		f(o) = \frac{b |e^+_0|  +  a|e^-_0| }{  |e^+_0|+  |e^-_0|}
	\end{align}
	and
	\begin{align}\label{eq:rl02}
		f(v^\pm_{n+1}) = \left( 1 + \frac{ |e^\pm_{n}|}{ | e^\pm_{n-1} |} +  \frac{ |e^\pm_{n}|}{ | e_{n} |}  \right) f(v^\pm_{n}) - \frac{ |e^\pm_{n}|}{ | e^\pm_{n-1} |} f(v^\pm_{n-1}) -  \frac{ |e^\pm_{n}|}{ | e_{n} |} f(v^\mp_{n}),
	\end{align}
	for all $n \in \Z_{\ge 1}$, where we use the notation $v^+_0 := v^-_0 := o$. 
\end{lemma}

\begin{proof}
	Suppose $a, b \in \C$ are given and  $f \in \HH (\cG)$ satisfies \eqref{eq:f=ab}. Since $f$ is linear on every edge and satisfies \eqref{eq:kirchhoff} at $v = o$, we get
	\[
		0 = f_{e^+_0}'(o) + f_{e^-_0}'(o)=\frac{ f(v^+_1) - f(o)}{ |e^+_0| } + \frac{ f(v^-_1) - f(o)}{ |e^-_0| } = \frac{ a - f(o)}{ |e^+_0|} + \frac{ b - f(o)}{ |e^-_0| },
	\]
which implies \eqref{eq:rl01}. Moreover, Kirchhoff conditions \eqref{eq:kirchhoff} at $v = v^\pm_n$, $n \ge 1$ read
	\begin{align*}
		\frac{  f(v^\pm_{n+1}) -    f(v^\pm_{n})     }{| e^\pm_n| } + \frac{  f(v^\pm_{n-1}) -    f(v^\pm_{n})     }{| e^\pm_{n-1}| }  + \frac{ f(v^\mp_{n}) -    f(v^\pm_{n})     }{| e_{n}| } = 0.
	\end{align*}
	This implies that $f$ is given by \eqref{eq:rl02}. Hence there is at most one $f \in \HH (\cG)$ satisfying \eqref{eq:f=ab} for given $a,b\in\C$. However, the same calculation shows that $f$ defined by \eqref{eq:rl01} and \eqref{eq:rl02} has this property. Thus, existence follows as well.
\end{proof}

From Lemma~\ref{lem:rlbasis},   it is clear that $\dim (\HH(\cG)) = 2$, and, moreover, 
\[
	\HH (\cG) = \operatorname{span} \{ \id_\cG, g_0 \},
\]
where $ \id_\cG $ denotes the constant function on $\cG$ and $g_0 \in \HH (\cG)$ is the function defined, for example, by the following normalization
\begin{align} \label{eq:defg}
	&g_0(0) = 0, &g_0(v^+_1) = |e^+_0|, &&g_0(v^-_1) = -| e^-_0|.
\end{align}
Notice that  $g_0(v^\pm_n)$, $n \ge 1$ are then given recursively by \eqref{eq:rl02}. 

\begin{lemma}\label{lem:rlg}
If $\vol(\cG) < \infty$, then
	\be\label{eq:g0notinH}
	\HH(\cG)\cap H^1(\cG) = \Span\{\id_\cG\}.
	\ee
\end{lemma}

The claim immediately follows from the fact that a rope ladder graph has exactly one end. However, let us present a direct proof based on the analysis of harmonic functions.

\begin{proof} 
Taking into account \eqref{eq:g0notinH}, we only need to show that $g_0\notin H^1(\cG)$. First, observe that $( g_0(v^+_n))_{n\ge 1}$ and $( g_0(v^-_n))_{n\ge 1}$ are strictly increasing positive, respectively, strictly decreasing negative sequences. Indeed, 
	\[
	-|e^-_0| = g_0(v^-_1) < 0 = g_0(o) < g_0(v^+_1)=|e^+_0|
	\] 
	by the very definition of $g_0$. Let $n \ge 1$ and assume now that we have already shown that $(g_0(v^+_k))_{k=1}^n$ is strictly increasing and $(g_0(v^-_k))_{k=1}^n$ is strictly decreasing. Since $g_0(o) = 0$, \eqref{eq:rl02} implies
	\begin{align*}
			g_0(v^+_{n+1}) &= \Big( 1 + \frac{ |e^+_{n}|}{ | e^+_{n-1} |} +  \frac{ |e^+_{n}|}{ | e_{n} |}  \Big) g_0(v^+_{n}) - \frac{ |e^+_{n}|}{ | e^+_{n-1} |} g_0(v^+_{n-1}) -  \frac{ |e^+_{n}|}{ | e_{n} |} g_0(v^-_{n}) \\
			&>  \left( 1 + \frac{ |e^+_{n}|}{ | e_{n} |} \right) g_0(v^+_{n}) + \frac{ |e^+_{n}|}{ | e^+_{n-1} |} ( g_0(v^+_{n}) - g_0(v^+_{n-1})) > g_0(v^+_{n}).
	\end{align*}
	A similar argument shows that $g_0(v^-_{n+1})< g_0(v^-_{n})$ and hence the claim follows by induction. 
Now monotonicity immediately implies
	\begin{align*}
		 \| g_{0}' \|^2_{L^2(\cG)}  &= \sum_{e \in \cE} \int_e |g_0'(x_e)|^2 \; dx_e \geq \sum_{n\ge 0} \int_{e_n} |g_{0}' (x_e)|^2 \; dx_e \\
		 &= \sum_{n=0}^\infty   \frac{|g_{0} (v^+_n) - g_{0} (v^-_n)|^2}{|e_n|}   
		 \geq  |g_{0} (v^+_1) - g_{0} (v^-_1)|^2 \sum_{n=0}^\infty   \frac{1}{|e_n|} = \infty,
	\end{align*}
	since $\vol(\cG)< \infty$. Thus $g_0\notin H^1(\cG)$.
\end{proof}

In particular, this also leads to the following result:

\begin{corollary} \label{lem:defrl}
	If $\vol(\cG) < \infty$, then $\Nr_\pm (\bH_0) \in \{1,2\}$. Moreover,  $ \Nr_\pm (\bH_0) = 1$ if and only if $g_0 \notin L^2(\cG)$.
\end{corollary}

\begin{proof}
	The claim about the deficiency indices follows from Corollary~\ref{cor:finvol=ker} and the fact that $\id_{\cG} \in L^2(\cG)$. The equivalences then follow from Lemma~\ref{lem:rlg}.
\end{proof}

As the next example shows, the inclusion $g_0\in L^2(\cG)$ heavily depends on the choice of edge lengths. 

\begin{example} \label{ex:polynomialrl} 
	Fix  $s > 3$ and equip the rope ladder graph with edge lengths
	\[
		|e^+_n| = | e^-_n| := \frac{1}{ (n+1)^s}, \qquad  |e_n | := \frac{2n}{ (n+1)^s - n^s},\quad n\in\Z_{\ge 0}.
	\]
	Then $|e_n |  \sim n^{2-s}$ for large $n$ and hence $\vol(\cG) < \infty$. Moreover, for this particular choice of edge lengths we have $g_{0} (v^\pm_n) = \pm n$ for all $n \ge 1$. Indeed, $g_0(v^\pm_1) = \pm 1$ by \eqref{eq:defg}. Assuming we have already proven that $g_0(v^\pm_k) = \pm k$ for $k \le n$ with some $n \ge 1$, we have by \eqref{eq:rl02}:
	\begin{align*}
			g_0(v^+_{n+1})  &= \Big( 1 + \frac{n^s}{(n+1)^s} + \frac{1}{(n+1)^s | e_{n} |} \Big ) n - \frac{n^s(n-1)}{(n+1)^s}   + \frac{n}{(n+1)^s | e_{n} |}  \\
					 		&= n + \frac{n^s}{(n+1)^s} + \frac{2 n}{(n+1)^s |e_{n} |} 
							= n + \frac{n^s}{(n+1)^s} + \frac{(n+1)^s - n^s}{(n+1)^s} = n+1.
	\end{align*}
	Analogously, $g_0(v^-_{n+1}) = -(n+1)$ and hence the claim follows by induction. 
	
	Applying Lemma~\ref{lem:rlg} and using again that $|e_n |  \sim n^{2-s}$ as $n\to \infty$, we conclude that $g_0 \in L^2(\cG)$ exactly (see Lemma~\ref{lem:Harmcont=Harmdiscr}) when
	\[
 \sum_{n\ge 1} |g_0(v_n^\pm)|^2 (|e_{n-1}^\pm| + |e_{n}^\pm|)= \sum_{n\ge 1} n^2 ( (n+1)^{-s} + n^{-s}) < \infty 
	\]
	and
	\[
 \sum_{n\ge 1} |g_0(v_n^\pm)|^2 |e_{n-1}| = \sum_{n\ge 1} \frac{2n^3}{ (n+1)^s - n^s} < \infty. 
	\]
	Clearly, the latter holds only if $s>5$. Hence, by Lemma~\ref{lem:defrl}, $\Nr_\pm (\bH_0) = 2$ for all $s>5$. In particular, 
	 $\ker(\bH) \subset H^1(\cG) \Leftrightarrow s \leq 5$. 
%	\hfill $\lozenge$
\end{example}
%%%%%%%%%%%%%%%%%%%%%%%%%%%%%%%%%%%%%%%%%%%%%%%%%%%%%%%%%%%%%%%
%%%%%%%%%%%%%%%%%%%%%%%%%%%%%%%%%%%%%%%%%%%%%%%%%%%%%%%%%%%%%%%%%%%%%%%%%%%%%%%%%%%%%%%%%%%%%%%%%%%%%%%%%%%%%%%%%%%%%%%%%%%%%%

\noindent
\ack We thank Matthias Keller, Daniel Lenz, Primo\v{z} Moravec, Andrea Posilicano and Wolfgang Woess for useful discussions and hints with respect to the literature. We also thank the referees for their comments which have helped us to improve the manuscript. 
N.N. appreciates the hospitality at the Institute of Mathematics, University of Potsdam, during a research stay funded by the OeAD (Marietta Blau-grant, ICM-2019-13386), where a part of this work was done.


\begin{thebibliography}{XXX}
\bibitem{ag}
N.\ I.\ Akhiezer and I.\ M.\ Glazman, \emph{Theory of Linear Operators in Hilbert Spaces. Vol. II}, Dover Publ., New York, 1993.

\bibitem{amca13}
O.\ Amini and L.\ Caporaso, \emph{Riemann--Roch theory for weighted graphs and tropical curves}, Adv.\ Math.\ {\bf 240}, 1--23 (2013).

\bibitem{AreEls19}
W.\ Arendt and A.\ F.\ M.\ ter Elst, \emph{Operators with continuous kernels}, Integr.\ Equ.\ Oper.\ Theory {\bf 91}, no.~5, Art.~45 (2019).

\bibitem{Arebuh94}
W.\ Arendt and A.\ V.\ Bukhvalov, \emph{Integral representations of resolvents and semigroups}, Forum.\ Math.\ {\bf 6}, 111--135 (1994).

\bibitem{bano07}
M.\ Baker and S.\ Norine, \emph{Riemann--Roch and Abel--Jacobi theory on a finite graph}, Adv.\ Math.\ {\bf 215}, no.~2, 766--788 (2007).

\bibitem{baru10}
M.\ Baker and R.\ Rumely, \emph{Potential Theory and Dynamics on the Berkovich Projective Line}, Math.\ Surv.\ Monographs, {\bf 159}. Amer.\ Math.\ Soc., Providence, RI, 2010.

\bibitem{baba03}
M.\ T.\ Barlow and R.\ F.\ Bass, \emph{Stability of parabolic Harnack inequalities}, Trans.\ Amer.\ Math.\ Soc.\ {\bf 356}, no.~4, 1501--1533 (2004).

\bibitem{bcfk06}
G.\ Berkolaiko, R.\ Carlson, S.\ Fulling, and P.\ Kuchment, 
\emph{Quantum Graphs and Their Applications},   Contemp.\ Math.\ {\bf 415}, Amer.\ Math.\ Soc., Providence, RI, 2006. 

\bibitem{bk13}
G.\ Berkolaiko and P.\ Kuchment, \emph{Introduction to Quantum Graphs}, 
Amer.\ Math.\ Soc., Providence, RI, 2013.

\bibitem{brke13}
J.\ Breuer and M.\ Keller, \emph{Spectral analysis of certain spherically homogeneous graphs}, 
Oper.\ Matrices {\bf 7}, no.~4, 825--847 (2013).

\bibitem{bl19}
J.\ Breuer and N.\ Levi, \emph{On the decomposition of the Laplacian on metric graphs},
Ann.\ Henri Poincar\'e {\bf 22}, no.~2, 499--537 (2020). 

\bibitem{bre}
H.\ Brezis, {\em Functional Analysis, Sobolev Spaces and Partial Differential Equations}, Universitext, Springer, New York, 2011.

\bibitem{bbi}
D.\ Burago, Yu.\ Burago, and S.\ Ivanov, {\em  A Course in Metric Geometry}, Graduate Stud.\ Math.\ {\bf 33}, 
Amer.\ Math.\ Soc., Providence, RI, 2001.

\bibitem{CarMug09}
S.\ Cardanobile and D.\ Mugnolo, {\em Parabolic systems with coupled boundary conditions}, 
J.\ Differential Equations {\bf 247}, 1229--1248 (2009).

\bibitem{car00}
R.\ Carlson, \emph{Nonclassical Sturm--Liouville problems and Schr\"odinger operators on radial trees}, 
Electr.\ J.\ Differential Equations {\bf 2000}, no.~71, pp. 1--24 (2000).

\bibitem{car08}
R.\ Carlson, \emph{Boundary value problems for infinite metric graphs}, in~\cite{ekkst08}.

\bibitem{car12}
R.\ Carlson, \emph{Dirichlet to Neumann maps for infinite quantum graphs}, Netw.\ Heterog.\ Media {\bf 7}, no.~3, 483--501 (2012).

\bibitem{csw}
D.\ I.\ Cartwright, P.\ M.\ Soardi, and W.\ Woess, \emph{Martin and end compactifications for nonlocally finite graphs}, 
Trans.\ Amer.\ Math.\ Soc.\ {\bf 338}, no~2, 679--693 (1993).

\bibitem{cdvtht}
Y.\ Colin de Verdi\`ere, N.\ Torki-Hamza, and F.\ Truc, 
\emph{Essential self-adjointness for combinatorial Schr\"{o}dinger operators II--metrically non complete graphs}, 
Math.\ Phys.\ Anal.\ Geom.\ {\bf 14}, no.~1, 21--38 (2011).

\bibitem{Cou03}
T.\ Coulhon, \emph{Off-diagonal heat kernel lower bounds without {P}oincar{\'e}},
J.\ London Math.\ Soc.\ (2) {\bf 68}, 795--816 (2003).

\bibitem{clmp}
D.\ Cushing, S.\ Liu, F.\ M\"unch, and N.\ Peyerimhoff, \emph{Curvature calculations for antitrees}, 
in: M.\ Keller et. al. (Eds.), ``Analysis and Geometry on Graphs and Manifolds", 
London Math.\ Soc.\ Lect. Notes Ser.\ {\bf 461}, Cambridge Univ.\ Press, 2020.

\bibitem{dfs}
D.\ Damanik, L.\ Fang, and S.\ Sukhtaiev, \emph{Zero measure and singular continuous spectra for quantum graphs}, 
Ann.\ Henri Poincar\'e {\bf 21}, no.~7, 2167--2191 (2020).

\bibitem{die}
R.\ Diestel, \emph{Graph Theory}, 5th edn., Grad.\ Texts in Math. 
{\bf 173}, Springer-Verlag, Heidelberg, New York, 2017.

\bibitem{dikue}
R.\ Diestel and D. K\"uhn, \emph{Graph-theoretical versus topological ends of graphs}, 
J.\ Combin.\ Theory, Ser. B  {\bf 87}, 197--206 (2003).

\bibitem{dk}
J.\ Dodziuk and L.\ Karp, \emph{ Spectral and function theory for combinatorial Laplacians}, 
in: ``Geometry of Random Motion" (Ithaca, N.Y., 1987), Contemp.\ Math.\ {\bf 73}, Amer.\ Math.\ Soc., Providence, pp. 25--40 (1988).

\bibitem{ekkst08}
P.\ Exner, J.\ P.\ Keating, P.\ Kuchment, T.\ Sunada, and A.\ Teplyaev, \emph{Analysis on Graphs and Its Applications}, 
Proc.\ Symp.\ Pure Math.\ {\bf 77}, Providence, RI, Amer.\ Math.\ Soc., 2008.

\bibitem{ekmn}
P.\ Exner, A.\ Kostenko, M.\ Malamud, and H.\ Neidhardt, \emph{Spectral theory of infinite quantum graphs}, 
Ann.\ Henri Poincar\'e {\bf 19}, no.~11, 3457--3510 (2018).

\bibitem{fo14}
M.\ Folz, \emph{Volume growth and stochastic completeness of graphs}, 
Trans.\ Amer.\ Math.\ Soc.\ {\bf 366}, 2089--2119 (2014).

\bibitem{f31}
H.\ Freudenthal, \emph{\"Uber die Enden topologischer R\"aume und Gruppen}, 
Math.\ Z.\ {\bf 33}, 692--713 (1931).

\bibitem{f44}
H.\ Freudenthal, \emph{\"Uber die Enden diskreter R\"aume und Gruppen}, 
Comment.\ Math.\ Helv.\ {\bf 17}, 1--38 (1944).

\bibitem{fuk10}
M.\ Fukushima, Y.\ Oshima, and M.\ Takeda, 
\emph{Dirichlet Forms and Symmetric Markov Processes}, 2nd edn., De Gruyter, 2010.

\bibitem{geog}
R.\ Geoghegan, \emph{Topological Methods in Group Theory}, 
Grad.\ Texts in Math.\ {\bf 243}, Springer, 2008.

\bibitem{g10}
A.\ Georgakopoulos, \emph{Uniqueness of electrical currents in a network of finite total resistance}, 
J.\ London Math.\ Soc.\ {\bf 82}, 256--272 (2010).

\bibitem{g11}
A.\ Georgakopoulos, \emph{Graph topologies induced by edge lengths}, 
 Discrete Math.\ {\bf 311}, no.~15, 1523--1542 (2011).

\bibitem{ghklw}
A.\ Georgakopoulos, S.\ Haeseler, M.\ Keller, D.\ Lenz, and R.\ Wojciechowski, \emph{Graphs of finite measure}, J.\ Math.\ Pures Appl.\ {\bf 103}, S1093--S1131 (2015).

\bibitem{GL12}
G.\  Golub and C.\ F.\ Van Loan, \emph{Matrix Computations}, 4th edn., The Johns Hopkins University Press,  Baltimore, 2012.

\bibitem{gm}
A.\ Grigor'yan and J.\ Masamune, \emph{Parabolicity and stochastic completeness of manifolds in terms of the Green formula}, 
J.\ Math.\ Pures Appl.\ {\bf 100}, 607--632 (2013).

\bibitem{hal}
R.\ Halin, \emph{\"Uber unendliche Wege in Graphen}, Math.\ Ann.\ {\bf 157}, 125--137 (1964).

\bibitem{hklw}
S.\ Haeseler, M.\ Keller, D.\ Lenz, and R.\ Wojciechowski, \emph{Laplacians on infinite graphs: Dirichlet and Neumann boundary conditions}, 
J. Spectr. Theory {\bf 2}, no.~4, 397--432 (2012).

\bibitem{hae}
S.\ Haeseler, \emph{Analysis of Dirichlet forms on graphs}, PhD thesis, Jena, 2014; \arxiv{1705.06322} (2017).

\bibitem{hs}
M.\ Hinz and M.\ Schwarz, {\em A note on Neumann problems on graphs}, preprint, \arxiv{1803.08559} (2018).

\bibitem{hop}
H.\ Hopf, \emph{Enden offener R\"aume und unendliche diskontinuierliche Gruppen}, 
Comment.\ Math.\ Helv.\ {\bf 16}, 81--100 (1943).

\bibitem{hu19}
B.\ Hua, \emph{Liouville theorem for bounded harmonic functions on manifolds and graphs satisfying non-negative curvature dimension condition}, Calc.\ Var.\ Partial Differential Equations\ {\bf 58}, no.~2, Art.~42 (2019).

\bibitem{huke14}
B.\ Hua and M.\ Keller, \emph{Harmonic functions of general graph Laplacians}, 
Calc.\ Var.\ Partial Differential Equations\ {\bf 51}, 343--362 (2014).

\bibitem{hkmw}
X.\ Huang, M.\ Keller, J.\ Masamune, and R.\ Wojciechowski, \emph{A note on self-adjoint extensions of the Laplacian on weighted graphs}, 
J.\ Funct.\ Anal.\ {\bf 265}, 1556--1578 (2013).

\bibitem{jp}
P.\ Jorgensen and E.\ Pearse, \emph{Gel'fand triples and boundaries of infinite networks}, 
New York J.\ Math.\ {\bf 17}, 745--781 (2011).

\bibitem{KanKlaVoi09}
U.\ Kant, T.\ Klau{\ss}, J.\ Voigt, and M.\ Weber, {\em Dirichlet forms for singular one-dimensional operators and on graphs}, 
J.\ Evol.\ Equ.\ {\bf 9}, 637--659 (2009).

\bibitem{kasue10}
A.\ Kasue, \emph{Convergence of metric graphs and energy forms}, 
Rev.\ Mat.\ Iberoam.\, {\bf 26}, no.~2, 367--448 (2010). 

\bibitem{kasue17}
A.\ Kasue, {\em Convergence of Dirichlet forms induced on boundaries of transient networks}, 
Potential Anal.\ {\bf 47}, no.~2, 189--233 (2017).

\bibitem{kato}
T.\ Kato, \emph{Perturbation Theory for Linear Operators}, 2nd ed., Springer-Verlag, Berlin-New York, 1976.

\bibitem{kl12}
M.\ Keller and  D.\ Lenz, \emph{Dirichlet forms and stochastic completeness of graphs and subgraphs}, 
J.\ reine angew.\ Math.\ {\bf 666}, 189--223 (2012).

\bibitem{klss}
M.\ Keller, D.\ Lenz, M.\ Schmidt, and M.\ Schwarz, \emph{Boundary representation of Dirichlet forms on discrete spaces}, 
J.\ Math.\ Pures Appl.\ {\bf 126}, 109--143 (2019).

\bibitem{klmw}
M.\ Keller, D.\ Lenz, M.\ Schmidt, and R.\ K.\ Wojciechowski, \emph{Note on uniformly transient graphs}, 
Rev.\ Mat.\ Iberoam. {\bf 33}, no.~3, 831--860 (2017).

\bibitem{klw}
M.\ Keller, D.\ Lenz, and R.\ K.\ Wojciechowski, \emph{Volume growth, spectrum and stochastic completeness of infinite graphs}, 
Math.\ Z.\ {\bf 274}, no.~3-4, 905--932 (2013).

\bibitem{kn17}
A.\ Kostenko and N.\ Nicolussi, \emph{Spectral estimates for infinite quantum graphs}, 
Calc.\ Var.\ Partial Differential Equations 
{\bf 58}, no.~1, Art.~15 (2019). 

\bibitem{kn19}
A.\ Kostenko and N.\ Nicolussi, \emph{Quantum graphs on radially symmetric antitrees}, 
J.\ Spectral Theory {\bf 11}, no.~2, 411--460 (2021).

\bibitem{KosPotSch08}
V.\ Kostrykin, J.\ Potthoff, and R.\ Schrader, {\em Contraction semigroups on metric graphs}, in~\cite{ekkst08}.

\bibitem{lu16}
T.\ Lupu, \emph{From loop clusters and random interlacements to the free field}, 
Ann.\ Probab.\ {\bf 44}, 2117--2146 (2016).

\bibitem{lu19}
T.\ Lupu, \emph{Convergence of the two-dimensional random walk loop-soup clusters to CLE}, 
J.\ Eur.\ Math.\ Soc.\ {\bf 21}, 1201--1227 (2019).

\bibitem{mar83}
J.\ T.\ Marti, \emph{Evaluation of the least constant in Sobolev's inequality for $H^1(0,s)$}, 
SIAM J.\ Numer.\ Anal.\ {\bf 20}, 1239--1242 (1983).

\bibitem{ma99}
J.\ Masamune, \emph{Essential self-adjointness of Laplacians on Riemannian manifolds with fractal boundary}, 
Commun.\ Partial Diff.\ Equ.\ {\bf 24}, no.~3-4, 749--757  (1999).

\bibitem{ma05}
J.\ Masamune, \emph{Analysis of the Laplacian of an incomplete manifold with almost polar boundary}, 
Rend.\ Mat.\ Appl.\ (7) {\bf 25}, no.~1, 109--126 (2005).

\bibitem{MugNit11}
D.\ Mugnolo and R.\ Nittka, 
\emph{Properties of representations of operators acting between spaces of vector-valued functions},
Positivity {\bf 15}, 135--154 (2011).

\bibitem{my}
A.\ Murakami and M.\ Yamasaki, {\em An introduction of Kuramochi boundary of an infinite network}, 
Mem.\ Fac.\ Sci.\ Eng.\ Shimane Univ.\ Ser.\ B Math.\ Sci.\ {\bf 30}, 57--89, 1997.

\bibitem{noe20}
N.\ Nicolussi, {\em Strong isoperimetric inequality for tessellating quantum graphs}, 
in: F.\ Fatihcan et.\ al. (Eds.), ``Discrete and Continuous Models in the Theory of Networks", 271--290, 
Oper.\ Theory: Adv.\ Appl.\ {\bf 64}, Birkh\"auser, Basel, 2020.

\bibitem{ouh}
E.\ M.\ Ouhabaz, \emph{Analysis of Heat Equations on Domains}, Princeton Univ.\ Press, Princeton and Oxford, 2005.

\bibitem{post}
O.\ Post, \emph{Spectral Analysis on Graph-Like Spaces}, Lect.\ Notes in Math.\ {\bf 2039}, Springer-Verlag, Berlin, Heidelberg, 2012

\bibitem{RB}
F.\ S.\ Rofe-Beketov, \emph{Self-adjoint extensions of differential operators in a space of vector-
valued functions}, Teor.\ Funkcii, Funkcional.\ Anal.\ i Prilo\v{z}en.\ {\bf 8}, 3--24 (1969). (in Russian)

\bibitem{schm}
K.\ Schm\"udgen, \emph{Unbounded Self-Adjoint Operators on Hilbert Space}, 
Grad.\ Texts in Math.\ {\bf 265}, Springer, Berlin, 2012.

\bibitem{shwu19}
F.\ Shokrieh and C.\ Wu, \emph{Canonical measures on metric graphs and a Kazhdan's theorem}, 
Invent.\ Math.\ {\bf 215}, 819--862 (2019).

\bibitem{soardi}
P.\ M.\ Soardi, \emph{Potential Theory on Infinite Networks}, Lect.\ Notes Math.\ {\bf 1590}, Springer, Berlin, 1994.

\bibitem{sol04}
M.\ Solomyak, \emph{On the spectrum of the Laplacian on regular metric trees}, 
Waves Random Media {\bf 14}, S155--S171 (2004).

\bibitem{s71}
J.\ R.\ Stallings, \emph{Group Theory and Three-Dimensional Manifolds}, 
Yale Univ.\ Press, New Haven, Connecticut, 1971.

\bibitem{stu94}
K.-T.\ Sturm, \emph{Analysis on local Dirichlet spaces I. Recurrence, conservativeness and $L^p$-Liouville properties}, 
J.\ reine angew.\ Math.\ {\bf 456}, 173--196 (1994).

\bibitem{wall}
C.\ T.\ C.\ Wall, \emph{Poincar\'e complexes: I}, Ann.\ of Math.\ {\bf 86}, no.~2, 213--245 (1967).

\bibitem{woe}
W.\ Woess, \emph{Random Walks on Infinite Graphs and Groups}, Cambridge Univ.\ Press, Cambridge, 2000.

\bibitem{woj11}
R.\ K.\ Wojciechowski, \emph{Stochastically incomplete manifolds and graphs}, in: D.\ Lenz et al. (Eds.), ``Random Walks, Boundaries and Spectra", 163--179, Progr.\ Probab.\ {\bf 64}, Birkh\"auser/Springer Basel AG, Basel, 2011.

\end{thebibliography}
\end{document}